\documentclass[sn-mathphys-ay]{sn-jnl}% Math and Physical Sciences Author Year Reference Style 
\usepackage{graphicx}%
\usepackage{multirow}%
\usepackage{amsmath,amssymb,amsfonts}%
\usepackage{amsthm}%
\usepackage{mathrsfs}%
\usepackage[title]{appendix}%
\usepackage{xcolor}%
\usepackage{textcomp}%
\usepackage{manyfoot}%
\usepackage{booktabs}%
\usepackage{algorithm}%
\usepackage{algorithmicx}%
\usepackage{algpseudocode}%
\usepackage{listings}%
\newtheorem{theorem}{Theorem}%  meant for continuous numbers
\newtheorem{proposition}[theorem]{Proposition}% 
\newtheorem{lemma}[theorem]{Lemma}% 
\newtheorem{corollary}[theorem]{Corollary}% 
\newtheorem{remark}[theorem]{Remark}%

\numberwithin{theorem}{section} % Theorem numbering is [section number]

\ifpdf
  \DeclareGraphicsExtensions{.eps,.pdf,.png,.jpg}
\else
  \DeclareGraphicsExtensions{.eps}
\fi

\usepackage[utf8]{inputenc}
\usepackage{fullpage}
\usepackage{amsmath, amssymb, amsfonts}
\usepackage{thmtools, thm-restate}
\usepackage{mathabx}
\DeclareFontFamily{U}{mathc}{}
\DeclareFontShape{U}{mathc}{m}{it}%
{<->s*[1.03] mathc10}{}
\DeclareMathAlphabet{\mathscr}{U}{mathc}{m}{it}
\usepackage{mathtools}
\usepackage{graphicx}
\usepackage{subcaption}
\usepackage{dsfont}
\usepackage{color}
\usepackage{empheq}
\usepackage[most]{tcolorbox}
\newtcbox{\mymath}[1][]{%
    nobeforeafter, math upper, tcbox raise base,
    enhanced, colframe=blue!30!black,
    colback=blue!30, boxrule=1pt,
    #1}
\usepackage{hyperref}
\hypersetup{
 pdfauthor={Wolfgang Dahmen and Olga Mula},
 pdftitle={Kinetic},
 pdfsubject={},
 pdfkeywords={}
}

\numberwithin{equation}{section}

\newcommand{\cB}{\ensuremath{\mathcal{B}}}

\newcommand{\cE}{\ensuremath{\mathcal{E}}}
\newcommand{\cF}{\ensuremath{\mathcal{F}}}

\newcommand{\cK}{\ensuremath{\mathcal{K}}}
\newcommand{\cL}{\ensuremath{\mathcal{L}}}

\newcommand{\cN}{\ensuremath{\mathcal{N}}}

\newcommand{\cR}{\ensuremath{\mathcal{R}}}

\newcommand{\cT}{\ensuremath{\mathcal{T}}}

\newcommand{\bC}{\ensuremath{\mathbb{C}}}
\newcommand{\bD}{\ensuremath{\mathbb{D}}}

\newcommand{\bH}{\ensuremath{\mathbb{H}}}

\newcommand{\bJ}{\ensuremath{\mathbb{J}}}

\newcommand{\bN}{\ensuremath{\mathbb{N}}}

\newcommand{\bR}{\ensuremath{\mathbb{R}}}
\newcommand{\bS}{\ensuremath{\mathbb{S}}}

\newcommand{\ett}{\tau}

\DeclareMathOperator{\Bop}{\mathcal{B}}
\DeclareMathOperator{\Cop}{{\mathcal{C}}}
\DeclareMathOperator{\Fop}{\mathcal{F}}
\DeclareMathOperator{\Kop}{\mathcal{K}}
\DeclareMathOperator{\Iop}{\mathcal{I}}
\DeclareMathOperator{\Mop}{\mathcal{M}}
\DeclareMathOperator{\Pop}{\mathcal{P}}
\DeclareMathOperator{\Zop}{\mathcal{Z}}
\DeclareMathOperator{\Eop}{\mathcal{E}}
\DeclareMathOperator{\Rop}{\mathcal{R}}
\DeclareMathOperator{\Aop}{\mathcal{A}}
\DeclareMathOperator{\Top}{\mathcal{T}}

\newcommand{\Vsp}{\mathbb{V}}

\newcommand{\oDelta}{\overline{\Delta}}
\newcommand{\uDelta}{{\Delta}}
\newcommand{\rc}{{r_\circ}}
\newcommand{\La}{\Lambda}
\newcommand{\rat}{{\mathfrak{q}}}
\newcommand{\worr}[1]{{\color{black}{#1}}}

\DeclareMathOperator{\Ddom}{\mathrm{D}}
\DeclareMathOperator{\Vdom}{{\mathrm{V}}}
\DeclareMathOperator{\Edom}{{\mathrm{E}}}

\DeclareMathOperator{\Usp}{\mathbb{U}}
\DeclareMathOperator{\Uspext}{\tilde{\mathbb{U}}}
\DeclareMathOperator{\Wsp}{\mathbb{W}}
\DeclareMathOperator{\Hsp}{\mathbb{H}}
\newcommand{\Us}{{\Usp^{(\sigma)}}}

\newcommand{\be}{\begin{equation}}
\newcommand{\ee}{\end{equation}}
\newcommand{\<}{\langle}
\renewcommand{\>}{\rangle}
\newcommand{\e}{\varepsilon}
\newcommand{\dist}{\operatorname{dist}}
\DeclareMathOperator*{\vspan}{span}
 \newcommand{\tripnorm}[1]{|\!|\!| #1|\!|\!|}
\newcommand{\s}{\ensuremath{{s}}}
\renewcommand{\v}{\ensuremath{{v}}}
\newcommand{\ug}{\ensuremath{{u}}}
\renewcommand{\lg}{\ensuremath{\lambda}}
\renewcommand{\u}{\ensuremath{u^\circ}}
\renewcommand{\l}{\ensuremath{{\lambda^\circ}}}
\DeclareMathOperator{\m}{\ensuremath{\mu^\circ}}
\newcommand{\n}{\ensuremath{n}}
\newcommand{\id}{\ensuremath{\Iop}}

\newcommand{\zp}{\ensuremath{z_\circ^\perp}}

\newcommand{\dv}{\ensuremath{\mathrm d\v}}
\newcommand{\dt}{\ensuremath{\mathrm dt}}
\newcommand{\dx}{\ensuremath{\mathrm dx}}

\newcommand{\ks}{\ensuremath{\kappa}}

\newcommand{\U}{\Usp}                      % Trial Space

\newcommand{\R}{\mathbb{R}}
\newcommand{\N}{\mathbb{N}}

\def\argmin{\mathop{\rm argmin}}

\newcommand{\om}[1]{{\color{black}{#1}}} % color: red
\newcommand{\corr}[1]{{\color{black}{#1}}}
\newcommand{\corrw}[1]{{\color{black}{#1}}}
\newcommand{\wo}[1]{{\color{black}{#1}}} % color: blue
\newcommand{\wor}[1]{{\color{black}{#1}}} % color: red
\DeclareMathOperator*{\argmax}{arg\,max}

\begin{document}

% % Sets running headers as well as PDF title and authors
% \headers{Eigenvalue Problem in Radiative Transfer}{W.~Dahmen, O.~Mula}

% % Title. If the supplement option is on, then "Supplementary Material"
% % is automatically inserted before the title.
% \title{Accuracy Controlled Schemes for the Eigenvalue Problem of the
% Radiative Transfer Equation\thanks{Submitted to the editors DATE.
% \funding{This research was supported by
% the  NSF Grants DMS 2038080, DMS-2012469, DMS-2245097, by the SmartState and Williams-Hedberg Foundation, by the  SFB 1481,
% funded by the German Research Foundation. and by the Emergences Project Grant ``Models and Measures'' from the Paris City Council.}}}

% % Authors: full names plus addresses.
% \author{Wolfgang Dahmen\thanks{Department of Mathematics, University of South Carolina, United States.}
% \and Olga Mula\thanks{Department of Mathematics and Computer Science, TU Eindhoven, the Netherlands.}}

\title[Eigenvalue Problem in Radiative Transfer]{Accuracy Controlled Schemes for the Eigenvalue Problem of the Radiative Transfer Equation}


\author[1]{\fnm{Wolfgang} \sur{Dahmen}}%\email{wolfgang.anton.dahmen@googlemail.com}

\author[2]{\fnm{Olga} \sur{Mula}}%\email{o.mula@tue.nl}

\affil*[1]{\orgdiv{Department of Mathematics}, \orgname{University of South Carolina}, \orgaddress{\country{United States}}}% Columbia, SC 29208 

\affil[2]{\orgdiv{Department of Mathematics and Computer Science}, \orgname{Eindhoven University of Technology}, \orgaddress{\country{the Netherlands}}}


\abstract{The {\em criticality problem} in nuclear engineering asks for the principal 
eigenpair of a Boltzmann operator describing neutron transport in a reactor core. Being able to reliably design, and control such reactors requires
assessing these quantities within quantifiable accuracy tolerances. In this paper we propose a paradigm that deviates from the common practice  of approximately solving the corresponding spectral problem with a fixed, presumably sufficiently fine  discretization.
Instead, the present approach is based on first contriving iterative schemes, formulated 
in function space, that are shown to converge at a quantitative rate without assuming any a priori excess regularity properties, and that exploit only properties of the optical parameters in the underlying radiative transfer model.
We develop the analytical and numerical tools for approximately realizing each iteration step within judiciously chosen accuracy tolerances, verified by a posteriori estimates, so as to still
warrant quantifiable convergence to the exact eigenpair. This is carried out in full first for a Newton scheme. Since this is  only locally 
convergent  we analyze in addition the convergence of a power iteration in function space to produce
sufficiently accurate initial guesses. Here we have to deal
 with intrinsic difficulties posed by compact but 
unsymmetric operators preventing standard arguments used in the finite dimensional case. Our main point  is   that we can avoid any condition on an initial guess to be already in a small neighborhood of the exact solution. We close with a discussion of remaining intrinsic obstructions to a certifiable numerical implementation, mainly related to not knowing the gap between the  {principal} eigenvalue and the next smaller one in modulus.}

\keywords{eigenvalue problem; spectral problem; neutron transport; radiative transfer; error-controlled computation; a posteriori estimation; iteration in function space}


\pacs[MSC Classification]{68Q25, 68R10, 68U05}

\editor{Endre Süli}

\maketitle

%%%%%%%%%%%%%%%%%%%%%%%
\section{Introduction}
%%%%%%%%%%%%%%%%%%%%%%%%
\subsection{Context, problem formulation, and main objectives}

Radiative transfer plays a key role in a number of scientific and engineering areas. It is, for example, relevant in understanding certain atmospheric processes and is also an essential constituent in modeling neutron transport in the context of nuclear
engineering. The mathematical problem studied in this paper arises from this latter field. When operating a nuclear reactor, engineers seek to create a sustainable chain reaction in which the neutrons that are produced balance the neutrons that are either absorbed or leave the boundary. When the balance is achieved, the neutron population finds itself in a steady state which does not evolve in time. The specific nature of the balance and the neutron distribution crucially depend on the material and geometry of the reactor, hence the importance of studying this problem at design stages in order to find appropriate core configurations. In the field of nuclear engineering, this task is known as the \emph{criticality} problem, and it is modelled through the computation of a \emph{generalized unsymmetric eigenvalue problem} of a \emph{linear Boltzmann} operator. Depending on the community, the process is also referred to as \textit{radiative transfer} or \textit{neutron transport}.

\subsubsection{A model problem}
\label{ssec:eigenprob}
To be specific, we consider the following model  problem that nevertheless exhibits all essential obstructions:
For a given bounded spatial domain $\Ddom\subset \bR^d$ (the nuclear reactor), we consider the \textit{population density of neutrons} $u=u(x, v)$ located at position $x\in \Ddom$ and travelling with velocity $v\in \Vdom \subset \bR^d$. The balance between the neutrons that are produced and the ones that are lost (by leakage or absorption) is mathematically expressed as the following generalized unsymmetric eigenvalue problem (usually called the criticality problem): the pair $(\ug,\lg)$ with $\ug\neq0$, and $\lg \in \bC $ is said to be an eigenpair of the radiative transfer equation if and only if
\begin{small}
\begin{equation}
\begin{alignedat}{2}
\v\cdot\nabla \ug(x,\v) + \sigma(x,\v) \ug(x,\v)
- \int_{\Vdom} \kappa (x,\v',\v) \ug(x,\v')\,\dv' &= \lg \int_{\Vdom} \varphi (x,\v',\v)\ug(x,\v')\,\dv',\qquad&& \forall (x,\v) \in \Ddom\times \Vdom, \\
\ug &= 0, \qquad &&\text{on $\Gamma_-$}.
\end{alignedat}
\label{Eig}
\end{equation}
\end{small}

This model is also called the radiative transfer equation when applied to uncharged particles other than neutrons (e.g.~light and radiation). In the above formula and in the following, $\nabla$ should be understood as the gradient with respect to the spatial coordinate $x$. Also, $\Gamma_-$ and $\Gamma_+$ are the incoming/outgoing phase-space boundaries defined as
\begin{equation}
  \begin{cases}
\Gamma_- &\coloneqq \{ (x,\v) \in \partial \Ddom \times \Vdom \mid \v\cdot \n(x) < 0 \} \subset \partial\Ddom \times \Vdom \\
\Gamma_+ &\coloneqq \{ (x,\v) \in \partial \Ddom \times \Vdom \mid \v\cdot \n(x) > 0 \} \subset \partial\Ddom \times \Vdom,
  \end{cases}
  \label{inflow}
\end{equation}
where $\n(x)$ denotes the unit outward normal at a point $x\in \partial\Ddom$. The functions $\sigma,\ \kappa$ and $\varphi$ are nonnegative and they are the so-called \emph{optical parameters} related to physical nuclear reactions. $\sigma$ is the total cross-section, $\kappa$ is the scattering cross-section and $\varphi$ is associated to the fission cross-section.

It is well-known that under quite general assumptions (which we will recall, and discuss later on), it is possible to apply a Krein-Rutman theorem and show that the eigenvalue $\lg$ with smallest modulus is simple, real, and positive, and associated to a nonnegative and real eigenfunction $\ug$. In the following, we refer to this pair as the \emph{principal eigenpair}, and we denote it by $(\u, \l)$. The value of $\l$ has a direct physical meaning since it describes the balance between the amount of fission effects on the one hand, and the amount of transport leakage and scattering effects on the other hand. The reactor is called supercritical if \wo{$\l<1$} (the fission chain reaction escalates and needs to be slowed down), subcritical if \wo{$\l>1$} (the chain reaction decreases and the energy production will eventually stop) and critical if $\l=1$ (balanced chain reaction, with sustained energy production). Designing a reactor in which $\l$ is as close to 1 as possible is a key inverse problem in nuclear engineering. To reliably address it, it is necessary to be able to solve forward problems with {\em rigorous, quantifiable accuracy}. \wo{Developing principal strategies towards that end}   is the focus of this paper.

The \wo{proposed approach} is iterative and   involves, as a key constituent, the solution of intermediate \emph{source} problems of the following form: for a given right-hand side function $q$, find $u$ solution to
\begin{equation}
  % \tag{SP}
  \label{SP}
\begin{alignedat}{2}
\v\cdot\nabla u(x,\v) + \sigma(x,\v) u(x,\v)
- \int_{\Vdom} \kappa (x,\v',\v)u(x,\v')\,\dv' &= q(x,\v),\qquad&& \forall (x,\v) \in \Ddom\times \Vdom,
\end{alignedat}
\end{equation}
endowed with some inflow boundary condition.

\subsubsection{Central goal}
Of course, the prediction capability of   numerical simulations based on such a model is challenged by a number of uncertainty sources, among them unavoidable model biases. However, once the model has been accepted as a basis for simulations, it is crucial to be able to guarantee the quality of the numerical outcome.
The central objective of this paper is to contrive  a \wo{{\em conceptual pathway},   along with
the relevant analytical tools, towards rigorously quantifying} the deviation of the numerical result from the \emph{exact, infinite-dimensional} solution of the eigenvalue problem in a relevant norm. Note that this goes beyond asking for {\em a priori} error estimates, which have been the subject of extensive research in the past, and whose validity depends typically on extra regularity assumptions which are generally hard to check in practical applications. Instead, our approach hinges, in particular, on  rigorous {\em computable a posteriori} quantities that hold without any excess regularity assumptions but just exploit suitable {\em stable variational formulations} of transport equations.

\subsection{Prior related work and challenges} 

The radiative transfer operator has been extensively studied from a diversity of point of views.  From the perspective of nuclear reactor physics, we refer to \cite{LM1984, Reuss2008, Stacey2018} for a general overview. At the mathematical level, relevant results concerning the well-posedness and properties of the equation can be found in \cite{Mika1972, AB1999, MK2005, DLvol6}. There has also been an intense activity in numerical analysis and scientific computing in order to derive numerical solvers, and study their convergence properties. We refer to, e.g., \cite{Kanschat2009,ACP2011,GS2011b,RGK2012} for important  results in this regard and to, e.g., \cite{WR2009, JDLCB2011, BLMM2014minaret, Mula2014, BJRF2018} for further related relevant developments.

Until very recently, all these works follow the classical paradigm:
\vspace*{1mm}

\begin{center}
 {\bf SA:} {\it first discretize the continuous problem, then solve a fixed discrete problem.}
\end{center}
\vspace*{1.5mm}
We refer to this as the ``standard approach'' (SA). While certainly this has its merits for many application scenarios, an accuracy controlled solution of eigenproblems for radiative transfer and Boltzmann type models poses particular challenges, revolving around 
the following issues: (a) the size of the discrete problem; (b) stability and solver accuracy;
(c) discretization error uncertainty; (d) stability of the spectrum. 

A few comments on these issues are in order. Regarding (a),  the presence of two groups of variables (spatial and velocity variables) renders such problems {\em high-dimensional}. Therefore,  when targeting relevant accuracy levels, related   discretizations 
tend to become very large. This is aggravated by the fact that, lacking rigorous local
error control, common approaches   employ in essence uniform or quasi-uniform meshes
that need to be fine enough to resolve relevant scales. In fact, such concerns are
confirmed by the accuracy controlled
source problem solver from \cite{DGM2020} which is based on totally different concepts
that will be taken up again later below. Matching the accuracy levels, achieved
in \cite{DGM2020} with the aid of adaptively refined grids, just using quasi-uniform meshes would have been infeasible.  
In addition, the appearance of {\em global integral operators} representing scattering and fission effects result in {\em densely populated} matrices.
Hence the applicability of efficient sparse linear algebra techniques is limited.

As for (b),
the lack of symmetry and of preconditioners of similar efficiency as for elliptic problems
further impede on the efficient accurate solution of such very large scale problems
unless strong model simplifications are accepted. Moreover, such discretizations are usually challenged by stability issues due to potentially strong transport effects and a lack of sufficient dissipation. Looking therefore for discretizations that have certain conservation properties still falls short of relating discrete solutions to the 
\wor{continuous} one.

Concerning (c),
when employing iterative solvers the question then is how accurately such large scale problems need to be solved and how to verify this accuracy because machine accuracy may be too demanding. \wo{It actually suffices if
the accuracy of discrete solutions matches}   {\em discretization error accuracy}. 
Unfortunately,  the actual discretization error
remains largely uncertain. In fact, resorting to {\em a priori error estimates} as
an orientation is not really a remedy, mainly for two reasons: first, they are typically derived under unrealistic regularity assumptions whose validity is hard to check in practice. Second, such error bounds involve norms of derivatives of the unknown solution. 
\wo{In summary, an a priori quantification of the discretization error accuracy is 
in general not realistic.}

\wo{This underlines the importance of deriving rigorous {\em a posteriori error bounds}
and to understand how to use them to effectively reduce the error. In stark contrast to elliptic problems, sharp rigorous (lower and upper) a posteriori error bounds for kinetic models are far less developed. In particular, existing ones do not come with any comparable concise refinement strategy and corresponding convergence analysis. Early natural a posteriori bounds 
are based on stability in $L_2$, \cite{Madsen}. Underlying variational formulations for
radiative transfer can be found in \cite{Agoshkov}. Mixed formulations based on parity principles are studied in \cite{Egger}. An a posteriori error analysis for the graph norm based on duality principles is given
in \cite{Han15,Hanbook} while in \cite{Mant} a posteriori bounds are derived that are based on certain scaled graph norms. These latter works consider either constant   kernels $\kappa(x,\v,v')$ or kernels of 
the form $\kappa(x,\v\cdot\v')$. 
To the best of our knowledge these works do {\em not} provide a {\em convergence analysis} 
of concrete adaptive refinement methods based on such error bounds. While this has even been a long-standing open problem for elliptic problems, the non-locality of radiative
transfer models seems to aggravate this issue. To the best of our knowledge, a first rigorous adaptive convergence theory for radiative transfer models
has been developed  in  \cite{DGM2020}. The critical new ingredient there is
the combination of a posteriori error bounds, based on combining well-posedness in the sense of the Babu$\check{s}$ka--Ne$\check{c}$as Theorem, with a quantifiable convergence of an 
idealized fixed point iteration in function space, see below for further comments in
this regard.}   
 
Finally, as for (d), relying on rough estimates may be delicate because the continuity of spectra for such
unsymmetric problems is less forgiving than for symmetric problems.

In summary, we claim that therefore a reliable accuracy quantification for problems
of the present type is hard
to realize on the basis of the standard approach {\bf (SA)}.
In fact, to the best of our knowledge, \cite{SG2011} is the only previous deviation 
from  {\bf (SA)} for the criticality problem in the following sense.
There, the authors study the convergence of an inexact inverse power iteration, viewed
as a perturbation of such an iteration in function space. As explained in more detail in the subsequent section, this is close in spirit to the present work.
However, the authors in \cite{SG2011} work under certain strong assumptions on the optical parameters: for the monoenergetic model\footnote{In the monoenergetic model, the particles are assumed to have all the same kinetic energy so, for the analysis, one works with $\vert v \vert=1$, thus $v\in \bS^{d-1}$, the unit sphere of $\bR^d$.}
with uniformly constant scattering and fission functions $\kappa$ and $\varphi$, the authors show that, on the continuous level, the general unsymmetric eigenproblem \eqref{Eig} is equivalent to a certain related eigenproblem for the scalar flux $\phi(x)=\int_V u(x, v)\dv$, involving a symmetric positive definite weakly singular integral operator (in space only). This correspondence to a symmetric problem (in a space of reduced dimension) allows them to give a convergence theory for an inexact inverse iteration. These \wor{techniques} cannot be extended for nonconstant functions $\kappa$ and $\varphi$. Moreover, concisely quantifying and certifying the necessary closeness
of perturbed iterates to exact ones seems to be missing.

\subsection{Conceptual constituents and novelty}\label{ssec:novelty}
In the light of the previous comments the approach proposed in the present paper deviates
fundamentally from the standard approach {\bf(SA)}. A primary conceptual novelty 
for the present context lies in {\em reversing the steps in {\bf (SA)}}.
In a nutshell, instead of discretizing right away the Boltzmann operator, we first contrive what we call an {\em idealized} or {\em outer} iteration in a suitable function space which needs to provably converge at a certain 
rate in the function space. The numerical method consists then of {\em approximately} carrying out
each step in the idealized iteration within a step-dependent, dynamically updated
accuracy tolerance. So, part of our task is to identify
such tolerances that necessarily decrease in the course of the outer iteration (and which are expected to be rather coarse at initial stages). 
Hence, at no stage of the numerical solution process  deals with a single
(potentially large) discretization: each iteration step produces a discretization needed to stay sufficiently close to the idealized iterate.

The approach is reminiscent of the classical 
{\em nested refinement} idea, carried out, however, in an infinite-dimensional context.
 To obtain this way  in the end results with certified
accuracy, it is crucial that the target tolerance in each  approximate outer iteration step is 
met. This underlines the role of {\em a posteriori} error bounds that need to be
provided for this purpose. Such bounds should hold {\em without any excess regularity}
assumptions and therefore need to rely on estimating {\em infinite-dimensional residuals}
in the right norm. While such techniques have a long history in the context of {\em elliptic problems} this is not at all the case for transport dominated problems. Our approach to deriving  
  such bounds   is essentially  based   on {\em stable ultraweak variational formulations} for the involved transport equations. To our knowledge, this is a crucial
distinction from all other computational approaches to the present eigenproblem.
To aid further commenting on these ingredients we summarize this paradigm  formally as follows:

(CC1) Contrive an {\em idealized} iteration in the relevant Hilbert space that can be shown to converge at a quantifiable rate.

(CC2) Determine perturbation tolerances within which the idealized iteration steps can be carried out so as to still warrant convergence to the exact continuous solution.

(CC3) Determine numerical tools that allow one to realize the tolerances in (CC2), certified by appropriate {\em a posteriori error bounds}.

\paragraph{Comments on (CC1) - (CC3)}
What is behind   (CC1) - (CC3) and their relation to existing work?
Regarding (CC1), we treat two such idealized iterations, namely {\em Newton's
method}, and {\em power iteration} for different respective purposes. The obvious tempting
aspect of Newton iterations is the potentially quadratic convergence, albeit generally
under the condition that the initial guess is sufficiently close to the exact zero.
Studying Newton's method in function space is certainly not a new idea, see e.g. \cite{anselone1968solution}. Likewise, there is a wealth of \wor{literature} addressing inexact Newton methods in Euclidean spaces, see e.g.~\cite{Tisseur2001,dongarra1983improving} as a few examples. Unfortunately, 
the findings we have been able to trace back do not apply to the issues
arising in the specific context of (CC1) - (CC3). Here, it is essential to quantify suitable initial neighborhoods in conjunction with concrete bounds on resolvents when 
restricted to certain subspaces to end up with concrete bounds for the 
mapping properties of Jacobians as we get closer to the exact principal eigenpair. 
In particular, we are not perturbing the operator but need to quantify the 
quality of {\em approximate} Newton updates depending on the state of the idealized iteration. Corresponding tools, presented   in Section \ref{sec:Newton-scheme},
are based on a refined analysis
of the mapping properties of the resolvent. These preparations then allow us
to address (CC2), namely to determine suitable perturbation bounds whose validity
would still warrant convergence to the exact principal eigenpair at a quantifiable rate.
Although this approach is closer in spirit to existing work on {\em adaptive} eigensolvers for 
{\em elliptic problems} (see e.g. \cite{DHS2007, GG2009, CDMSV2018}) the techniques 
in these works do not carry over to the present setting, among other things, due to the lack of symmetry in the present setting. 

We further emphasize an essential distinction from working either in a {\em fixed} finite dimensional
or in a {\em fixed} infinite-dimensional context which is typically the case in 
``non-adaptive'' works. In our case the discretizations vary with evolving accuracy demands. 
While we show how to preserve even quadratic convergence 
for our perturbed idealized iteration, the numerical work will increase significantly
because the tolerance will be seen to decrease also quadratically.
Therefore, we show how properly chosen, less stringent tolerances lead to first order
convergence where we can however {\em choose} the concrete error reduction rate. Which variant would be preferable in practice depends again on ``\wor{critical problem parameters}'' -- in the end on the spectral gap, namely the distance of the principal eigenvalue from
the rest of the spectrum.

Finally, (CC3) brings in a central building block in our approach.
An essential precondition is the fact that ideal iterations and their approximations
are based on the same stable variational formulation. Specifically, we heavily exploit
that, for
the {\em source problem} \eqref{SP}, an {\em accuracy controlled solver} has been 
developed in \cite{DGM2020}. Details about its implementation can be found in \cite{GKM2017,DGM2020}. It is based on the same paradigm of contriving first
idealized iteration schemes -- in this case a fixed point scheme -- in function space followed by an approximate realization of
each iteration step within dynamically updated accuracy tolerances, monitored
by rigorous a posteriori quantities. This hinges crucially on {\em uniformly stable
variational formulations} for the radiative transfer equation, that allows us to
tightly estimate errors by residuals in the dual norm of the test space, despite the lack of coercivity. 
The evaluation of this dual norm, in turn, is facilitated by  {\em Discontinuous Petrov Galerkin} (DPG) concepts. We briefly recapitulate in Section~\ref{sec:vf} the relevant facts needed in this paper. While rigorous a posteriori error control that does not require excess regularity is well-understood for elliptic problems and their close relatives, this is by far not the case for 
the Boltzmann operators arising in the present situation. 
Regarding the role of these latter findings for the present work,
  writing the source problem 
\eqref{SP} abstractly as $\cB u=q$, this allows us to build a routine
\be
\label{routine}
\cB, \,q,\, \eta \,\mapsto  [\cB^{-1},q,\eta]\quad \mbox{such that}
\quad \big\Vert \cB^{-1} q-[\cB^{-1},q,\eta]\big\Vert_{L_2(\Ddom\times \Vdom)} \le \eta.
\ee
In this paper we invoke this routine as a  central tool for an accuracy controlled solution of the criticality eigenvalue problem \eqref{Eig} which is the key objective of this paper. An important finding is that the eigensolver reduces in the end to
a repeated application of the source problem solver subject to judiciously chosen tolerances.

\paragraph{Narrowing the information gap} As the analysis of steps (CC1) - (CC2) shows,
a final ``perfectly certified'' practical realization requires knowledge of 
certain ``critical problem parameters'' that enter for instance, the specification of
perturbation tolerances. Our analysis shows that this can be reduced to knowing the spectral 
gap that quantifies the distance of the principal eigenvalue from the rest of the spectrum.
Of course, this is in general not known, and an interesting question is whether an initial guess on the spectral gap can be upgraded computationally. 
The answer to this question is not clear. In fact,  from an
information theoretic point of view the present kind of spectral problem is 
known to be intrinsically
hard, see e.g. \cite{ben2015computing}. So the present work may be viewed as 
addressing this fact from a ``constructive angle''. This is one of the issues addressed in 
Section \ref{sec:inverse-power-it}.

But even if the spectral gap were known, a corresponding neighborhood of the principal eigenpair
that guarantees convergence of the Newton scheme may be very small so that finding an
admissible initial guess is far from trivial. 
In finite dimensions the {\em power method} often serves to complement  
a locally convergent higher order scheme because initial guesses are then far less 
constrained. According to (CC1) - (CC2), we discuss 
 in Section \ref{sec:inverse-power-it} also the power method as an idealized 
 iteration in function space. Since in this case a perturbed realization is simpler
 than for the Newton scheme, we only discuss convergence in function space, i.e., (CC1).
Indeed, this is the critical issue in the present scenario of an {\em unsymmetric compact operator}.  We have not been able
to find any related results on quantifiable convergence  in the literature. 
One obvious reason is that the typical proof strategies in the matrix case do not carry over. In particular,  one cannot in general make use of an eigenbasis with quantifiable
stability properties. Our approach is therefore quite different and hinges on 
functional calculus and Riesz projections. While one can establish 
convergence per se,  a quantification of convergence
properties seems to require at least some weak spectral decay properties.
We show that if the operator belongs to a {\em Schatten class}
the convergence of the power method can be quantified.
 
\subsection{Outline}
The paper is organized as follows. In Section \ref{sec:weak-formulations} we present the weak formulation of  source and eigenvalue   problems \eqref{SP} and \eqref{Eig}. We detail our working assumptions on the optical parameters essentially in line with \cite{balPhD1997,AB1999, DGM2020}, and  recall known results on the existence and uniqueness of the principal eigenpair $(\u, \l)$. Section \ref{sec:Newton-scheme} presents a Newton-based iterative scheme to approximate  $(\u, \l)$. The scheme is formulated first in an idealized setting at an infinite-dimensional functional level for which we prove local quadratic convergence to $(\u, \l)$. Section \ref{sec:Newton-scheme-perturbed} is then devoted to the approximate realization of this scheme along with the analysis of the perturbed scheme. In the spirit of \cite{DGM2020}, the strategy hinges on realizing the Newton updates
within some judiciously chosen error tolerance, so as to arrive in the end at the desired target accuracy by controlling the deviation of the numerically perturbed iteration from the idealized one. 
In Section \ref{sec:Newton-scheme-perturbed},  we develop strategies that allow us to compute 
approximate Newton-updates within desired accuracy tolerances by just  repeatedly applying the routine  \eqref{routine} in conjunction with
accuracy controlled application schemes of global operators, also discussed in \cite{DGM2020}. Finally, in Section \ref{sec:inverse-power-it} we discuss the generation of  
a suitable initial guess that lies in the convergence neighborhood.
We show that, under certain weak structural spectral assumptions, an inverse power iteration indeed converges. We conclude in Section \ref{sec:conclusion} with some comments putting the findings in this paper into further perspective.

%%%%%%%%%%%%%%%%%%%%
\section{Weak Formulation of the Eigenvalue Problem}
\label{sec:weak-formulations}
%%%%%%%%%%%%%%%%
Our approach hinges on a stable ultra-weak formulation of the source problem \eqref{SP}. It requires certain (rather mild) assumptions on the optical parameters which we give in Section \ref{sec:ass-opt-param}. We then define the formulation in Section \ref{sec:vf} and discuss the properties of the eigenvalue with smallest modulus in Section \ref{sec:krein-rutman}. 

\subsection{Assumptions on the domain and the optical parameters}
\label{sec:ass-opt-param}
We will work with the following assumptions  {which are in essence those adopted
in related previous works like \cite{GLPS1988,balPhD1997,AB1999,DGM2020,DLvol6} as well: 
\begin{itemize}
\item[(H1)] The spatial domain $\Ddom$ is bounded, convex and has piecewise $\Cop^1$ boundary. This guarantees that $\partial\Ddom$ is smooth enough to have uniquely defined {\em unit outward normals} $\n(x)$ at almost all points $x\in \partial\Ddom$.
\item[(H2)] The velocity domain $\Vdom$ is a compact subset of $\bR^d$ which does not contain $0$. Its $d$-dimensional \wo{Haar} measure is assumed to have mass $|\Vdom|=1$.
\wo{Moreover, $\Vdom$ will always be assumed to contain a set which is homeomorphic to the sphere in $\bR^d$.}  Any vector $\v\in \bR^d$ can be identified with the pair
$$(\s,E) = (v/|v|, |v|^2/2) \in \bS^{d-1}\times \bR^+,$$ where $\s$ is the direction of propagation in the unit sphere $\bS^{d-1}$ of $\bR^d$, and $E$ is the kinetic energy. Hence, 
we can adopt the identification $\Vdom = \bS^{d-1} \times \Edom$ with $\Edom = [E_{\min}, E_{\max}]$ and $0< E_{\min} \leq E_{\max}$.
{
\item[(H3)] The nonnegative kernels $\kappa, \varphi: \Ddom\times \Vdom\times \Vdom \to \bR_+$ \corrw{belong to $L_2(\Ddom\times \Vdom\times \Vdom)$ and} satisfy
\begin{equation}
\label{kernels}
\begin{array}{ll}
\max\Big\{\int_{\Vdom} \kappa (x,v',v)\dv,\, \int_{\Vdom} \kappa (x,v,v')\dv\Big\} \le M,& (x,v')\in \Ddom\times \Vdom\\[2mm]
\max\Big\{\int_{\Vdom} \varphi (x,v',v)\dv,\, \int_{\Vdom} \varphi (x,v,v')\dv\Big\} \le M, & (x,v')\in \Ddom\times \Vdom,
\end{array}
\end{equation}
for some constant $M<\infty$.
\item[(H4)]
The cross-section $\sigma:\Ddom\times \Vdom \to \bR_+$ is bounded on $\Ddom\times \Vdom$ and there
exists an $\alpha >0$ such that
\be 
\label{pos}
\min\Big\{\sigma(x,v')-\int_{\Vdom}\kappa(x,v',v)\dv,\, \sigma(x,v')-\int_{\Vdom}\kappa(x,v,v')\dv\Big\}\ge \alpha,\quad (x,v')\in \Ddom\times \Vdom.
\ee
In the following, we call $\alpha$ the accretivity constant in view of property \eqref{accretive} below.
\item[(H5)]
In addition we assume that the fission kernel is strictly positive, i.e.,
there exists a $c_f>0$ such that
\be
\label{fispos}
\varphi(x,v, v')\ge c_f ,\quad (x,v, v')\in \Ddom\times \Vdom\times \Vdom.
\ee
}
\item[(H6)]  \wo{$\varphi\in C(\overline{\Ddom};L_2(\Ddom\times\Vdom))$. }
\end{itemize}
\bigskip

Note that (H4) implies 
\be
\label{sigma+}
\sigma(x,v) \ge \alpha>0,\quad \forall\,(x,v)\in \Ddom\times \Vdom.
\ee
Moreover, (H4) and (H5) mean that the difference between absorption and scattering remains
positive and that fission takes place everywhere in the phase space.
It has been pointed out in \cite{AB1999} that these assumptions
can be relaxed in favor of physically more realistic ones. In fact,
at the expense of some additional technical effort, all conclusions
remain valid when replacing \eqref{fispos} by the  more physical 
assumption
\be\label{relaxed}
\kappa(x,v,v')+ \varphi(x,v,v')\ge c_f,\quad x\in \Ddom,\, v,v'\in \Vdom,
\ee
\corrw{where now the two kernels may vanish at mutually different places, see
\cite[Chapter II.2]{balPhD1997}.}
Under these premises the scattering and fission operators
\be
\label{eq:scat-kernel}\displaystyle
\begin{array}{c}
u \mapsto  (\Kop u)(x,v) \coloneqq \displaystyle\int_{\Vdom} \kappa(x, v, v') u(x, v') \dv'\\[3mm]
u \mapsto  (\Fop u)(x,v) \coloneqq \displaystyle \int_{\Vdom} \varphi(x, v, v') u(x, v') \dv'
\end{array}
\ee
are bounded linear operators from $L_2(\Ddom\times\Vdom)$ to $L_2(\Ddom\times\Vdom)$,
i.e., 
\be
\label{L2}
\Kop,\,\Fop \in \cL\big(L_2(\Ddom\times \Vdom),L_2(\Ddom\times \Vdom)\big).
\ee
It will be convenient to introduce the bilinear forms
\be
\label{kf}
\begin{array}{c}
k(u,w):= \displaystyle\int_{\Ddom\times\Vdom\times \Vdom}\kappa(x,v',v)u(x,v') w(x,v')\dx\dv\dv',\\[3mm] 
f(u,w):= \displaystyle\int_{\Ddom\times\Vdom\times \Vdom}\varphi(x,v',v)u(x,v') w(x,v')\,\dx\dv\dv',
\end{array}
\ee
corresponding to their variational definitions
% \mnote{wouldn't it be better to use upper case letters $K,F$ for the kernels? \om{Check changes on operators B, T, K, F and shortcuts for them.}}
\be
\label{KF}
(\Kop u)(w)= k(u,w),\quad (\Fop u)(w)= f(u,w),\quad u, w\in L_2(\Ddom\times \Vdom),
\ee
where $\Kop u, \Fop u$ are viewed as functionals acting on $L_2(\Ddom\times \Vdom)$.
%\om{[OM: Shall we specify here that this is the duality pairing, not to be confused with the scalar product?]}
  
\subsection{Variational Formulation of the Source Problem and Related Mapping Properties}
\label{sec:vf}
\newcommand{\bGG}{\mathbb{G}}
In order to eventually arrive at a properly posed eigenproblem we will 
introduce first a stable weak formulation of the source problem
\eqref{SP} following in essence \cite{DGM2020}. \wo{An important first step is to determine the mapping properties of the pure transport operator which, in strong form, is given as $\Top u= v\cdot\nabla u + \sigma\,u$. Introducing the {\em graph space}
\begin{equation}
%\label{H}
\Hsp (\Ddom\times \Vdom)\coloneqq \{w \in L_2(\Ddom\times \Vdom) \;:\; \v \cdot\nabla w \in L_2(\Ddom\times \Vdom)\}= \{ u\in L_2(\Ddom\times\Vdom): \Top u\in L_2(\Ddom\times\Vdom)\},
\end{equation}
endowed with the   norm
\be
\label{GN}
\Vert w\Vert _{\Hsp (\Ddom\times \Vdom)}^2:= \Vert w\Vert _{L_2(\Ddom\times\Vdom)}^2 +
\Vert \v\cdot\nabla w\Vert ^2_{L_2(\Ddom\times\Vdom)},
\ee
one readily sees that the bilinear form $t(u,w):= \int_{\Ddom\times\Vdom}w \Top u\,dx dv$
is continuous over $\Hsp (\Ddom\times \Vdom)\times L_2(\Ddom\times\Vdom)$. Under the given
assumptions on the optical parameters, one can show that $\Top$ is lower bounded on $\Hsp (\Ddom\times \Vdom)$ (see e.g. \cite{DGM2020}), i.e., there exists a positive constant $c>0$ such that
$\|\Top u\|_{L_2(\Ddom\times \Vdom)}\ge c\|u\|_{L_2(\Ddom\times \Vdom)}$. Moreover,
it is well-known that if a function $w\in \Hsp (\Ddom\times \Vdom)$ vanishes on either $\Gamma_{\pm}$
it possesses a trace on $\Gamma_{\mp}$ that belongs to the weighted $L_2$-space with norm $\|u\|^2_{L_{2,v}(\Gamma_{\mp})}=\int_{\Gamma_{\mp}} n\cdot v u^2\,d\gamma$ and
\be
\label{G0}
\bH_{0,\mp}(\Ddom\times\Vdom):= {\rm clos}_{\|\cdot\|_{\Hsp (\Ddom\times \Vdom)}}\{u\in C^1(\Ddom\times\Vdom): 
u|_{\Gamma_\mp}=0\}
\ee
is a closed subspace of $\Hsp (\Ddom\times \Vdom)$.  
One then deduces that the operator $\Top$, induced by the weak formulation
\be
\label{swT}
\<\Top u,w\>= \int_{\Ddom\times\Vdom}\big\{v\cdot\nabla u+ \sigma u\big\}w\,\dx\dv= \<f,w\>,
\quad w\in L_2(\Ddom\times\Vdom),
\ee
 is an isomorphism from either space $\bH_{0,\mp}(\Ddom\times\Vdom)$ onto $L_2(\Ddom\times\Vdom)$, and that the norm equivalences 
\be
\label{eqnorms}
\|u\|_{\Hsp (\Ddom\times \Vdom)}\eqsim  \|\Top u\|_{L_2(\Ddom\times\Vdom)}\eqsim \|v\cdot \nabla u\|_{L_2(\Ddom\times\Vdom)}
%\eqsim \big(\|u\|^2_{L_2(\Ddom\times\Vdom)}+ \|v\cdot\nabla u\|^2_{L_2(\Ddom\times\Vdom)}\big)^{1/2} := \|u\|_{H(\Ddom\times\Vdom)}
,\quad u\in  
\bH_{0,\mp}(\Ddom\times\Vdom),
\ee
hold, see e.g. \cite{DHSW2012}. Completely analogous relations are valid for the formal adjoint
$\Top^* u= -v\cdot\nabla u + \sigma u$ (to be distinguished from the {\em dual operator}
$\Top': L_2(\Ddom\times\Vdom)\to \Hsp_{0,\pm} (\Ddom\times \Vdom)'$).

There is a second,    referred to as
{\em ultraweak} formulation. Applying integration by parts, yields  
$$
\<\Top u,w\>:=
\int_{\Ddom\times \Vdom}\big\{ {v\cdot\nabla u}+\sigma u\big\}w\,\dx\dv= 
\int_{\Ddom\times \Vdom} u\big\{\sigma w -\worr{v\cdot \nabla w}\big\} \dx\dv + \int_{\partial\Ddom\times\Vdom} (v\cdot n)uwd\gamma.
$$
The integral over $\Ddom\times\Vdom$ is now well-defined for $u\in L_2(\Ddom\times\Vdom)$ and $w\in \Hsp(\Ddom\times\Vdom)$. \worr{While $\int_{\partial\Ddom\times\Vdom} (v\cdot n)uwd\gamma$ is not defined for any $u\in L_2(\Ddom\times\Vdom)$, substituting for
$u$ given boundary values $g\in L_{2,v}(\Gamma_-)$ and restricting the test functions $w$
to $ \Hsp_{0,+}(\Ddom\times \Vdom)$,}   the trace integral becomes $\int_{\Gamma_-} (v\cdot n)gwd\gamma$ and is thus well-defined. \worr{For homogeneous boundary conditions the boundary integral vanishes.} The corresponding ultraweak formulation reads
\begin{align}
\label{uwT}
t(u,w)&= \int_{\Ddom\times \Vdom} u\big\{\sigma w -\worr{v\cdot \nabla w}\big\} \dx\dv=\<u,\Top^* \worr{w}\>\nonumber\\
&= \int_{\Gamma_-}|v\cdot n|g w\,d\gamma + \<f,w\>, \quad w\in\Wsp:= \Hsp_{0,+}(\Ddom\times \Vdom),
\end{align}
induces the continuous extension of $\Top: \Hsp_{0,-}(\Ddom\times\Vdom)\to L_2(\Ddom\times\Vdom)$ to $\Top:L_2(\Ddom\times\Vdom)\to \big(H_{0,+}(\Ddom\times\Vdom)\big)'$, where we use for convenience the same notation for the transport operator
induced by the weak formulation $\<\Top u,w\>=\<f,w\>$ or by \eqref{uwT}. Note that in the latter case boundary conditions become part of the ``right-hand side functional'' and
are thus {\em natural boundary conditions}. Since both formulations will be used later we summarize the relevant facts for later record as follows.
\begin{theorem}[see \cite{DGM2020}]
 \label{thm:T-1}
 Let
 \be
 \label{UW}
 \U:= L_2(\Ddom\times\Vdom),\quad \Wsp:= H_{0,+}(\Ddom\times\Vdom).
 \ee
 Under the assumptions (H1), (H2) and equation \eqref{sigma+} of (H5), $\Top$ defined
 by \eqref{swT} is boundedly invertible as mapping from $H_{0,-}(\Ddom\times\Vdom)$ onto $\U$, i.e.,
 \be
 \label{Ts-1}
 \Top \in \cL(H_{0,-}(\Ddom\times\Vdom),\Usp),\quad \Top^{-1}\in \cL(\Usp,H_{0,-}(\Ddom\times\Vdom)).
 \ee
 Moreover, for $\Top$, defined by \eqref{uwT}, 
 one has
\be
\label{T-1}
\Top \in \cL(\Usp,\Wsp'),\quad \Top^{-1}\in \cL(\Wsp',\Usp).
\ee
\end{theorem}

In very much the same way we can define the ultraweak formulation of the complete 
radiative transfer source problem as: find $u\in \U=L_2(\Ddom\times\Vdom)$ such that for
$q\in \Wsp'\supset \U$
\be
\label{bf}
b(u,w):=\int_{\Ddom\times \Vdom}\Big\{ u\big(\sigma w- \corrw{v\cdot\nabla w}) -\int_{\Vdom} \ks\, u w \dv'\Big\}\dx \dv 
= \int_{\Gamma_-}|v\cdot n|gw\,d\gamma + q(w),\quad w\in \Wsp.
\ee
 Note that, as natural boundary conditions, they need not be
 incorporated in any trial space. Of course, this part of the right-hand side disappears
 for homogeneous boundary conditions $g=0$ considered in the eigenproblem.}

In summary, for $q\in \Wsp'$  finding $u\in \Usp$ such that
\begin{align}
\label{SP-w}
b(u,w)= q(w),\quad w\in \Wsp,
\end{align}
is called 
{\em ultra-weak formulation} of the source problem \eqref{SP} with homogeneous
inflow boundary conditions. As usual, when $q\in L_2(\Ddom\times \Vdom)$, 
we have  $q(w) = \int_{\Ddom\times \Vdom} q w \dx\dv$.
The accuracy controlled approximate solution of
\eqref{SP-w} will be a central constituent of subsequent eigensolver strategies.

Splitting 
$$
b(\cdot,\cdot)= t(\cdot,\cdot)- k(\cdot,\cdot),
$$
where $t(\cdot,\cdot)$ is the bilinear form from \eqref{uwT} and $k(\cdot,\cdot)$ of the scattering part,
it is straightforward to see that the bilinear forms $b(\cdot,\cdot)$,
$t(\cdot,\cdot)$, $k(\cdot,\cdot)$, $f(\cdot,\cdot)$ are continuous on $\Usp\times \Wsp$.
Defining in analogy to \eqref{KF} the operators
$$
(\Bop u)(w)= b(u,w),\quad (\Top u)(w) = t(u,w),\quad \forall\,u\in \Usp,\, w\in \Wsp,
$$ 
we have, in line with the remarks at the end of the previous section, that
$\Bop, \Top, \Kop, \Fop \in \cL(\Usp,\Wsp')$.
This allows us to interpret \eqref{SP-w} as an operator equation
\begin{equation}
\label{SP-op}
\Bop u = (\Top -\Kop)u = q.
\end{equation}
Its unique solvability hinges on   the mapping properties of
the transport operator $\Top$, given in Theorem \ref{thm:T-1}.
 In particular, using the norm equivalences \eqref{eqnorms}
in conjunction with the accretivity of the Boltzmann operator $\Bop$
\begin{equation}
\label{accretive}
(\Bop u,u)\ge \alpha \Vert u\Vert ^2_{L_2(\Ddom\times \Vdom)},\quad \forall u\in \Hsp _{0,-}(\Ddom\times \Vdom),
\end{equation}
which readily follows from equation \eqref{pos} of (H4), one derives
the following facts with the aid of 
the Babu$\check{s}$ka--Ne$\check{c}$as-Theory.
 \begin{theorem}[see \cite{DGM2020}]
\label{thm:Binverse}
If hypothesis (H1) -- (H4)   hold, then $\Bop$ is a linear norm-isomorphism from
$\Usp$ onto $\Wsp'$, i.\,e., $\Bop^{-1}\in \cL(\Wsp',\Usp)$ exists and the condition number $\kappa_{\Usp,\Wsp'}(\Bop)
:= \Vert \Bop\Vert _{\cL(\Usp,\Wsp')}\Vert \Bop^{-1}\Vert _{\cL(\Wsp',\Usp)} <\infty$ is finite.
\end{theorem}
\begin{remark}
\label{rem:weak}
The bound 
\begin{equation}
\label{eq:Binv}
\Vert \Bop^{-1} \Vert_{\cL(\Usp, \Usp)} \leq \alpha^{-1},
\end{equation}
immediately follows from   \eqref{accretive}. Employing the norms $\Vert \Top^*w\Vert _{\Usp}$
or $\Vert \Bop^*w\Vert _{\Usp}$ for $\Wsp$ is expected to yield more favorable bounds for $\Vert \Bop^{-1}\Vert _{\cL(\Wsp',\Usp)}$ when $\alpha$ is small in view of the 
Babu$\check{s}$ka--Ne$\check{c}$as Theorem, \wo{see \eqref{cC}}.
\end{remark}

We proceed collecting a few more prerequisites for a proper formulation 
of the eigenproblem \eqref{Eig}. A key observation
concerns
the {\em weighted} $L_2$-space 
$$
\Us := L^2_\sigma(\Ddom\times \Vdom),\quad \Vert u\Vert ^2_{\Us}
:= \int_{\Ddom\times \Vdom}\sigma |u|^2\dx \dv.
$$
Under our assumptions on the total cross-section $\sigma$, 
one has the norm equivalence
\be
\label{Leq}
\alpha \Vert u \Vert_{\Usp} \leq \Vert u\Vert_{\Us}\leq \Vert u\Vert_{\Usp} \max_{(x, v)\in \Ddom\times \Vdom} \sigma(x,v),\quad \forall u\in \Usp.
\ee
Therefore, the statements in Theorems \ref{thm:T-1}, \ref{thm:Binverse} remain
valid (with constants depending now also on the lower and upper bound for $\sigma$).
\wo{The following fact is shown in \cite[Chapter II.2]{balPhD1997}. 
\begin{lemma}
\label{lem:contr}
Adhering to the above notation, one has
\be
\label{T-1K}
\Vert \Top^{-1}\Kop\Vert _{\cL(\Us,\Us)} \le \rho <1,
\ee
with $\rho$ depending on $\alpha$.
\end{lemma}
To ease accessibility of the underlying arguments we give a proof in Appendix \ref{appendixA}.}   

\begin{theorem}
\label{thm:positive-op}
If (H1) to (H4) hold, then $\Kop$, $\Fop$, $\Top^{-1}$ and $\Bop^{-1}$ are positive operators on
$$
\Usp^+ = L_2^+(\Ddom\times \Vdom) \coloneqq \{ u \in L_2(\Ddom\times \Vdom) \;:\; u(x,v)\geq 0,\, \text{a.e.~}(x,v)\in \Ddom\times \Vdom \}.
$$
\end{theorem}

\begin{proof}
Since this is in essence a known fact (see e.g. \cite{AB1999,balPhD1997})
it suffices to sketch the arguments.
First, $\Kop$ and $\Fop$ are positive by construction (see \eqref{eq:scat-kernel}), and positivity of $\Top^{-1}$ follows from  the method of characteristics. Now, to prove that $u = \Bop^{-1} q \in \Usp^+$ for every $q \in \Usp^+$,  one readily checks  that
  $u$   is  
the unique fixed point 
$$
u = \Top^{-1}(\Kop u +  q).
$$
Since, by \eqref{T-1K}, $\Top^{-1}\Kop$ is a contraction on $\Us$, the fixed point iteration
\be
\label{fpiter}
u_{n+1} = \Top^{-1}\Kop u_n + \Top^{-1} q = \sum_{j=0}^n (\Top^{-1}\Kop)^j \Top^{-1} q, \quad \forall n \geq 0,
\ee
converges to $u$ in $\Us$ (hence in $\Usp$). Hence, we may write
$$
u = \sum_{j=0}^\infty (\Top^{-1}\Kop)^j \Top^{-1} q.
$$
We conclude that $u\in \Usp^+$ from the fact that $\Kop$ and $\Top^{-1}$ are positive operators.
 This concludes the proof that $\Bop^{-1}$ is a positive operator.
\end{proof}

The well-posedness of the operator equation \eqref{SP-op},
based on the above choices for $\Usp$ and $\Wsp$, determines now the meaning of 
\eqref{Eig} as an operator eigenvalue problem:
\be 
\label{Eig-op}
\text{Find} \quad(\u, \l) \corr{\,\in \Uspext\coloneqq \Usp \times \bR} \quad \text{such that} 
  \quad\Bop \u = \l \, \Fop \u 
\ee
through the weak formulation
\begin{equation}
\label{Eig-w}
b(\u, w) = \l f(\u, w),\quad \forall w\in \Wsp.
\end{equation}
The analysis of this eigenproblem is based on a reformulation 
involving the inverse of the {\em Boltzmann operator} $\Bop$ as explained 
in the next section. \corr{For a subsequent analysis, it will be useful to equip the space $\Uspext$ with the norm
\begin{align}
\Vert (u,\nu)\Vert := \Big(\Vert u\Vert ^2 + \nu^2\Big)^{1/2}, \quad \forall (u, \nu) \in \Uspext.
\label{eq:norm-eigenpair}
\end{align}
}

\subsection{The Principal Eigenpair $(\u, \l)$}
\label{sec:krein-rutman}

The physical meaning of the principal eigenpair $(\u, \l)$ discussed in Section \ref{ssec:eigenprob} suggests nonnegativity of $\u$, positivity and simplicity of $\l$, and it is important that the neutronic model guarantees these properties. 
It is indeed  folklore  in the literature that, under certain assumptions on the optical parameters, the smallest eigenvalue in \eqref{Eig-op} is real, strictly positive, and simple. A detailed proof of these properties can be found in \cite{balPhD1997} (see also \cite{AB1999}).
For the convenience of the reader, we sketch the underlying line of reasoning
to an extent that is relevant for the subsequent developments in the present paper.

To that end, following in essence \cite{balPhD1997},  a first step is to reformulate \eqref{Eig-op} such that
Krein-Rutman theory is applicable. As shown in the previous section,
$\Bop$ is an isomorphism from $\Us$ to $\Wsp'$ so that
\be
\label{Cop}
\Cop \coloneqq \Bop^{-1} \Fop \in \cL(\Us,\Us) = \cL(\Usp,\Usp).
\ee
Thus, \eqref{Eig-op} is equivalent to finding $(\u,\l)\in \Uspext$ such that
\be
\label{C}
\Cop \u = (\l)^{-1} \u.
\ee
Since we will now be concerned with operators 
in $\cL(\Usp,\Usp)=\cL(\Us,\Us)$, we will  simplify notation writing 
$\Vert \cdot\Vert =\Vert \cdot\Vert _{\Us} = \<\cdot,\cdot\>^{1/2}$. For most statements,
it does not matter whether which of the equivalent scalar products we
are using except when the norm of $\Top^{-1}\Kop$ matters. In this sense
the above choice is in what follows just a convenient default. 
We sometimes refer to $\Us$ explicitly when the choice of the metric is
important.

It will be convenient to denote for any $\ug\in \Usp$ by
\be
\label{span}
\< \ug\> := \vspan\{ \ug \}
\ee
the linear span of $\ug$. 
 
\begin{theorem}
\label{prop;pso}
The operator $\Cop \in \cL(\Usp, \Usp)$ is positive and compact.  Moreover, there exists
a unique simple largest positive  eigenvalue $\m$ of
\begin{equation}
\label{Ceigen}
\Cop \u= \m\u .
\end{equation}
 The corresponding (up to normalization) unique eigenstate $\u$ is  non-negative. We henceforth refer to $(\m,\u)$ as principal eigenpair,  
and
\be
\label{U0}
\Usp_\circ \coloneqq\< \u\>  = \<\Cop \u \>
\ee
is the eigenspace associated to the principal eigenvalue $\l$ of our problem 
\eqref{Eig-op}.
\end{theorem}

\begin{proof} We only sketch the main arguments of the proof following in essence \cite{balPhD1997}.   \wo{For easier access to key arguments
we present additional  relevant details in Appendix \ref{appendixB}}. Denoting $\id$ as the identity operator, we write
$$
\Cop = (\Top- \Kop)^{-1}\Fop = (\id - \Top^{-1}\Kop)^{-1}\Top^{-1}\Fop.
$$
Next we use contractivity of $\Top^{-1}\Kop$ on $\Us$ to write
$$
\Cop u = \sum_{\wor{j=0}}^\infty (\Top^{-1}\Kop)^j \big(\Top^{-1}\Fop\big) u, 
$$ 
and use positivity of the operators $\Top^{-1},\Fop, \Kop$ to 
confirm positivity of $\Cop$. 

\wo{Since $(\id - \Top^{-1}\Kop)^{-1}=\Bop^{-1}\in
\cL(\Wsp',\Us)\supset \cL(\Us,\Us)$ compactness
of $\Cop$ follows  as soon as we have established compactness of $\Top^{-1}\Fop$.
To that end, it suffices to confirm compactness of $\cF^*\cT^{-*}: \cL(L_2(\Ddom\times\Vdom),L_2(\Ddom\times\Vdom))$. In view of \eqref{T-1}, this in turn
follows when $\cF^*\in \cL(H_{0,+}(\Ddom\times \Vdom),L_2(\Ddom\times\Vdom))$ is compact.
Under the given assumptions on $\Ddom,\Vdom$, condition (H6) allows us to apply
\worr{\cite[Corollary 1, Appendix to \S 5, p. 415]{DLvol6}} which states exactly this fact.}
\wo{Since  applying bounded operators to a compact one preserves compactness, the compactness of $\Cop$ follows from the above Neumann series.
 Krein-Rutman's Theorem
asserts then that the spectral radius of $\Cop$ agrees with a positive eigenvalue $\m$
($=\l^{-1}$), associated with a non-negative eigenstate $\u$.
The proof of strict positivity of $\u$ (except on $\Gamma_-$)
and of simplicity of $\l$ can be found in the thesis 
\cite[Chapter II.2]{balPhD1997}. We state these claims as a lemma whose proof is given
for the interested reader in Appendix B.}
\wo{
\begin{lemma}
\label{lem:bal}
The eigenvector $\l$ is simple. Moreover,  the associated eigenstate $\u$ is (up to normalization)  unique and is strictly positive on $(\Ddom \times \Vdom)\setminus \Gamma_-$.
\end{lemma}
This concludes the proof of Theorem \ref{prop;pso}.
}
\end{proof}

\wo{
In the original terms Theorem \ref{prop;pso} can be restated as follows.
\begin{corollary}
\label{cor:spectrum}
If (H1) to (H6) hold, then the eigenvalue problem \eqref{Eig-op}
has a smallest  simple real  positive eigenvalue $\l=(\m)^{-1}$ where $\m$ is the largest eigenvalue of $\Cop$ from Theorem \ref{prop;pso}. Its associated eigenvector $\u$ is positive   and   equals  the  eigenstate of $\Cop$ associated to $\m$.
\end{corollary}
}

\section{Strategy and Concepts for (CC1), (CC2)}
\label{sec:Newton-scheme}
%%%%%%%%%%%%%%%%%%%%%%%%%%%%%%%%%
 This section is devoted to preparing the theoretical foundations for the
paradigm outlined in  Section \ref{ssec:novelty}.
In view of the tempting prospect of a quadratic convergence we 
focus first on    {\em Netwon's method} as idealized iteration in the {\em infinite-dimensional space} $\Usp$, see (CC1). We  indeed show local quadratic convergence to the exact principal eigenpair $(\u,\l)$ of problem \eqref{Eig} provided that the initial guess
belongs to a certain neighborhood that can be quantified in terms of the  spectral gap
(to be defined below) and mapping properties of resolvent. Based on these findings we analyze perturbed versions (CC2)
that serve as the basis of numerical realizations.
%We formulate the eigenstate problem as a Newton iteration in the infinite-dimensional space $\Usp$. 
Such an envisaged numerical realization is to approximately realize the required operations 
within judiciously chosen, {\em dynamically updated accuracy tolerances}. Their validity
needs to be confirmed 
in an a posteriori fashion to avoid requiring any excess regularity.  A key role is played
by establishing refined mapping properties of resolvents restricted to certain subspaces.
\wo{Properties of resolvents for compact operators are most conveniently described for
Hilbert spaces over $\bC$. We therefore use in what follows for convenience the same notation $\Usp$ 
also for its complexification.
}

\subsection{\wor{Newton's method}}\label{ssec:Newton}
%%%%%%%%%%%%%%%%%%%%%%%%
As noted earlier, the problem of finding the principal eigenpair of the 
generalized operator eigenproblem \eqref{Eig-op} is equivalent to solving
the standard eigenproblem \eqref{Ceigen} with $\m=(\l)^{-1}$.
We recall that we will use the notation   
$\Vert \cdot\Vert =\Vert \cdot\Vert _{\Us} =: \<\cdot,\cdot\>^{1/2}$.
 It will be convenient to normalize eigenstates such that
\be
\label{gauge}
 \Vert\Cop \u \Vert=2
 \quad \Leftrightarrow \quad
 \Vert \u \Vert= 2 \l \wo{= 2/\m}.
 \ee

To formulate a Newton scheme, 
we define the {\em residual} function
\begin{align*}
R: \Uspext &\to \Uspext \\
(u, \nu) &\mapsto
R(u, \nu) \coloneqq
\begin{bmatrix}
R_1(u, \nu) \\
R_2(u, \nu)
\end{bmatrix}
=
\begin{bmatrix}
\Mop_\nu u\\
1 - \Vert \Cop u \Vert^2 /2.
\end{bmatrix}
\end{align*}
where, \wo{for every $\nu\in \bC\setminus\{0\}$,}  
$$
\Mop_\nu \coloneqq \id - \nu \Cop    %=\nu^{-1}\Rop_{\Cop}(\nu^{-1}).
$$
is an operator from $\Usp$ into itself. 

 The (Fr\'{e}chet) derivative $DR(\bar u,\bar\nu): \Uspext\to 
 \Uspext$ of $R$ at a
 point $(\bar u, \bar \nu)$, evaluated at a point $(u,\nu)$, is given by
\be
\label{Rlin}
DR(\bar u, \bar \nu) (u,\nu)
\coloneqq
\begin{bmatrix}
\Mop_{\bar \nu}  & - \Cop \bar u  \\
- \left< \Cop \bar u, \cdot \right> & 0 \\
\end{bmatrix}
\begin{bmatrix}
u \\ \nu
\end{bmatrix}.
\ee

\wo{Assuming for the moment invertibility of $DR(\bar u, \bar \nu) (u,\nu)$,} the Newton scheme consists in building a sequence
\begin{equation}
\label{Newton}
(u_{n+1}, \lambda_{n+1}) = (u_n+\delta_n^{(u)}, \lambda_n +\delta_n^{(\lambda)}),\quad \forall n\in \bN,
\end{equation}
starting from a given initial guess $(u_0, \lambda_0)\in \Uspext$. The component $(\delta_n^{(u)}, \delta_n^{(\lambda)})\in \Uspext$ is an incremental update satisfying
\begin{equation}
\label{newton-update}
DR(u_n, \lambda_n)(\delta_n^{(u)}, \delta_n^{(\lambda)}) = - R(u_n, \lambda_n).
\end{equation}
The operator equation \eqref{newton-update} has the block-structure of a saddle point problem which can be written as
\be
\label{spp}
\begin{bmatrix}
\Mop_{\lambda_n}  & - \Cop u_n  \\
- \left< \Cop u_n, \cdot \right> & 0 \\
\end{bmatrix}
\begin{bmatrix}
\delta_n^{(u)} \\ \delta_n^{(\lambda)}
\end{bmatrix}
=
- 
\begin{bmatrix}
\Mop_{\lambda_n} u_n\\
1-\Vert \Cop u_n \Vert^2/2
\end{bmatrix}.
\ee
Assuming that $\lambda_n \neq \l$, and applying $\Mop_{\lambda_n}^{-1}$ to the first line yields
\begin{equation}
\label{eq:hn}
\delta_n^{(u)} = \delta_n^{(\lambda)} \Mop_{\lambda_n}^{-1} \Cop u_n - u_n .
\end{equation}
Since, from the second line, $\left< \Cop u_n , \delta_n^{(u)}\right> = 1-\Vert \Cop u_n \Vert^2/2$, we deduce the value of $\delta_n^{(\lambda)}$ by taking the scalar product with $\Cop u_n$ in \eqref{eq:hn}, which yields
%\mnote{I think the original formula for $\delta_n^{(\lambda)}$ is not correct because $\<\Cop u_n,u_n\>
%\neq \Vert \Cop u_n\Vert ^2$}
%\omnote{OM: I agree. There was a mistake.}
\be
\label{mun}
\delta_n^{(\lambda)} = \frac{1+\<\Cop u_n,u_n\> - \Vert \Cop u_n\Vert ^2/2}{\left< \Cop u_n, \Mop_{\lambda_n}^{-1} \Cop u_n\right>}.
\ee
The updated iterates \eqref{Newton} can thus be written as
\begin{align}
 u_{n+1} &= \delta_n^{(\lambda)} \Mop_{\lambda_n}^{-1} \Cop u_n 
 = \frac{1+\<\Cop u_n,u_n\> - \Vert \Cop u_n\Vert ^2/2}{\left< \Cop u_n, \Mop_{\lambda_n}^{-1} \Cop u_n\right>}
 \Mop_{\lambda_n}^{-1} \Cop u_n \\[3mm]
 \lambda_{n+1} &= \lambda_n + \frac{1+\<\Cop u_n,u_n\> - \Vert \Cop u_n\Vert ^2/2}{\left< \Cop u_n, \Mop_{\lambda_n}^{-1} \Cop u_n\right>}.
\end{align}
The whole scheme relies on \wo{evaluating $\Mop_{\lambda_n}^{-1} \Cop u_n$, i.e., 
on invertibility of $\Mop_{\lambda_n}$, which
requires}  finding at each step the function $z_n\in \Usp$ such that
\begin{equation}
\label{eq:main-pb-to-invert}
\Mop_{\lambda_n} z_n = \Cop u_n.
\end{equation}

Solving \eqref{eq:main-pb-to-invert} has to be handled with care since, on the one hand, the sequence $\big( (u_n,\lambda_n)\big)_{n\geq 0}$\\ is  to converge to the principal eigenpair $(\u,\l)$ but, on the other hand, as $(u_n, \lambda_n)$ approaches $(\u,\l)$, the condition of the operator $\Mop_{\lambda_n}=\id - \lambda_n \Cop$ tends to infinity. In particular, we lose uniqueness since the operator $\Mop_\l$ has the nontrivial kernel $\Usp_\circ$ from \eqref{U0}. However, by uniqueness of the principal eigenpair $(\u,\l)$, the restriction $\Mop_\l \vert_{\Usp_\circ^\perp}$ of $\Mop_\l $ to $\Usp_\circ^\perp$ is injective. This motivates to study mapping properties at the principal eigenpair in order to rigorously prove a quantifiable convergence of the \worr{Newton scheme}.

%%%%%%%%%%%%%%%%%%%%%%%%%%%%%%%%%%%
\subsection{\wo{Some Prerequisites from the Spectral Theory of Compact Operators}}\label{ssec:spec}

To quantify the mapping properties of the Jacobian \eqref{Rlin} as well as to prove
later convergence of the power method as a means to generate suitable initial guesses
(see Theorem \ref{thm:power} in Section \ref{sec:inverse-power-it}}) we draw on some classical
facts from the spectral theory of compact operators on Hilbert spaces, see e.g. \cite{Conway}. Let $\sigma(\Cop)$ denote the {\em spectrum} of $\Cop$ (\worr{which contains \{0\} for infinite-dimensional Hilbert spaces}). 
Since there is hardly any confusion with indexing in the \worr{Newton iterates} it will be convenient to
enumerate the elements of the
spectrum
$$
\sigma(\Cop)= \{\mu_j:j=1,\ldots,\infty\}
$$
where $|\mu_j|$ decreases with increasing $j\in\N$. Thus, we replace the original
notation $(\l,\ug^\circ)$ for the principal eigenpair by $(\mu_1,\ug_1)$,
keeping in mind that $\mu_1= \mu^\circ= (\l)^{-1}$, $\ug^\circ=\ug_1$.

The open set 
$\rho(\Cop):= \bC\setminus \sigma(\Cop)$ is called {\em resolvent set}.
The operator
\be
\label{resolventop}
\Rop_{\Cop}(\zeta):= (\zeta\id -\Cop)^{-1},\quad \zeta\in \rho(\Cop),
\ee
is called the {\em resolvent operator}
\wo{which is
obviously related to the   operator $\Mop_\nu$  by  
\be
\label{resop}
\Mop_\nu = \nu \Rop_{\Cop}(\nu^{-1})^{-1}.
\ee
}
 $\Rop_{\Cop}(\zeta)$ is known to be an isomorphism on $\U$ for each $\zeta\in \rho(\Cop)$.

Below as well as later in  Section \ref{sec:inverse-power-it} the so called {\em Riesz projections} and related facts about functional \worr{calculus} play a crucial role.
For any subset $\omega\subset \sigma(\Cop)$ they are defined as
\be
\label{Riesz0}
\cE_{\Cop}(\omega)= \frac{1}{2\pi i}\int_{\Gamma(\omega)}\Rop_{\Cop}(\zeta)
d\zeta,
\ee
where \worr{we always assume a clockwise orientation of the closed rectifiable curve
$\Gamma(\omega)$ forming the boundary of a domain $\Omega\subset \bC$ that contains $\omega$
but does not intersect $\sigma(\Cop)\setminus\omega$, i.e.,} 
$$
\Gamma(\omega)=\partial\Omega(\omega)\quad\mbox{and}\quad\omega\subset\Omega(\omega), \quad\Omega(\omega)\cap \sigma(\Cop)\setminus\omega =\emptyset.
$$
$\cE_{\Cop}(\omega)$ is independent of the specific contour $\Gamma$ as long as $\Gamma$
hast the above properties.
Clearly, $\Cop$ commutes with $\Eop_{\Cop}(\omega)$. When $\omega= \{\mu\}$ contains only a single element \worr{$\mu\in \sigma(\Cop)\setminus\{0\}$}
we   write for convenience briefly $\Eop_{\Cop}(\mu):= \Eop_{\Cop}(\{\mu\})$.

Moreover, $\Eop_{\Cop}(\omega)$ is
known to be a {\em projection}, i.e., $\Eop_{\Cop}(\Omega)= \Eop^2_{\Cop}(\Omega)$,
and more generally
 \be
\label{anihi}
 {\cE}_{\Cop}(\mu) {\Eop}_{\Cop}(\mu')=\delta_{\mu,\mu'}{\cE}_{\Cop}(\mu)\quad \forall\,\, \mu,\mu'\in \sigma(\Cop),
\ee
and for any $\omega \subset \sigma(\Cop)$ one has a {\em direct sum decomposition} of
$\U$ 
\be
\label{directsum}
\U = \cE_{\Cop}(\omega)\U \oplus \, \cE_{\Cop}\big(\sigma(\Cop)\setminus\omega\big)\U,
\ee
into the invariant subspaces $\cE_{\Cop}(\omega)\U$, $\cE_{\Cop}\big(\sigma(\Cop)\setminus\omega\big)\U$ of $\U$.
Recall also, that for each $\mu\in \sigma(\Cop)$ there exists a unique $n_\mu\in\N$
such that the spaces $V_\mu:={\rm ker}(\mu\id-\Cop)^{n}$ agree for all $n\ge n_\mu$ and
$n_\mu$ is the smallest number with this property. Moreover, 
\be
\label{invariant}
V_\mu := \Eop_{\Cop}(\mu)\Usp, \quad \Cop V_\mu \subseteq V_\mu.
\ee
In these terms the decomposition \eqref{directsum} can be refined as
%one has the direct sum decomposition
\be
\label{Rdeco}
\Usp = \bigoplus_{k\in\N} V_{\mu_k} = \bigoplus_{k\in\N} \corrw{\Eop}_{\Cop}(\mu_k)\Usp.
\ee

%%%%%%%%%%%%%%%%%%%%%%%%%%%%%%%%%%%%%
\subsection{Mapping Properties at the Principal Eigenpair}

%%%%%%%%%%%%%%%%%%%%%%

In this section, we derive certain mapping properties of operators involving the principal eigenpair. These results are prerequisites for the convergence proof of   the Newton scheme which is presented in the next section. 

\wo{We begin with bounding $\Mop_\nu$. Recalling that $\Bop^{-1}\in \cL(\Wsp',\Usp)$ 
\om{and that $\Us$ is continuously embedded in $\Wsp'$},
it is useful to keep in mind that
\be
\label{cC}
\Vert \Cop\Vert _{\cL(\Us ,\Us)} \le \Vert \Bop^{-1}\Vert _{\cL(\Wsp',\Us)}\Vert \Fop\Vert _{\cL(\Us,\om{\Wsp'})}.
\ee
\wo{A specification of this bound for $\|\Cop\|_{\cL(\Us ,\Us)}$ depends on the specific case at hand. We emphasize that such a bound may be much more favorable than
  the straightforward estimate $\Vert \Cop\Vert _{\cL(\Us ,\Us)} \le \Vert \Bop^{-1}\Vert _{\cL(\Us,\Us)}
  \Vert \Fop\Vert _{\cL(\Us,\Us)} \le \alpha^{-1}
\Vert \Fop\Vert _{\cL(\Us,\Us)}$ which may be too pessimistic for small $\alpha$}
(see Remark \ref{rem:weak}). 
}
\wo{For convenience we abbreviate in what follows the norm of any operator $\Zop\in \cL(\Us,\Us)$ as
$$
\Vert \Zop\Vert = \Vert \Zop\Vert _{\cL(\Us,\Us)} = \sup_{\Vert w\Vert =1}
\Vert \Zop w\Vert .
$$
 
In subsequent discussions it will be convenient to use the shorthand notation
\be
\label{M}
M_\nu := \Vert \Mop_\nu \Vert   \le 1+|\nu| \Vert \Cop\Vert \le 2\max\{1,|\nu| \Vert \Cop\Vert \},\quad \forall \nu \in \bC.
\ee

The following lemma serves as a major tool for analyzing the mapping properties of $\Mop_\l$.
\begin{lemma}
\label{lem:acretive}
Adhering to the above notation, one has  
\be
\label{newtheta}
\theta := \inf_{w\in \U_\circ^\perp}\sup_{v\in \U_\circ^\perp}\frac{\<\Mop_\l w,v\>}{\|w\|\|v\|} >0,
\ee
see \eqref{U0} for the definition of $\U_\circ$. 

Moreover, defining
$$
\Mop_\l^\circ:= P_{\U_\circ^\perp}\Mop_\l\!\mid_{\U_\circ^\perp}\in \cL\big( \U_\circ^\perp,\U_\circ^\perp \big),
$$
where $P_{\U_\circ^\perp}$ denotes the 
$\U$-orthogonal projection of $\U$ to $\U_\circ^\perp$, we have 
\be
\label{thetaint}
\|(\Mop^\circ_\l)^{-1}\|_{\cL(\U^\perp_\circ,\U^\perp_\circ)} =\theta^{-1}.
\ee
\end{lemma}
\noindent{\it Proof.}
Since the range ${\rm ran}\,\big(\Mop_\l|_{\U^\perp_\circ}\big)$ is closed we need to show that $\u\notin {\rm ran}\,\big(\Mop_\l|_{\U^\perp_\circ}\big)$. Suppose that
there exists a $w_\circ\in \U_\circ^\perp$ such that $\Mop_{\l}w_\circ= \u$. This means
that 
$$
\<\u,v\>= \< \Mop_\l w_\circ,v\> =  \< \Mop_\l( w_\circ+ c\u),v\>,\quad \forall\, c\in\bC,\,v\in\U.
$$
This is the same as saying that there exists a $\hat v\in \U$ such that 
$\<\Mop_\l \hat v, v\> =\<\u,v\>$ holds for all $v\in\U$. But, by the previous comments, $\hat v$ can also be 
written as $\hat v= \tilde c \u + v_1$ for some $v_1\in \Vsp := \cE_{\Cop}( \sigma(\Cop)\setminus\{\m\})\U$. Since $\Mop_\l \u=0$ this means that $\Mop_\l v_1= \u$ which is a contradiction since $\Mop_\l v_1\in \Vsp $.

 To proceed,
 denote by $P_{\U_\circ^\perp}$ the 
$\U$-orthogonal projection of $\U$ to $\U_\circ^\perp$. We have shown above that 
\be
\label{kercond}
{\rm ker}\Big( P_{\U_\circ^\perp}(\id - \l\Cop)|_{ \U_\circ^\perp}\Big)
=\{0\}.
\ee
Let
$$
\Mop_\l^\circ:= P_{\U_\circ^\perp}\Mop_\l\!\mid_{\U_\circ^\perp}
$$
which, by definition, maps $\U_\circ^\perp$ into itself. As a closed subspace $\U_\circ^\perp$ is also a Hilbert space endowed with the norm $\|\cdot\|$. 
Moreover, $\Mop_\l^\circ = \id|_{\U^\perp_\circ}- \l P_{\U^\perp_\circ}\Cop\!\mid_{\U^\perp_\circ}$ and $P_{\U^\perp_\circ}\Cop$ is a compact operator taking $\U^\perp_\circ$ into itself.
Hence Fredholm's alternative is valid and says, on account of \eqref{kercond}, that 
$\Mop^\circ_\l$ maps $\U_\circ^\perp$ onto itself.
By the Open Mapping Theorem, $\Mop_\l^\circ: \U^\perp_\circ\to\U^\perp_\circ$ is boundedly invertible. The Babu$\check{s}$ka--Ne$\check{c}$as Theorem then says that
$$
\theta= \inf_{w\in \U_\circ^\perp}\sup_{v\in \U_\circ^\perp}\frac{\<\Mop^\circ_\l w,v\>}{\|w\|\|v\|}  >0.
$$
Since $\<\Mop_\l^\circ w,v\>= \<\Mop_\l w,v\>$ for all $v\in \U_\circ^\perp$ the assertion follows.\hfill $\Box$\\
%\end{proof}

The quantitative mapping properties of the Jacobian \eqref{Rlin}, and hence
the convergence properties of the Newton scheme, rely in essence 
on the size of $\theta$ as shown later below. 
The convergence of the power iteration, discussed in
Section \ref{sec:inverse-power-it}, instead  depends on another spectral property
of $\Cop$ that we introduce next.
Let
\be
\label{rat}
\rat=\rat_{\Cop}:= \frac{\l}{\La}<1,\quad \mbox{where}\quad \La:= \max\big\{|\lambda|: \lambda\neq \l,\, {\rm ker}\Mop_\lambda \neq \om{\{0\}} \big\},
\ee
which we use to encode the (relative) {\em spectral gap}
\be
\label{sgap}
\uDelta:= 1- \rat.
\ee
Although we will not make any direct use of the following remarks we   pause to briefly comment  on the relation between $\theta$ and $\uDelta$
which is in general a strict lower bound for
\be
\label{oDelta}
\oDelta := \big|1- \frac{\l}{\lambda_\La}\Big|=\big|1- \frac{\l}{\bar\lambda_\La}\Big|,
\ee
\worr{where $\lambda_\Lambda$ is the eigenvalue with $|\lambda_\Lambda|=\Lambda$.}
More precisely, $\uDelta=\oDelta$ if and only if $\lambda_\La\in \R$ which is, for instance, 
the case when $\Cop$ is a normal operator. 

In fact, it is $\oDelta$ that relates more closely to 
$\theta$ as explained next.
By \eqref{thetaint}, $\theta$ is the smallest singular value of
the mapping $\Mop^\circ_\l$ and hence of its adjoint $(\Mop^\circ_\l)^*\in \cL(\U_\circ^\perp,\U_\circ^\perp )$. Let us denote by $\rc\in \U_\circ^\perp$ 
the associated left singular vector of $\Mop^\circ_\l$ with the standard normalization $\|\rc\|=1$.
Let $\lambda_\La$ be a generalized eigenvalue of $\lambda\Cop u=u$ with second-smallest
modulus, i.e.,  
$|\lambda_\La|=\La$. Then $\bar\lambda_\La$ is a generalized eigenvalue of 
the adjoint problem $\nu\Cop^*u=u$ (with the same modulus) and let $u_\Lambda^*\in \U_\circ^\perp$ denote
the associated eigenstate.
\begin{remark}
\label{rem:relation}
Adhering to the above notation and denoting by ${\U_\circ^*}^\perp := \< u^*_\circ\>^\perp$,
where $u^*_\circ$ is the principal eigenstate of $\Cop^*$, we have 
$$  
\Vsp^*:=\cE_{\Cop^*}(\sigma(\Cop^*)\setminus\{\m\}) = \U_\circ^\perp,\quad  \Vsp= \cE_{\Cop}(\sigma(\Cop)\setminus\{\m\})= {\U_\circ^*}^\perp.
$$
(see \eqref{ident} in Appendix C),
and
\be
\label{sandwich}
\uDelta |\< \rc,u^*_{\La}\>| \le
\oDelta |\< \rc,u^*_{\La}\>| \le \theta \le \Big(1-{\rm dist}\big(\U_\circ^\perp,
{\U_\circ^*}^\perp\big)\Big)\,\oDelta \,
\ee
where 
$$
{\rm dist}\big(\U_\circ^\perp,
{\U_\circ^*}^\perp\big):= \max\Big\{ \max_{\substack{w\in \U_\circ^\perp\\
\|w\|=1}}\|w- P_{{\U^*_\circ}^\perp}w\|, \max_{\substack{w\in {\U_\circ^*}^\perp\\
\|w\|=1}}\|w- P_{{\U_\circ}^\perp} w\|\Big\}.
$$

  Moreover, when $\Cop$ is a normal operator we have
\be
\label{normal}
\theta =\oDelta =\uDelta.
\ee
The discrepancies in \eqref{sandwich} can be viewed as quantifying the deviation of $\Cop$ from normality.
\end{remark}
We provide details behind the above claims in Appendix C.
}
\bigskip

The next \worr{theorem} shows that the mapping $DR{(\u, \l)}\in \cL(\Uspext,\Uspext)$ is boundedly invertible.

\begin{theorem}
\label{lem:saddle}
At the principal eigenpair $(\u, \l)$, 
the mapping $DR{(\u, \l)}\in \cL(\Uspext,\Uspext)$ is boundedly invertible
so that, in particular,
there exists a positive constant $\beta$
such that 
\be
\label{Dinverse}
\Vert DR{(\u, \l)}^{-1} \Vert\coloneqq \Vert DR{(\u, \l)}^{-1} \Vert_{\cL(\Uspext,\Uspext)} \leq \beta,
\ee
where
\be
\label{beta}
\beta \le \left(\Big(\frac{1}{\theta} +
\frac 12 \Big(1+ \frac{M_\l}{\theta}\Big)\Big)^2 +\frac 14 \Big(1+\frac{M_\l}{\theta}\Big)^2\Big(1+\frac{M_\l}2\Big)^2\right)^{1/2}.
\ee
\end{theorem}

\begin{proof}
Recall that
\begin{align*}
DR{(\u, \l)} =
\begin{bmatrix}
\Mop_\l  & - \Cop \u  \\
- \left< \Cop \u, \cdot \right> & 0 \\
\end{bmatrix}
\end{align*}
and define 
\begin{equation}
\label{bilinla}
\begin{cases}
a(u, z) &\coloneqq \left< \Mop_\l u, z \right>,\quad \forall (u,z)\in \Usp\times \Usp,\\
 b(u,\nu) &\coloneqq - \nu \left< \Cop \u, u\right>,\quad \forall (u,\nu)\in \Usp\times \R.
\end{cases}
\end{equation}
Then, for any given $(g, \zeta)\in \Uspext$, the weak formulation of finding $(u,\nu)\in \Uspext$ such that $DR{(\u, \l)}(u,\nu) = (g, \zeta)$, is given
by
\begin{equation}
\label{saddle}
\begin{array}{lcll}
a(u, z)  + b(z, \nu) &=& \left< g, z\right> ,&\quad \forall z \in \Usp\\[2mm]
 b(u, \alpha) &=& \alpha \zeta ,&\quad \forall \alpha \in \bR.
 \end{array}
\end{equation}
We show next that \eqref{saddle} satifies the  {\em Ladyzhenskaya--Babu$\check{s}$ka--Brezzi}
(LBB)-conditions (see Theorem 2.34 of \cite{EG2013}). 
We have already seen that both
bilinear forms are continuous on $\Usp\times \Usp$ and $\Usp\times \R$, respectively.
Specifically, from \eqref{M} and \eqref{gauge}, it holds that
\be
\Vert a\Vert := \sup_{u\in \Usp } \sup_{z\in \Usp } \frac{a(u, z)}{ \Vert u \Vert \Vert z \Vert }\le M_\l,\quad \Vert b\Vert := \sup_{\nu\in \bR}  \sup_{z\in \Usp_\circ} \frac{| b(z, \nu) |}{\Vert z\Vert | \nu |}\le \Vert \Cop \u\Vert = 2.
\ee
Next, we observe that we have
$$
\Usp_\circ^\perp = \{ u\in \Usp \; : \; b(u,\nu)=0,\, \forall \nu\in \bR \},
$$
so that, by \wo{\eqref{newtheta} in Lemma \ref{lem:acretive}},
$$
\inf_{u\in \Usp_\circ^\perp} \sup_{z\in \Usp_\circ^\perp} \frac{a(u, z)}{ \Vert u \Vert \Vert z \Vert } \geq \theta.
$$
Finally,
$$
\inf_{\nu\in \bR}  \sup_{z\in \Usp} \frac{| b(z, \nu) |}{\Vert z\Vert | \nu |}
= \inf_{\nu\in \bR}  \sup_{z\in \Usp}  \frac{|\nu| \left< \Cop \u, z \right>}{\Vert z\Vert | \nu |}
= \Vert \Cop \u \Vert = 2,
$$
where we have again used the normalization $\Vert \Cop \u \Vert=2$ from \eqref{gauge}. The saddle point problem \eqref{saddle} is thus well-posed and we have the a priori estimates
\begin{equation}
\label{bounds}
\Vert u \Vert \leq \frac{1}{\theta}\Vert g\Vert  + \frac 12\Big(1+ \frac{M_\l}{\theta}\Big)
|\zeta|,\quad
| \nu | \leq   \frac 12\Big(1+ \frac{M_\l}{\theta}\Big)\Vert g\Vert  +       \frac{M_\l}{4}\Big(1+ \frac{M_\l}{\theta}\Big)|\zeta|.
\end{equation}
%with $c_1 = \theta^{-1}$, $c_2 = c_3 = \lambda(1+ \theta^{-1} \Vert a \Vert )$, $c_4 = \lambda^{-2} \Vert a \Vert (1+ \theta^{-1} \Vert a \Vert )$.
Using the norm \eqref{eq:norm-eigenpair} for the space $\Uspext$, the claims \eqref{Dinverse}, \eqref{beta} follow by straightforward calculations.
%
%we have
%\be
%\label{DRinv}
%\Vert DR{(u, \lambda)}^{-1} \Vert_{\Uspext}
%\coloneqq
%\sup_{(f,\xi)\in \Uspext} 
%\dfrac{\Vert DR{(u, \lambda)}^{-1} (f,\xi)\Vert_{\Uspext} }{\Vert (f,\xi) \Vert_{\Uspext}} \leq \gamma \coloneqq  \left( (c_1+c_3)^2 + (c_2+c_4)^2 \right)^{1/2},
%\ee
%which completes the proof.

 \end{proof}

\subsection{Perturbation results}\label{ssec:pert}
%%%%%%%%%%%%%%%%%%%%%%%%%%%%%%
%\om{Remark: Lemma 3.7, corollary 3.8, lemma 3.9 do not seem to be really used for the proof of theorem 3.5}
%\\
%\om{Notation needs to be streamlined: I would replace $\id-\lambda \Cop$ by $\Mop_\lambda$ everywhere. The notation $\bar u, \bar \lambda$ for running variables looks a bit awckward. Perhaps we could call the principal eigenstate $\bar u, \bar \lambda$. We should agree on what is a good running variable to denote functions in $\Usp$. So far we change a lot and use $h$, $z$,... I would also remove the arrows in vectors such as $\v$}

Recall from \eqref{Newton} and \eqref{newton-update}, that each Newton iteration
with current approximation $(u_n, \lambda_n)$ to the 
principal eigenpair $(\u, \l)$
requires inverting the operator $DR(u_n, \lambda_n)$. The main result in this section states that $DR(u, \nu)$ has indeed a uniformly bounded inverse for all $(u, \nu)$ in a full neighborhood of $(\u, \l)$. Thus inverting  $DR(u_n, \lambda_n)$ as required in \eqref{newton-update} is allowed as long as the $(u_n, \lambda_n)$ remain in that neighborhood. Proving this statement requires  extending Theorem \ref{lem:saddle} to a neighborhood of $(\u, \l)$. We   present this result in our next theorem. To introduce it, for a given metric space $X$, it will be convenient to denote by $B(u,\tau)$ the closed ball with center $u\in X$ and radius $\tau\geq 0$.

\begin{theorem}
\label{thm:main}
There exists a radius $\tau>0$ that depends on $\theta, \l , \Vert \Cop\Vert $ such that  
\be
\label{DRbound}
\Vert DR{(u, \nu)}^{-1} \Vert \leq \bar\beta,
\quad
\forall\, (u, \nu)\in
\cN, 
\ee
where
\begin{equation}
\label{eq:neighN}
\cN  \coloneqq B(\u,\ett)\times B\Big(\l,\frac{\theta}{4\Vert \Cop\Vert }\Big).
\end{equation}
The constant $\bar\beta$ depends only on $\l,\theta,\Vert \Cop\Vert $, and it holds that
%\omnote{Update upper bound for $\beta$. See end proof Theorem \ref{thm:main}.}
\be
\label{barbeta}
0 < \bar\beta \le 
\left( \frac{16}{\theta^2} + \Big(\frac{2\l}{2\l -\tau}\Big)^2\left( 1 + 4 \frac {\bar M }{\theta} \right)^2  \right)^{1/2}
+ \frac{2\l}{2\l -\tau}\left( 1 + 4 \frac {\bar M }{\theta} \right)
\left( 1+\frac{\bar M 4\l^2}{2\l -\tau)^2}\right)^{1/2},
%
% \left(\Big(\frac{6}{\theta} +
%  \Big(1+ \frac{6\bar M}{\theta}\Big)\Big)^2 + \Big(1+\frac{6\bar M}{\theta}\Big)^2\Big(1+ \bar M \Big)^2\right)^{1/2}\lesssim \frac{\max\{1,\l\}}{\theta},
\ee
where
\be
\label{tau}
\tau = \min\Big\{\frac{1}{8\Vert \Cop\Vert }, \frac{2\theta}{25(1+\l\Vert \Cop\Vert +\theta/4)\Vert \Cop\Vert }, \frac{\theta}{25\bar M}\Big\},
\qquad
\bar M \coloneqq 1+ \l \Vert \Cop\Vert + \frac{\theta}4. 
\ee
% is a positive constant that also depends on $\l,\theta,\Vert \Cop\Vert $.
\end{theorem}

\begin{remark}
\label{rem:tau}
Notice that the admissible radius $\tau$ depends 
on \wo{$\theta$} \om{(even linearly so, when replacing the middle expression
in the definition of $\tau$ by the bound $\frac{2\theta}{25(1+\l\Vert \Cop\Vert )\Vert \Cop\Vert }$)} which, \wo{in view of \eqref{sandwich},} reflects an increasing difficulty when 
the spectral gap \wo{$\Delta$ is small.} Specifically, suppose that   $\Vert \Cop\Vert $  has moderate size of order one, then $\bar\beta \lesssim \theta^{-1}\max\big\{1,\l\big\}$. 
\end{remark}

The remainder of this section is devoted to the proof of Theorem \ref{thm:main},
based on a number of perturbation results.

%\omnote{Lemma 3.6: The factor 1/2 in $\theta/2||\Cop||$ is arbitrary. We could generalize and take a parameter $\kappa\in (0, 1)$. Not sure that this brings a lot, though.}

\begin{lemma}
\label{lem:pert-lambda}
%Let $0 < a < 1+\theta$ and $\nu$ such that
For all $\nu\in B(\l, \frac{\theta}{2 \Vert \Cop \Vert})$, $\Mop_\nu$ is \wo{boundedly invertible on}  $\Usp_\circ^\perp$, \wo{and}
\be
\label{theta2}
\wo{\inf_{w\in \Usp_\circ^\perp}\sup_{v\in \Usp_\circ^\perp}\frac{
\left< \Mop_\nu w, v \right>}{\|w\|\|v\|} \geq \frac{\theta}2.}
\ee
\end{lemma}

\begin{proof} \wo{Boundedness of $\Mop_\nu$ follows from \eqref{M}. Regarding the
inf-sup condition, we have 
for every $w,v\in \Usp_\circ^\perp$,  
\begin{align*}
\left< \Mop_\nu w, v \right> &= \left< \Mop_\l w, v \right> - \left<  (\nu-\l) \Cop) w, v \right> \geq \theta \Vert w \Vert \Vert v\Vert  - (|\nu|-\l)   |\left< \Cop w, v \right> \\\
&\ge (\theta  - |\nu-\l|\,\Vert \Cop\Vert )\Vert u\Vert \Vert v\Vert
\end{align*}
which confirms the claim. The same argument shows validity of the swapped inf-sup
condition
$$
\inf_{v\in \Usp_\circ^\perp}\sup_{w\in \Usp_\circ^\perp}\frac{
\left< \Mop_\nu w, v \right>}{\|w\|\|v\|} \geq \frac{\theta}2.
$$
The claim follows now from the Babu$\check{s}$ka--Ne$\check{c}$as Theorem \cite{EG2013}.}
\end{proof}

\begin{lemma}
\label{lem:uubar}
For every $\e>0$,
$$
\left< \Cop u, u \right> \geq  a_\e \Vert u \Vert^2, \quad \forall u\in B(\u, \e \Vert \u \Vert),
$$
where
\be
\label{gamma}
 a_\e := (1-\e)^2\big((\l)^{-1}-\e\Vert \Cop\Vert (2+\e)\big).
\ee
Moreover, there exists $\e_0=\e_0(\Vert \Cop\Vert ,\l )$ such that 
\be
\label{la2}
 a_\e \ge \frac{1}{2\l },\quad \forall \e\le\e_0,
\ee
where under the assumption $\e\le 1$
\be
\label{eps0}
\e_0 = \frac{1}{16\Vert \Cop\Vert \l}
\ee
suffices to ensure \eqref{la2}.
\end{lemma}
\begin{proof}
Using that $\left< \Cop \u, \u\right> = (\l)^{-1} \Vert \u \Vert^2$, we have for any $u$ such that $\Vert u - \u \Vert\leq \e\Vert \u \Vert $,
\begin{align*}
\left< \Cop u, u \right>
&= \left< \Cop (u-\u), u - \u\right> + \left< \Cop \u, \u\right>
+ \left< \Cop(u - \u), \u \right> + \left< u - \u, \Cop \u \right> \\
&\geq (\l)^{-1} \Vert \u \Vert^2 - \Vert \Cop\Vert \Vert u - \u \Vert^2 - 2 \Vert \Cop \Vert \Vert \u\Vert \,\Vert  u - \u \Vert \\
&\geq (\l)^{-1} \Vert \u \Vert^2 - \Vert \Cop\Vert \e^2\Vert  \u \Vert^2 - 2 \e\Vert \Cop \Vert \Vert \u\Vert ^2  = \big((\l)^{-1}-\e\Vert \Cop\Vert (2+\e)\big)\Vert u\Vert ^2 \\
&\ge (1-\e)^2\big((\l)^{-1}-\e\Vert \Cop\Vert (2+\e)\big)\Vert  u\Vert ^2.
\end{align*}
\end{proof}

The following observation extends Lemma \ref{lem:acretive} to pairs $(u, \nu)$ near $(\u,\l)$.

\begin{lemma}
\label{lem:pert-Cu}
Let $(u, \nu)\in B(\u, \e \Vert \u \Vert)\times B\left( \l, \frac{\theta}{4 \Vert \Cop \Vert} \right)$. Then
\be
\label{Mbar}
\wo{\left< \Mop_\nu z, z'\right> \geq  \frac{\theta}4 \Vert z \Vert\Vert z'\Vert,\quad \forall \, z, z' \in \vspan\{ \Cop u \}^\perp}
\ee
provided that $\e\le \e_1$, where $\e_1= \e_1(\Vert \Cop\Vert ,\l ,\theta)$ is defined in \eqref{eps11} below.
\end{lemma}

\begin{proof}
Note first that the unperturbed result $u = \u$ and $\nu = \l$ was proven in Lemma \ref{lem:acretive}. To prove  the general case,   we fix $\e>0$ and let $(u, \nu)\in B(\u, \e \Vert \u \Vert)\times B(\l, \frac{\theta}{4 \Vert \Cop \Vert})$. That is,
$$
\Vert u - \u \Vert\leq \e\Vert  \u\Vert  \quad\text{and} \quad |\nu - \l |\le \frac{\theta}{4\Vert \Cop\Vert },
$$
i.e., $u$ and $\u$ differ by at most $\e$ in relative accuracy (and similarly for $\nu$ and $\l$).

For the fixed $u$, let now $z\in \vspan\{ \Cop u \}^\perp$, and let $\zp \coloneqq \Pop_{\Usp_\circ^\perp} (z)$ be its projection to $\Usp_\circ^\perp$. The first step of the proof consists in showing that $z$ and $z^\perp_\circ$ also differ by the order of $\e$ in relative accuracy. Denoting by $\Pop_v w = \<w,v\>v/\Vert v\Vert ^2$ the orthogonal projection of an element $w\in \Usp$ to the subspace spanned by a function $v\in \Usp$, we have

\begin{align}
%\label{S1}
z - \zp 
&= \Pop_{\u} z =  \Pop_{\Cop \u} z
= \frac{\<z,\Cop \u\> \Cop \u}{\Vert \Cop \u\Vert ^2} =  \frac{\<z,\Cop(\u- u)\> \Cop \u}{\Vert \Cop \u\Vert ^2},
 \end{align}
 where we have used that $\<z,\Cop u\>=0$. Taking norms, yields
 \be
  \label{zw2}
 \Vert z - \zp \Vert \leq \Vert z \Vert \Vert \Cop \Vert \frac{\Vert u - \u \Vert}{ \Vert \Cop \u \Vert} \leq  \bar C \e  \Vert z \Vert ,
 \ee
 with
 \be
 \label{barC}
 \bar C = \l  \Vert \Cop \Vert.    
 \ee

% \dwr{We'll shorten this later. I just wanted make sure on what the constants
% depend. There is perhaps a slicker argument: $\frac{\<\u,\Cop u\>}{\Vert \u \Vert \Vert \Cop u\Vert }$ is the cosine of the angle between $\u$ and $\Cop u$,
% Rotating $\u/\Vert  \u \Vert $ to $\Cop u/\Vert \Cop u\Vert $ along a geodesic on the sphere would rotate the space $\Usp_\circ^\perp= {\rm span}\{\Cop  \u\}^\perp$ into ${\rm span}\{\Cop u\}^\perp$  by that angle. Hence $z$ would land in a point $\zp \in {\rm span}\{\Cop u\}^\perp$
% (of actually the same norm). $z$ and $\zp$ have the same angle, hence the
%% relative distance is bounded by that of $\Cop(\u- u)$, giving
 %even a better estimate.}
%%%%%%%%%%%%%%%%%%%%%%%%
\newcommand{\tzp}{\tilde{z}_\circ^\perp}
%%%%%%%%%%%%%%%%%%%%%%%%%%%%%%%%%%%%%%%%%
 %\begin{align*}
% Since $\Vert \Cop \bar u - \Cop u \Vert\leq \e\Vert \Cop \Vert \Vert u\Vert $, we have for $w_\perp :=\Pop_\perp w$ one has  $\Vert w - u_0^\perp\Vert \lesssim \e$.
Note next that, using \eqref{M}, the one-parameter family of linear operators
 $\Mop_{\nu}$ can be bounded uniformly for 
 $\nu \in B\Big(\l,\frac{\theta}{4\Vert \Cop\Vert }\Big)$ as 
 %\omnote{I do not see how  the bound for $M_\nu$ was derived (see \eqref{Mbar}). I get as a bound: $1+(\l+\theta/4||\Cop||)||\Cop||$.}
\be\label{MbarNu}
M_\nu =\Vert \Mop_{\nu}\Vert 
\le \max_{|\nu-\l|
\le \frac{\theta}{4\Vert \Cop\Vert }}\Big\{1+ |\nu|\Vert \Cop\Vert \Big\}
\le 1+\l \Vert \Cop\Vert +\frac{\theta}4 =: \bar M.
%\om{2\max\{1, \theta/4\}}, %\quad \Vert \bar b\Vert = \Vert \Cop u\Vert \le 2(1+\e_0\l\Vert \Cop\Vert ),
\ee
%where we have used \eqref{gauge} in the last step. 
\wo{We abbreviate $\tzp := \Pop_{\u} \tilde z$
Under the given assumption on $\nu$ we can invoke now Lemma \ref{lem:pert-lambda} and \eqref{zw2} to conclude that for any $z, \tilde z\in {\rm span}\{\Cop u\}^\perp$, 
\begin{align}
\left< \Mop_{\nu} z, \tilde z\right>
&=
\left< \Mop_{\nu} (z-\zp), (\tilde z-\tzp)\right>
+ \left< \Mop_{\nu} \zp, \tzp\right>
+ \left< \Mop_{\nu} \zp, (\tilde z-\tzp) \right>
+ \left< \Mop_{\nu} ( z-\zp) , \tzp\right>\nonumber \\
&\geq \frac{\theta}{2} \Vert \zp \Vert\Vert\tzp\Vert - M_\nu \Vert z-\zp\Vert\Vert 
\tilde z-\tzp\Vert -  M_\nu (\Vert \tzp\Vert \Vert z-\zp\Vert + \Vert \zp\Vert \Vert 
\tilde z-\tzp\Vert)
%\quad \om{[\frac \theta 4 \rightarrow \frac \theta 2 ? \text{ see Lemma \ref{lem:pert-lambda} }]}
\nonumber\\
&\ge \frac{\theta}{2} \Vert \zp \Vert\Vert \tzp\Vert - \e^2\bar C^2M_\nu\Vert z\Vert
\Vert\tzp\Vert 
- 2M_\nu\Vert \zp\Vert \Vert z\Vert \bar C\e.
\end{align}
Since by \eqref{zw2}, $(1+\e\bar C)\Vert z\Vert \ge \Vert \zp\Vert \ge (1-\e\bar C)\Vert z\Vert $ we obtain
\be
\label{allw}
\left< \Mop_{\nu} z, \tilde z\right>\ge\Big\{ \frac{(1-\e\bar C)^2\theta}{2}
- \e \bar C M_\nu\big(\e\bar C + 2(1+\e\bar C)\big)\Big\}\Vert z\Vert \Vert\tilde z\Vert.
%\ge \frac{\theta}6\Vert z\Vert ^2,
\ee
Therefore, there exists an $\e_1=\e_1(\Vert \Cop\Vert ,\l,\theta)$ such that
\be
\label{allw2}
\left< \Mop_{\nu} z, \tilde z\right>
\ge \frac{\theta}4\Vert z\Vert \Vert\tilde z\Vert, \quad \forall \e \leq \e_1.
\ee
}
In fact, elementary calculations show that \eqref{allw2} is valid for
\be
\label{eps11}
\e_1:= \frac{1}{8\bar C}\min \Big\{3,\frac{8\theta}{25 \bar M}\Big\},
\ee
 where $\bar M$ is given in \eqref{Mbar}.

Regarding the dependencies of $\e_1$, we have used \eqref{barC}, \eqref{eps0},
%that $\bar C$ depends on $\Vert \Cop\Vert , \l$ 
and that $M_\nu=\Vert \Mop_{\nu}\Vert $ depends on $\Vert \Cop\Vert ,\l,\theta$, see \eqref{Mbar}. This concludes the proof.
\end{proof}
\medskip

We are now prepared to complete the proof of Theorem \ref{thm:main}: \\[-1mm]
\noindent
\begin{proof}[Proof of Theorem \ref{thm:main}]
Let $\e \leq \bar \e \coloneqq \min(\e_0, \e_1)$ where $\e_0$ and $\e_1$ have been specified in (the proofs of) Lemmata \ref{lem:uubar} and \ref{lem:pert-Cu},
see \eqref{eps0}, \eqref{eps11}.
 Hence, the above perturbation results remain uniformly valid for all pairs
\be
\label{ass}
(u, \nu)\in \cN = B(\u,\ett)\times B\left(\l, \frac{\theta}{4 \| \Cop\|}\right)\quad \mbox{with}\quad \ett = \Vert \u \Vert \bar \e = 2\l \bar \e.
\ee
Specifically, we can take in view  of \eqref{eps0} and \eqref{eps11}
\be
\label{taudef}
\tau :=2 \l\, \min\Big\{\frac{1}{16\bar C}, \frac{\theta}{25 \bar M \bar C}\Big\}
= \min\Big\{\frac{1}{8\Vert \Cop\Vert }, \frac{2\theta}{25(1+\l\Vert \Cop\Vert +\theta/4)\Vert \Cop\Vert }\Big\}.
\ee

We fix now $(\bar u, \bar \nu) \in \cN$. The operator equation: for a given $(g, \zeta)\in \Usp\times \bR$   find the solution $(u, \nu) \in \Usp \times \bR$ of
$$
DR(\bar u, \bar \nu)(u, \nu) = (g, \zeta)
$$
 can be written as the saddle-point problem
\begin{equation}
\label{saddlebar}
\begin{array}{lcll}
\bar a(u, z)  + \bar b(z, \nu) &=& \left< g, z\right> &\quad \forall z \in \Usp\\[2mm]
\bar b(u, \alpha) &=& \alpha \zeta &\quad \forall \alpha \in \bR,
\end{array}
\end{equation}
where the $(\bar u, \bar \nu)$-dependent bilinear forms read
\be
\label{bilinbar}
\bar a(u,z):= \<\Mop_{\bar \nu} u,z\>,
\quad
\bar b(u,\alpha):=-\alpha\<\Cop \bar u,u\>.
\ee
Note that we have proceeded in the same spirit as earlier in the proof of Lemma \ref{lem:saddle} (see \eqref{bilinla}). Here we consider the general case where $(\bar u, \bar \nu)\in \cN$  recovering \eqref{bilinla} for $(\bar u, \bar \nu) = (\u, \l)$.
%\mnote{the notation is not quite satisfactory at several places. $\mu$ plays
%again the role of $\lg$ and $w$ should better belong to $\Wsp$ although this context
%is somewhat remote by now. \om{Significant notation changes. Hopefully better now...}}

To prove well-posedness of \eqref{saddlebar} we proceed exactly as in Lemma
\ref{lem:saddle} and show that \eqref{saddlebar}
satisfies the  LBB conditions.
%\wo{conditions in the Babu\v{s}ka--Ne\v{c}as Theorem (in brief (BN) conditions)}. 
 To that end, it will be convenient to derive bounds for $\| \Cop \bar u\|$ first.
Since we have chosen $\bar u \in B(\u, \tau)=B(\u, \bar \e \| \u\|)$, we can apply Lemma \ref{lem:uubar} for the case $u=\bar u$ and $\e=\bar \e \leq \e_0 $ to deduce that
$$
\| \Cop \bar u \| \| \bar u \| \geq \left< \Cop \bar u, \bar u\right> \geq \frac{\| \bar u \|^2}{2\l}
$$
Therefore, by the reverse triangle inequality,
$$
\| \Cop \bar u \|
\geq \frac{\| \bar u \|}{2\l}
\geq \frac{1}{2\l} (\| \u \| - \| \bar u - \u \| )
\geq 1-\bar \e
$$
Similarly, by the direct triangle inequality
$$
\| \Cop \bar u \| \leq \|\Cop \u \| + \| \Cop (\bar u - \u) \| \leq 2(1+\bar \e \l \| \Cop\|).
$$
To summarize, we have just proven that
\begin{equation}
1-\bar \e \leq \| \Cop \bar u \| \leq 2(1+\bar \e \l \| \Cop\|).
\label{eq:bounds-C-bar-u}
\end{equation}

Let us now turn to proving that the LBB conditions are satisfied. We first note that $\bar a$ and $\bar b$ are bounded bilinear forms since
\be
\Vert \bar a\Vert := \sup_{u\in \Usp } \sup_{z\in \Usp } \frac{\bar a(u, z)}{ \Vert u \Vert \Vert z \Vert }\le M_{\bar \nu},\quad \Vert \bar b\Vert := \sup_{\nu\in \bR}  \sup_{z\in \Usp} \frac{| b(z, \nu) |}{\Vert z\Vert | \nu |}\le \Vert \Cop \bar u\Vert  \leq 2(1+\bar \e \l \| \Cop\|).
\ee
Next, denoting
$$
Z
=  \{ w\in \Usp \; : \; \bar b(w,\alpha)=0,\, \forall \alpha \in \bR \}
 = \{ w\in \Usp \; : \; \left< \Cop \bar u, w\right>=0 \} = \vspan\{ \Cop \bar u \}^\perp,
$$
applying Lemma \ref{lem:pert-Cu} to $\Mop_{\bar \nu}$ yields \wo{inf-sup stability} of $\Mop_{\bar \nu}$ on $Z$,
$$
\wo{\left<\Mop_{\bar \nu} w, w'\right> \geq \frac{\theta}4 \Vert w \Vert\Vert w'\Vert,\quad \forall w, w' \in \vspan\{ \Cop \bar u \}^\perp.}
$$
%\omnote{Corrected/improved end of this proof, which was written in a confusing way.}
In addition, using the lower bound from \eqref{eq:bounds-C-bar-u},
$$
\inf_{\alpha\in \bR}  \sup_{z\in \Usp} \frac{| \bar b(z, \alpha) |}{\Vert z\Vert | \alpha |}
=  \sup_{z\in \Usp}  \frac{ \left< \Cop \bar u, z \right>}{\Vert z\Vert }
= \Vert \Cop \bar u \Vert \geq 1- \bar \e.
$$
% \mnote{the corrected $\theta$-factor has been used.}
This proves that the LBB conditions are satisfied. Thus, problem \eqref{saddlebar} is well-posed, and invoking standard stability estimates e.g. from \cite{BrezziFortin},  we have
\begin{align*}
\| u \| &\leq \frac{4}{\theta} \| g\| + \frac{1}{1-\bar \e}\left( 1 + 4 \frac {M_{\bar \nu}}{\theta} \right) |\xi | \\
&\leq \underbrace{\left( \frac{16}{\theta^2} + \frac{1}{(1-\bar \e)^2}\left( 1 + 4 \frac {M_{\bar \nu}}{\theta} \right)^2  \right)^{1/2} }_{\coloneqq \beta_1} \|(g, \xi) \|  , \\
|\nu| &\leq  \frac{1}{1-\bar \e}\left( 1 + 4 \frac {M_{\bar \nu}}{\theta} \right) \|g\| + \frac{M_{\bar \nu}}{(1-\bar \e)^2}\left( 1 + 4 \frac {M_{\bar \nu}}{\theta} \right) |\nu| \\
&\leq \left(
  \frac{1}{(1-\bar \e)^2}\left( 1 + 4 \frac {M_{\bar \nu}}{\theta} \right)^2
  + \frac{M^2_{\bar \nu}}{(1-\bar \e)^4}\left( 1 + 4 \frac {M_{\bar \nu}}{\theta} \right)^2
\right)^{1/2} \| (g, \xi)\| \\
&=
\underbrace{\frac{1}{(1-\bar \e)}\left( 1 + 4 \frac {M_{\bar \nu}}{\theta} \right)
\left( 1+\frac{M_{\bar \nu}}{(1-\bar \e)^2}\right)^{1/2}}_{\coloneqq \beta_2}
\|(g, \xi)\|.
\end{align*}
Thus, we
conclude that
$$
\| DR(\bar u, \bar \nu)^{-1} \|_{\cL(\Uspext, \Uspext)} \leq \bar \beta = \sqrt{\beta^2_1 + \beta^2_2},\quad \forall (\bar u, \bar \nu) \in \cN,
$$
where \eqref{barbeta} follows from substituting $\e = \frac{\tau}{2\l}$.
Thanks to the above inequalities, and recalling from \eqref{MbarNu} that
$$
M_{\bar \nu} \leq \bar M \coloneqq 1+\l \Vert \Cop\Vert +\frac{\theta}4,\quad \forall \bar \nu \in B\left(\l, \frac{\theta}{4\|\Cop\|}\right),
$$
the assertion follows from \eqref{taudef}.
\end{proof}

\subsection{Convergence of the Newton Scheme}\label{ssec:conv}
%%%%%%%%%%%%%%%%%%%%%%

In this section, we prove that the Newton scheme \eqref{Newton} is locally quadratically convergent to $(\u, \l)$. Proving this will require leveraging  Lipschitz continuity of $DR$, a property which we record in the next lemma.

\begin{lemma}
\label{lem:Lip}
The derivative $DR$ is Lipschitz continuous, and
\be
\label{Lip}
\Vert DR(u_1,\nu_1) - DR(u_2, \nu_2) \Vert_{\cL(\Usp\times \R,\Usp\times \R)} \leq \gamma \Vert (u_1,\nu_1) - (u_2, \nu_2) \Vert_{\Uspext},\quad \forall\,\, (u_1,\nu_1),   (u_2, \nu_2) \in \Usp \times \R,
\ee
where the Lipschitz constant $\gamma$ can be bounded by $\gamma \le \sqrt{2}\Vert \Cop\Vert $.
\end{lemma}
\begin{proof}
Since $DR(\cdot,\cdot)$ is linear in both arguments, it is Lipschitz continuous if and only if it is bounded. A repeated combination of the triangle inequality and the Cauchy--Schwarz inequality yields the bound.  
\end{proof}

We are now in position to prove convergence of the Newton scheme. \corrw{The main arguments in the development are actually classical, and can be found in different references (see, e.g., \cite{OR2000, DR2008}).} It will be convenient to denote the ball in $\Uspext$ centered at $(\u,\l)$ of radius $\omega$ as
$$
K_\omega( \u,\l )
\coloneqq
B((\u,\l), \omega)
=\{ (u,\mu) \in \Uspext \; : \; \Vert (\u,\l) - (u,\mu) \Vert < \omega \}.
$$

\begin{theorem}[Convergence of the Newton scheme]
\label{thm:conv}
Let $\cN$ be the neighborhood defined in Theorem \ref{thm:main}. Let $\omega>0$ be sufficiently small such that
\be
\label{omega}
K_\omega( \u,\l ) \subset \cN,\quad \text{and} \quad \omega \leq \frac{2}{\bar\beta\gamma},
\ee
where $\bar \beta$ is defined in Theorem \eqref{thm:main}, and $\gamma$ is the Lipschitz constant from Lemma \ref{lem:Lip}.

If the initial guess $(u_0,\lambda_0) \in K_\omega( \u,\l)$, then the sequence of Newton iterates $((u_n, \lambda_n))_{n=0}^\infty$ stays in $K_\omega( \u,\l)$ and has quadratic convergence,
\begin{align}
\label{quadratic-1}
\Vert (u_{n+1}, \lambda_{n+1}) - (\u, \l) \Vert
& \leq \frac{\bar\beta \gamma}{2} \Vert (u_n,\lambda_n) -  (\u, \l) \Vert^2.
\end{align}
\end{theorem}

\begin{remark}
\label{rem:omega}
Notice that, in view of Remark \ref{rem:tau}, the condition on $\omega$ in \eqref{omega} agrees in essence with the condition on the neighborhood $\cN$
and therefore introduces no severe additional constraints.
\end{remark}

\begin{proof}
We sketch it for the convenience of
the reader since the proof consists of ``lifting'' standard arguments to an operator level. By induction, for a given $n\geq0$, we assume that $(u_n, \lambda_n) \in K_\omega(\u,\l)$ (the case $n=0$ being true by assumption). We show that there is a contraction in the error at the next step, and that $(u_{n+1}, \lambda_{n+1}) \in K_\omega(\u,\l)$. To that end,  since $R(\u,\l)=0$, we have
\begin{align*}
(u_{n+1}, \lambda_{n+1}) - (\u, \l)
&= (u_n, \lambda_n) - (\u, \l) - DR(u_n, \lambda_n)^{-1} R(u_n, \lambda_n) \\
&= (u_n, \lambda_n) - (\u, \l) - DR{(u_n, \lambda_n)}^{-1} \left( R(u_n, \lambda_n) - R(\u, \l) \right) \\
&= DR{(u_n, \lambda_n)}^{-1} \left( R(\u, \l) -  R(u_n, \lambda_n)
- DR{(u_n, \lambda_n)}( (\u, \l) -  (u_n, \lambda_n)) \right).
\end{align*}
Note that the derivative $DR{(u_n, \lambda_n)}^{-1}$ is well defined by Theorem \ref{thm:main} since $(u_n,\lambda_n) \in K_\omega(\u,\l) \subset \cN$. Taking norms gives
\begin{align}
\label{quadconv}
\Vert (u_{n+1}, \lambda_{n+1}) - (\u, \l) \Vert
& \leq
\bar\beta \Vert R(\u, \l) -  R(u_n, \lambda_n)
- DR{(u_n, \lambda_n)}( (\u, \l) -  (u_n, \lambda_n)) \Vert.
% &\leq \frac{\bar\beta \gamma}{2} \Vert (\u, \l) -  (u_n, \lambda_n) \Vert^2,
\end{align}
To derive a bound for the right-hand side of \eqref{quadconv}, we introduce the segment
$$
\ell_n(t) \coloneqq  (u_n, \lambda_n)+t((\u, \l)-(u_n, \lambda_n)),\quad \forall t\in [0,1]
$$
joining $(u_n, \lambda_n)$ and $(\u, \l)$, and we consider 
$$
\phi(t) \coloneqq R\circ \ell_n(t)
$$
which is the restriction of $R$ to $\ell_n(t)$. Note that since
$$
\phi'(t)= DR\left((u_n, \lambda_n)+t((\u, \l)-(u_n, \lambda_n))\right) \left( (\u, \l) - (u_n, \lambda_n)\right),
$$
we have
\begin{align*}
&\Vert R(\u, \l) -  R(u_n, \lambda_n)
- DR{(u_n, \lambda_n)}( (\u, \l) -  (u_n, \lambda_n)) \Vert  
= \Vert \phi(1)-\phi(0) - \phi'(0) \Vert \\
&\quad \leq \int_0^1 \Vert \phi'(t)-\phi'(0) \Vert \dt  
\leq \frac{\gamma}{2} \Vert (\u, \l) -  (u_n, \lambda_n) \Vert^2.
\end{align*}
Inserting this last inequality into \eqref{quadconv} yields \eqref{quadratic-1}.

We finally prove that $(u_{n+1},\lambda_{n+1}) \in K_\omega((\u, \l))$. This follows from the   inequalities
\begin{align*}
\Vert (u_{n+1}, \lambda_{n+1}) - (\u, \l) \Vert
& \leq \frac{\bar\beta \gamma}{2} \Vert (\u, \l) -  (u_n, \lambda_n) \Vert \Vert (\u, \l) -  (u_n, \lambda_n) \Vert  
\leq \frac{\bar\beta \gamma \omega}{2} \Vert (\u, \l) -  (u_n, \lambda_n) \Vert\\  
&  \leq \Vert (\u, \l) -  (u_n, \lambda_n) \Vert
< \omega .
\end{align*}
\end{proof}

%%%%%%%%%%%%%%%%%%%%%%%%%%%%%%%%%%
\section{Perturbed Newton scheme}
\label{sec:Newton-scheme-perturbed}
\subsection{Goals and notation}
The idealized Newton scheme \eqref{Newton} is formulated on a continuous level (see (CC1) in Section \ref{ssec:novelty}), i.e., the Newton updates \eqref{newton-update} require exact inversions of the system \eqref{newton-update}. According to (CC2), (CC3) in
Section \ref{ssec:novelty}, we aim at a practical realization based on approximating these 
updates within judicious controllable tolerances.
% provides numerical approximations to these exact updates whose accuray should be assessible. 
%Therefore we require a solver that provides numerical approximations to equations involving the operator $DR$. 
Specifically, for any $(\bar u, \bar \nu)\in \Uspext$, any target accuracy $\eta>0$, and any right-hand side $f\in \Wsp'\times \bR$, numerical solver has to deliver an approximation $(w_\eta, \alpha_\eta)\in \Uspext$ to the operator equation
$$
DR(\bar u, \bar \nu) (w, \alpha) = f.
$$
We denote a numerical solver that realizes accuracy $\eta$ with output $(w_\eta, \alpha_\eta)$ by
\be
\label{eq:approx-inv-DR}
[ DR{(\bar u, \bar \nu )}^{-1}, f ; \eta ]\to (w_\eta, \alpha_\eta),
\ee
meaning that
$$
\Vert (w_\eta, \alpha_\eta) - DR{(\bar u, \bar \nu)}^{-1} f\Vert  \leq \eta.
$$
  
We postpone discussing possible realizations of $[ DR{(\bar u, \bar \nu )}^{-1}, f ; \eta ]$. Assuming for the moment to have such a scheme at hand, we discuss   in  Section \ref{sec:conv-pert-newton} 
first for which (dynamically adjusted) tolerances $\eta=\eta_n$ a corresponding
 inexact Newton method  converges at a linear or even quadratic rate, see (CC2), Section
 \ref{ssec:novelty}. 

Finally, we discuss in Section \ref{ssec:prac} strategies for numerically realizing such accuracy controlled inexact Newton updates.

\subsection{Scheme and Convergence Analysis}
\label{sec:conv-pert-newton}
Recall from \eqref{Newton} and \eqref{newton-update} that the exact Newton iteration produces a sequence $(u_n, \lambda_n)_{n\geq 0}$ defined by
\be 
u_{n+1} = u_n + \delta_n^{(u)},\quad
\lambda_{n+1} = \lambda_n + \delta_n^{(\lambda)},\quad n\in \N_0,
\end{equation}
with
\be
\label{N-update}
(\delta_n^{(u)}, \delta_n^{(\lambda)})
= - DR(u_n,\lambda_n)^{-1} R(u_n, \lambda_n).
\ee
The inexact iteration gives rise to a sequence $(\bar u_n, \bar \lambda_n)_{n\geq 0}$ defined by
\begin{align}
\bar u_{n+1} &= \bar u_n + \bar h_n^{(u)} , \quad 
\bar \lambda_{n+1} = \bar \lambda_n + \bar h_n^{(\lambda)} ,\quad n\in \N_0,
\end{align}
where
\be
\label{eq:bar-h}
(\bar h_n^{(u)}, \bar h_n^{(\lambda)})
= [ DR(\bar u_n, \bar \lambda_n)^{-1}, -R(\bar u_n, \bar \lambda_n); \eta_n ]
\ee
for some accuracy tolerance $\eta_n$. The computable updates $(\bar h_n^{(u)}, \bar h_n^{(\lambda)})$  approximate the solution to the operator equation
$$
DR(\bar u_n, \bar \lambda_n)(h_n^{(u)}, h_n^{(\lambda)})
=
-  R(\bar u_n, \bar \lambda_n).
$$
By definition of $[ DR(\bar u_n, \bar \lambda_n)^{-1}, -R(\bar u_n, \bar \lambda_n); \eta_n ]$, we thus have
\be
\label{eta-bar-h}
\Vert  (\bar h_n^{(u)}, \bar h_n^{(\lambda)}) - (h_n^{(u)}, h_n^{(\lambda)}) \Vert \leq \eta_n.
\ee
The next theorem gives an upper bound for the value of $\eta_n$ that preserves the same quadratic convergence rate as in the exact Newton scheme. For its proof, it will be convenient to introduce notation for the convergence error of the exact and perturbed schemes
\begin{align}
  e_n \coloneqq \Vert  (u_n, \lambda_n) - (\u, \l)\Vert , \qquad 
  \bar e_n \coloneqq \Vert  (\bar u_n, \bar \lambda_n) - (\u, \l)\Vert .
\end{align}

\begin{theorem}
\label{thm:convergence}
  \corr{For the neighborhood $\cN = B(\u,\ett)\times B\Big(\l, \frac{\theta}{4 || \Cop||}\Big)$, $\ett = \Vert \u \Vert \bar \e$, defined in \eqref{ass}}, let $\omega>0$ be sufficiently small such that 
  \be
  \label{omega3}
  K_\omega( \u,\l ) \subset \cN,\quad \text{and} \quad \omega \leq \frac{1}{3\bar\beta\gamma}.
  \ee
  Then, if $(\bar u_0, \bar \lambda_0) \in K_\omega( \u,\l )$,
  \begin{equation}
  \bar e_{n+1} \leq \eta_n + \frac{\bar \beta \gamma}{2}(\bar e_n^2+2e_n^2).
  \label{eq:conv-pert-newton}
  \end{equation}
  In addition, if 
  %\omnote{Cannot get quadratic convergence with $\eta_n \lesssim \bar e_n$ as in the previous notes. I need $\eta_n \lesssim \bar e_n^2$}
  \be
  \label{eta-n-rule}
  \eta_n \leq \min\left(\frac \omega 2, \frac{\bar \beta \gamma}{2} \bar e_n^2\right),
  \ee
 all iterates remain in $K_\omega( \u,\l )$, i.e., 
  \be
  \label{inK}
  (u_n,\lambda_n) \in   K_\omega( \u,\l ),\quad \forall n \geq 0.
  \ee
Finally,
  we retain  quadratic convergence as in the unperturbed case, namely
  as soon as $n_0$ satisfies
  \be
\label{n0}
 {\bar\beta\gamma}\bar e_{n_0}^2\le  {\omega},
\ee
one has
  %and we have
  \begin{equation}
  \bar e_{n+1} \le 4\bar\beta\gamma \bar e_n^2,\quad  n\ge n_0. 
  %\left( \frac{\bar \beta \gamma}{2} \right)^{2n+1} \bar e^2_0,
  \label{eq:conv-pert-newton-full-induction}
  \end{equation}
 \end{theorem}

\begin{proof}
For any $n\geq 0$, we have
\begin{align}
\bar e_{n+1} =
\Vert  (\bar u_{n+1}, \bar \lambda_{n+1}) - (\u, \l)\Vert 
&\leq \Vert  (\bar u_{n+1}, \bar \lambda_{n+1}) - (u_{n+1}, \lambda_{n+1})\Vert 
+ \Vert  (u_{n+1}, \lambda_{n+1}) - (\u, \l)\Vert  \\
&\leq \Vert  (\bar u_{n+1}, \bar \lambda_{n+1}) - (u_{n+1}, \lambda_{n+1})\Vert  + \frac{\bar \beta \gamma}{2} e_n^2,
\label{eq:ineq-aux}
\end{align}
where we have \corr{used \eqref{quadratic-1}} in the last inequality. We next build a bound for the first term of the inequality. We have
\begin{align*}
  &\Vert  (\bar u_{n+1}, \bar \lambda_{n+1}) - (u_{n+1}, \lambda_{n+1})\Vert  
  =
  \Vert  (\bar u_{n}, \bar \lambda_{n}) + (\bar h_n^{(u)}, \bar h_n^{(\lambda)})- (u_{n}, \lambda_{n}) - (\delta_n^{(u)}, \delta_n^{(\lambda)})\Vert  \nonumber\\
  &\leq \Vert  (\bar h_n^{(u)}, \bar h_n^{(\lambda)}) - (h_n^{(u)}, h_n^{(\lambda)}) \Vert 
  +
  \Vert  (\bar u_{n}, \bar \lambda_{n}) - DR(\bar u_n, \bar \lambda_n)^{-1} R(\bar u_n, \bar \lambda_n) - (u_n, \lambda_n) + DR(u_n, \lambda_n)^{-1}R(u_n, \lambda_n) \Vert  
  \nonumber\\
  &\leq \eta_n + \Vert  (\bar u_{n}, \bar \lambda_{n}) -(\u, \l) - DR(\bar u_n, \bar \lambda_n)^{-1} \left( R(\bar u_n, \bar \lambda_n) -R(\u, \l) \right) \Vert  \\
  & \qquad\;\, + \Vert  (u_n, \lambda_n) -(\u, \l) - DR(u_n, \lambda_n)^{-1}\left( R(u_n, \lambda_n) - R(\u,\l) \right) \Vert ,
\end{align*}
where we have added and subtracted $(\u, \l)$ and used that $R(\u,\l)=0$ to derive the last inequality. Now, by the same arguments as in the proof of Theorem \ref{thm:conv}, we conclude that
$$
\Vert  (\bar u_{n}, \bar \lambda_{n}) -(\u, \l) - DR(\bar u_n, \bar \lambda_n)^{-1} \left( R(\bar u_n, \bar \lambda_n) -R(\u, \l) \right) \Vert 
\leq \frac{\bar \beta \gamma}{2} \Vert  (\bar u_{n}, \bar \lambda_{n}) -(\u, \l) \Vert ^2
= \frac{\bar \beta \gamma}{2} \bar e_n^2,
$$
and
$$
\Vert  (u_n, \lambda_n) -(\u, \l) - DR(u_n, \lambda_n)^{-1}\left( R(u_n, \lambda_n) - R(\u,\l) \right) \Vert  \leq \frac{\bar \beta \gamma}{2} e_n^2.
$$
Hence, we obtain 
\be
\label{sofar}
\Vert  (\bar u_{n+1}, \bar \lambda_{n+1}) - (u_{n+1}, \lambda_{n+1})\Vert  \leq \eta_n + \frac{\bar \beta \gamma}{2} ( \bar e_n^2 + e_n^2 ).
\ee
Inserting this inequality into \eqref{eq:ineq-aux} yields \eqref{eq:conv-pert-newton}.
Since  by assumption, $(\bar u_0,\bar\lambda_0)=(u_0,\lambda_0)\in 
K_\omega((\u,\l))$, we infer inductively from  \eqref{omega3} and \eqref{eta-n-rule} that
\be
\label{en}
\bar e_{n+1} \le \eta_n+ \frac{\bar\beta\gamma}{2}\big(\bar e_n^2 + 2e_n^2\big)
= \eta_n + \frac{3\bar\beta\gamma\omega}{2}\omega \le \eta_n 
+  \frac{\omega}2 <\omega ,
\ee
which is \eqref{inK}.

Concerning inequality \eqref{eq:conv-pert-newton-full-induction}, substituting
\eqref{eta-n-rule} into \eqref{sofar}, yields under the provision \eqref{n0}
the asserted quadratic convergence \eqref{eq:conv-pert-newton-full-induction}.
\end{proof}

Realizing an update accuracy of the order $\eta_n\sim \bar e_n^2$ may actually
entail a bit of a challenge for a numerical realization. It is important to note though that perhaps of
equal practical importance is the fact that a less demanding tolerance $\eta_n$
would still give rise to a robust first-order scheme with  control on
the quantitative error reduction rate as detailed next. 
\begin{remark}
\label{rem:linear}
 Suppose that for any given $0< \zeta \ll 1$, $n_1$ is large enough to ensure
 that $3\bar\beta\gamma \bar e_n\le \zeta/2$, $n\ge n_1$, then whenever
\be
\label{relaxeta}
\eta_n = \min\Big\{\frac{\omega}2,\frac{\zeta}2 \bar e_n\Big\},
\ee
 \eqref{eq:conv-pert-newton}
says that  for $n\ge n_1$  
$
\bar e_{n+1} \le \Big(\frac{\zeta}2 + \frac{3\bar\beta\gamma}2 \bar e_n\Big)\bar e_n,
$
which shows that
%of 
\be
\label{linred}
\bar e_{n+1}\le \zeta \bar e_n.
\ee
Hence an update tolerance proportional to the current accuracy, with sufficiently
small proportionality factor, gives rise to a linear error reduction with
a reduction factor as small as one wishes.
\end{remark}

The significance of these observations lies in the following familiar consequence.
Suppose one has $\bar e_{n+1}\le c \bar e_n^p$ for some $0<c<\infty$ and $p\ge 1$.
Then, triangle inequality yields
\be
\label{triangle}
(1- c \bar e_n^{p-1}) \bar e_n \le \Vert \bar u_{n+1}- \bar u_n\Vert \le 
(1+ c \bar e_n^{p-1}) \bar e_n.
\ee
Thus, for quadratic convergence $p=2$ and any fixed $c<\infty$, and for linear
convergence $p=1$ but $c\le 1/2$ say, one has (for $n$ sufficiently large)
\be
\label{n-n+1}
\frac 12 \bar e_n \le \Vert \bar u_{n+1}- \bar u_n\Vert \le \frac 32 \bar e_n.
\ee
which means that the computable {\em a posteriori} quantity $\Vert \bar u_{n+1}- \bar u_n\Vert $ provides a tight error bound for the current approximation $\bar u_n$.

\subsection{A posteriori estimation of $\bar e_n$}\label{sec:apost}
The remarks at the end of the previous section show that a current error $\bar e_n$
can be assessed through a computable a posteriori quantity provided that
either \eqref{eq:conv-pert-newton-full-induction} or  \eqref{linred} (for sufficiently small $\zeta$) hold. This in turn, hinges on two pillars.
First, one needs to estimate $\bar e_n$ in order to determine a suitable $\eta_n$
(we cannot use \eqref{triangle} at this stage because $\bar u_{n+1}$ has yet to be computed).
Second, once $\bar e_n$ and hence $\eta_n$ is known, one needs to solve
the Newton update problem \eqref{eq:bar-h} with the tolerance $\eta_n$.
We defer this latter issue to the next section and discuss here the first one
through deriving an {\em a posteriori} error bound that does not require the subsequent
approximation. The rationale is, in principle, simple. We have already shown
that the linearization of $R(u,\nu)=0$ is well-posed in a certain neighborhood of
$(\u,\l)$ which can be used as follows.
%This is proven through the following norm equivalence.
%
\begin{proposition}
\label{prop:apost}
\corr{There exists a constant $0< \overline{C}<\infty$, depending only on the problem parameters $\Vert \Cop\Vert,\,\bar M$} such that 
%for any $(\bar u,\lambda u)\in K_\omega(u,\lambda)$
\be
\label{tight}
\corr{\bar\beta} \Vert (u,\nu)-(\u,\l)\Vert \le \Vert R(u,\nu)\Vert     \le \overline{C}
\Vert (u,\nu)-(\u,\l)\Vert ,\quad \forall\,\,(u,\nu)\in K_\omega(\u,\l).
\ee
\end{proposition}
 \begin{proof}
Arguing as before, for any $(u, \nu)\in K_\omega((\u,\l))$ we introduce the
segment
$$
\ell(t) = (\u,\l) + t ((u, \nu)- (\u,\l) ), \quad \forall t \in [0,1]
$$
and the restriction of $R$ along that segment
$$
\phi (t) = R\circ \ell(t),\quad \forall t\in [0,1].
$$
It follows that
\begin{align}
\label{int}
\Vert R(u, \nu)\Vert 
&= \Vert R(u, \nu)- R(\u,\l)\Vert 
= \Vert \phi(1) - \phi(0) \Vert
= \Vert \int_0^1 \phi'(t)\dt \Vert
\leq \max_{t\in [0,1]} |\phi'(t)|\nonumber\\
&\le \overline{C} \Vert (u, \nu)-(\u,\l) \Vert,
\end{align}
where
$$
\overline{C}\coloneqq\sup_{t\in [0,1]} \Vert DR( (\u,\l)+t(  ( u, \nu) - (\u,\l) ) \Vert
$$
depends on $\bar M, \Vert \Cop\Vert $.

To obtain also a lower bound, we apply the mean value value theorem to $\phi$ over $[0,1]$ to deduce that there exists $\bar t \in (0, 1)$ such that
$$
R(u,\nu) - R(\u,\l)
= \phi(1)-\phi(0)
= \phi'(\bar t)
= DR( \ell(\bar t) )((u, \nu)-(\u,\l)).
$$
Since $\ell(\bar t) \in K_\omega((\u,\l))$, Theorem \ref{thm:main} guarantees injectivity of $DR( \ell(\bar t) )$, from which it follows that
$$
(u, \nu)-(\u,\l) = DR^{-1}( \ell(\bar t)  ) R(u,\nu).
$$
Taking norms and using \eqref{DRbound} yields the lower bound
$$
\Vert (u, \nu)-(\u,\l) \leq \bar \beta^{-1} \Vert R(u,\nu) \Vert.
$$
As a side remark, note that one can actually derive that the exact value of $\bar t$ is $1/2$ by exactly integrating $\phi(1)-\phi(0)=\int_0^1 \phi'(t)\dt$ with the expression for $\phi'(t)$ which reads
$$
\phi'(t)
=DR(\ell(t))((u,\nu)-(\u,\l))
=
\begin{pmatrix}
  (I- (\l + t(\nu-\l))\Cop (u-\u) - \nu\Cop(\u+t(u-\u))\\
  - \<\Cop(\u+t(u-\u)),u-\u)\>
  \end{pmatrix}.
$$
\end{proof}

Proposition \ref{prop:apost} implies that the accuracy provided by each perturbed step can be assessed through $\Vert R(u_n, \lambda_n)\Vert$, and this quantity can be used to steer the target accuracy of subsequent calculations. Hence, at this stage, things boil down to computing accuracy controlled Newton updates. We sketch the essence of a corresponding scheme in the next section, leaving details to forthcoming work with numerical realizations.

\subsection{Towards a Numerical Realization}
\label{ssec:prac}
Recall that in Section \ref{sec:conv-pert-newton} we have identified appropriate
accuracy tolerances and in Section \ref{sec:apost} we have shown how to assess
the accuracy of a given approximation. The objective of this section is to provide the analytical basis for numerically realizing approximate Newton updates 
within a given error tolerance. The basic strategy is to reduce this task to judicious applications   of a routine  
\be
\label{routine2}
(w,\eta)\in \Usp\times \R_+ \mapsto [\Cop,w;\eta] \quad \mbox{such that}
\quad \big\Vert \Cop w - [\Cop,w;\eta]\big\Vert \le \eta.
\ee
This routine, in turn, can be realized with the aid 
of an error controlled application scheme $[\cF,z;\eta]$, (i.e.,  
$\Vert \cF z- [\cF,z;\eta]\Vert \le \eta$) in combination with
 an error controlled solver for the {\em source problem}
$[\cB^{-1},q;\eta]$, providing an $\eta$-accurate approximation to the solution
of $\cB u=q$, i.e., $\Vert u- [\cB^{-1},q;\eta]\Vert \le \eta$.
Routines of the form $[\cF,z;\eta]$ have been discussed in \cite{DGM2020} for
non-trivial kernels, using wavelet compression or low-rank approximation.
The routine $[\cB^{-1},q;\eta]$ is the overall main result of \cite{DGM2020}.
So it is justified for the purpose of the subsequent discussion to claim availability of these two routines.

Then, the following
can be easily verified and will be used repeatedly.
\begin{remark}
\label{rem:routine}
With the routines $[\cF,\cdot;\cdot]$ and $[\cB^{-1},\cdot;\cdot]$ at hand
%In fact, $w=\Cop f$ is equivalent to solving $\cB w=\cF f$. Hence,   
\be
\label{Creal}
 [\Cop,f;\eta] := [\cB^{-1},[\cF, f;\eta_1];\eta_2]
\ee
satisfies $\Vert [\Cop,f;\eta]-\Cop f\Vert \le \eta$ provided that $\Vert \cB^{-1}\Vert \eta_1+\eta_2\le \eta$.
\end{remark}

There are actually several pathways of employing  $ [\Cop,f;\eta]$. The
subsequent discussion is to shed light on intrinsic obstructions that arise
along the way as well as to line out remedies.

 To this end,
recall from \eqref{spp}, \eqref{mun} that, given the exact Newton iterates $(u_n, \lambda_n)$, the exact Newton updates  can be obtained via block elimination as follows
\begin{equation}
\label{exactNup}
  (\delta_n^{(u)}, \delta_n^{(\lambda)})
  = \corr{-DR}(u_n, \lambda_n)^{-1}R(u_n, \lambda_n)
  \quad
  \Longleftrightarrow
  \quad
  \begin{cases}
  \delta_n^{(u)} &= \delta_n^{(\lambda)} \Mop_{\lambda_n}^{-1} \Cop u_n - u_n \\
  \delta_n^{(\lambda)} &= \dfrac{1+\<\Cop u_n,u_n\> - \Vert \Cop u_n\Vert ^2/2}{\left< \Cop u_n, \Mop_{\lambda_n}^{-1} \Cop u_n\right>}.
  \end{cases}
\end{equation}
A  first natural approach is to realize
the   scheme \eqref{eq:bar-h}  
by approximately computing the quantities in the  right-hand part of \eqref{exactNup}. However, in addition to $[\Cop,f;\eta]$ this would require
a routine: 
for every $f\in \Usp$, $\nu\in \bR$ and target tolerance $\eta>0$,  compute
\wo{approximations to $\Mop_{\lambda_m}^{-1}\Cop u_n$. Since we have an accuracy controlled
of $\Cop$ at hand, it would be natural to develop a routine}
\be %gin{alignat}{3}
%z_\eta &= [\Cop, f; \eta]
%\quad &&\mbox{such that}&&\quad || z_\eta - \Cop f ||\leq \eta ,\\
w_\eta = [\Mop_\nu^{-1}, f; \eta]
\quad  \mbox{such that}  \quad \Vert  w_\eta - \Mop^{-1}_\nu f \Vert \leq \eta.
\label{schemes}
\ee
The obvious problem is that the operator $\Mop_\nu$ becomes increasingly harder
to invert when $\nu$ approaches $\l$, as the values $\lambda_n$ would do.
\wo{We postpone a way of circumventing these difficulties, provided that (a good lower bound for) the spectral gap $\Delta$ (see \eqref{sgap}) is known, to the end of this section.}\\
 
\wo{We address therefore first another} option that takes advantage of Theorem \ref{thm:main}.
In fact, well-posedness of the saddle point problem just means that $DR(\bar u,\bar\nu)$ is a $\Usp$-isomorphism in a neighborhood of $(\u,\l)$, depending on
$\l,\Vert \Cop\Vert$ and \wo{$\theta$}. Hence, errors are equivalent to residuals in the sense that
\begin{align}
\label{err-res}
\Vert (\bar w,\bar\nu)- (\delta^{(\u)}_n,\delta^{(\l)}_n)\Vert ^2 &\eqsim \Vert DR(u_n,\lambda_n)
(\bar w,\bar\nu)+ R(u_n,\lambda_n)\Vert ^2\nonumber\\
&= \Vert \Mop_{\lambda_n} \bar w- \bar \nu \Cop u_n+ \Mop_{\lambda_n}u_n\Vert ^2
+ \Big(1- \frac{\Vert \Cop u_n\Vert ^2}2 -\big\<\Cop u_n,\bar w\big\>\Big)^2.
\end{align}
Hence, we can view the update $(\delta^{(\u)}_n,\delta^{(\l)}_n)$ as the unique minimizer of a {\em quadratic functional}. Note that, just using the   routine 
 \eqref{Creal}, we can compute the quantities
$$
f_n:=\Cop u_n,\quad g_n := \Mop_{\lambda_n}u_n,\quad s_n:= 1- \frac{\Vert f_n\Vert ^2}2
$$
within any desirable accuracy which leaves us with solving
\be
\label{quadratic}
(\bar\delta^{(u)}_n,\bar\delta^{(\lambda)}_n)= \argmin_{\bar w,\bar \nu}Q(\bar w,\bar \nu),\quad Q(\bar w,\bar \nu):=\frac 12\Big\{
\Vert \Mop_{\lambda_n}\bar w -\bar\nu f_n +g_n\Vert ^2 +
\big(s_n - \<f_n,\bar w\>\big)^2\Big\}.
\ee
Note that the exact update component $\delta^{(u)}_n$ belongs to $\<\Cop u_n\>^\perp$. We know from Lemma \ref{lem:pert-Cu} that $\Mop_{\lambda_n}$
is well-conditioned on $\<\Cop u_n\>^\perp$. Hence, a gradient descent in $\<\Cop u_n\>^\perp$ converges rapidly towards the minimizer. So, we do not suggest to minimize $Q$ over a fixed finite-dimensional trial space but seek to
realize an accuracy controlled gradient descent in function space.

To that end, since $2Q(\bar w,\bar \nu)= \Vert \Mop_{\lambda_n}\bar w-\bar\nu f_n\Vert ^2 +2\big\<\Mop_{\lambda_n}\bar w-\bar\nu f_n,g_n\big\>+ \Vert g_n\Vert ^2
+ s_n^2 +\<f_n,\bar w\>^2 -2s_n\<f_n,\bar w\>$, it suffices to minimize
$$
P(\bar w,\bar\nu) := \frac 12 \Vert \Mop_{\lambda_n}\bar w-\bar\nu f_n\Vert ^2
+ \big\<\Mop_{\lambda_n}\bar w-\bar\nu f_n,g_n\big\>+ \frac 12 \<f_n,\bar w\>^2 -\corr{s_n\<f_n,\bar w\>}.
$$
\corr{Although $\nabla P(\bar w,\bar \nu)$ can be determined along the same lines, as
stated earlier for   $DR$, we repeat the specific calculation for the convenience of the
reader.}
Straightforward manipulations yield
\begin{align*}
2\big(P(\bar w+ t h,\bar\nu + t\omega) - P(\bar w,\bar\nu)\big)&
=   t^2\big\<\Mop_{\lambda_n} h-\omega f_n,\Mop_{\lambda_n} h-\omega f_n\big\> +2t \big\<\Mop_{\lambda_n} h-\omega f_n,\Mop_{\lambda_n}\bar w-\bar\nu f_n \big\>\\
&\quad + 2t\big\<\Mop_{\lambda_n} h-\omega f_n,g_n\big\> + t^2\<f_n,h\>^2+ 2t\<f_n,\bar w\>
\<f_n,h\> - 2s_nt\<f_n,h\>.
\end{align*}
Thus
\begin{align*}
%\label{gradQ}
\lim_{t\to 0}\frac 1t \Big\{P(\bar w+ t h,\bar\nu + t\omega) - P(\bar w,\bar\nu)\Big\}
&= \big\<\Mop_{\lambda_n} h-\omega f_n,\Mop_{\lambda_n}\bar w-\bar\nu f_n +g_n\big\>
 +(\<f_n,\bar w\>-s_n)\<f_n,h\> \\
&= \big\< \Mop_{\lambda_n}^*(\Mop_{\lambda_n}\bar w-\bar\nu f_n+g_n)
+(\<f_n,\bar w\>-s_n)f_n,h\big\>\\
&\qquad-\omega \big\<f_n, \Mop_{\lambda_n}\bar w-\bar\nu f_n+g_n\big\>.
\end{align*}
Hence,
$$
\big\<\nabla P(\bar w,\bar \nu),(h,\omega)\big\> = \left\<{
\Mop_{\lambda_n}^*(\Mop_{\lambda_n}\bar w-\bar\nu f_n+g_n)
+(\<f_n,\bar w\>-s_n)f_n\choose - \big\<f_n, \Mop_{\lambda_n}\bar w-\bar\nu f_n+g_n\big\>},{h\choose \omega}\right\>,
$$
which says that 
\be
\label{steep}
D(\bar w,\bar\nu):= -\nabla P(\bar w,\bar\nu)= 
{
-\Mop_{\lambda_n}^*(\Mop_{\lambda_n}\bar w-\bar\nu f_n+g_n)
+(\<f_n,\bar w\>-s_n)f_n\choose  \big\<f_n, \Mop_{\lambda_n}\bar w-\bar\nu f_n+g_n\big\>}
\ee
is the direction of steepest descent in $\Uspext=\Usp\times \R$. 
In the light of the preceeding remarks, since we know that $\delta^{(u)}_n\in 
\<\Cop u_n\>^\perp$, a natural initialization
for such a gradient descent scheme would be
$$
\bar \omega^0 := \Cop u_n- P_{\<\Cop u_n\>}u_n \in \<\Cop u_n\>^\perp,\quad 
\bar\nu^0=0.
$$
Then $D(\bar w^0,\bar\nu^0)$ simplifies somewhat to
\be
\label{simpler}
D(\bar w^0,\bar\nu^0):= -\nabla P(\bar w,\bar\nu)= 
{
-\Mop_{\lambda_n}^*(\Mop_{\lambda_n}\bar w^0-\bar\nu f_n+g_n)
-s_nf_n\choose  \big\<f_n, \Mop_{\lambda_n}\bar w-\bar\nu f_n+g_n\big\>}
={D_1(\bar w^0,\bar\nu^0)\choose D_2(\bar w^0,\bar\nu^0)}.
\ee
The first component does perhaps not belong to $\<\Cop u_n\>^\perp$
but is close to. This suggests taking
$$
{\bar w^1\choose\bar\nu^1}={\bar w^0\choose \bar\nu^0}+ \xi{P_{\<\Cop u_n\>^\perp} D_1(\bar w^0,\bar\nu^0)
\choose D_2(\bar w^0,\bar\nu^0)}, 
$$
for a \corrw{suitable} step-size $\xi>0$,
and repeat  the step based on \eqref{simpler}. One then obtains analogous update formulas
for ${\bar w^k\choose\bar\nu^k}$. In fact, one could even determine 
(at the expense of further approximate applications of $\Cop, \Cop^*$) an optimal stepsize but we skip corresponding details.

Executing such a descent strategy requires only the scheme $[\Cop,f;\eta]$
(as well as a similar variant $[\Cop^*,f;\eta]$).
Without going into further details (deferred to forthcoming numerical work)
the update tolerances in those applications need to be {\em fixed fractions} of the
target accuracy $\eta_n$ in \eqref{relaxeta} whith $\zeta$ depending on which
convergence strategy is being pursued, see \eqref{linred} or \eqref{eta-n-rule}.

Finally, whether the target accuracy has been met for a given iterate ${\bar w^k\choose\bar\nu^k}$ can be checked by evaluating the residual $Q(\bar w^k,\bar\nu^k)$ from \eqref{quadratic}. Therefore, this strategy allows one, in principle, to realize the Newton update \eqref{eq:approx-inv-DR} 
in an accuracy controlled fashion. Due to the controlled condition of $\Mop_{\lambda_n}$ on $\<\Cop u_n\>^\perp$ an error reduction by a fixed factor,
required by \eqref{linred}, will be achieved after a uniformly bounded
number of descent steps while strategy \eqref{eta-n-rule} would require 
the order of $|\log \bar e_n|$ steps.\\ 

\wo{As announced above, returning to \eqref{exactNup}, we briefly sketch now an alternate approach towards an error controlled
approximation of $\Mop_{\bar\lambda}^{-1}\Cop\bar u$ when $\Delta$ is known and
$(\bar\lambda,\bar u)$ is already an accurate approximation to the principal eigenpair $(\l,\u)$. To that end, it will be convenient to rewrite
\begin{equation}\label{sgap2}
\oDelta = \m | \nu -\m |,\quad \m=(\l)^{-1}, \nu= 
\argmax \{|\mu|: \mu \in \sigma(\Cop)\setminus =\{\m\}
\}.
\end{equation}
Specifically, assume that (e.g. by the above technique or with the aid of the 
scheme discussed in the next section) $\bar\mu = \bar\lambda^{-1}$ already
satisfies $|\bar\mu -\m| \le \l\oDelta/8$, say. Consider a disc $\Omega\subset\bC$
with boundary $\Gamma$, center $\bar\mu$ and radius $\l\oDelta/2$. Thus $\Gamma$ keeps a distance at least $3\l\oDelta/8$ from $\m$ and $\nu$ hence from $\sigma(\Cop)$, while
$|\Gamma|=  \pi \l\oDelta$. Writing
$$
\Mop_{\bar\lambda}^{-1}\Cop = \bar\mu (\bar\mu\id -\Cop)^{-1}\Cop,
$$
classical functional calculus says that
\be
\label{contour}
\Mop_{\bar\lambda}^{-1}\Cop \bar u = \frac{1}{2\pi i}\int_\Gamma  \frac{\bar\mu\zeta}{\bar\mu -\zeta}(\id\zeta- \Cop)^{-1}\bar u d\zeta,
\ee
(see   Section \ref{ssec:spec}).
By construction, the resolvent $(\zeta\id-\Cop)^{-1}$ possesses a holomorphic extension to
a symmetric strip $A$ of width $3\l\Delta/8$ around $\Gamma$. It is well known (see e.g. \cite{Stenger}) that the trapezoidal rule provides then an approximation to the above
contour integral with exponential accuracy. Specifically, consider equidistant quadrature points $\zeta_j\in\Gamma$, $j=1,\ldots,N_\e$, and 
\be
\label{QN}
I_{h}(\bar u):=\frac{h}{2\pi i}\sum_{j=1}^{\pi \l\Delta/h} \frac{\bar\mu \zeta_j}{\bar\mu-\zeta_j}(\id\zeta_j-\Cop)^{-1}\bar u  .
\ee
Then, employing $N = \pi \l\oDelta/h$ quadrature points (see e.g. \cite{Jurgens2005, Jurgens2006, DJ})
$$
\big\|I_{h}(\bar u)- \Mop_{\bar\lambda}^{-1}\Cop \bar u\big\|_{\U} \le C e^{-\pi 3\l\oDelta /h8} = C e^{-3N/8}.
$$
Thus, given the exact states $(\id\zeta_j-\Cop)^{-1}\bar u$, a constant multiple of
$N_\e\eqsim |\log \e|$ quadrature points suffice to realize accuracy $\e$. Hence, if in addition one is able to generate approximations
$r_j(\bar u)\approx (\id\zeta_j-\Cop)^{-1}\bar u$ of order $\e$, i.e.,
$$
\Big\|I_{h_\e}(\bar u)- \frac{h_\e}{2\pi i}\sum_{j=1}^{N_\e} \frac{\bar\mu \zeta_j}{\bar\mu-\zeta_j}r_j(\bar u)\Big\|_{\U}\lesssim \e,
$$
\worr{one generates approximations to $\Mop_{\bar\lambda}^{-1}\Cop \bar u$ of order $\e$.}
This leaves the task of approximately solving the $|\log\e|$ operator equations
\be
\label{Cops}
(\zeta_j\id - \Cop)w_j= \bar u,\quad j=1,\ldots, N_\e.
\ee
The principal gain is that in these equations the $\zeta_j$ remain uniformly 
away from $\sigma(\Cop)$ (with a distance of order $\m\Delta$) and remain in this sense
well conditioned. Moreover, \eqref{Cops} is equivalent to solving
$$
(\Bop -\zeta_j^{-1}\Fop )w_j = \zeta_j^{-1}\bar u,
$$
which has a similar structure as the original source problem, now with a modified global part, hence should be amenable
to an adapted version of $[\Bop^{-1},f;\eta]$. Details are left to forthcoming work.}
%%%%%%%%%%%%%%%%%%%%%%%%%%%%%%%
\section{\wo{Initialization Strategies and Power Iteration}}
\label{sec:inverse-power-it}
 A fully certified numerical realization of (CC1) - (CC3), based on Newton's method as idealized iteration, requires an initial guess in the {\em admissible} neighborhood
specified in Theorem \ref{thm:convergence}. First, this requires knowledge of the quantities 
\wo{$\|\Cop\|, \theta , \l, \u$, which are generally  not known. We recall that $\theta$ is small when the spectral gap $\Delta$ is small, see \eqref{sandwich}, \eqref{sgap}}.  Therefore, the overarching objective in this section is to approximate these quantities, in principle, as accurately as one wishes,
without the need of a sufficiently close initial guess. This calls for
a {\em globally convergent} idealized iteration in (CC1).

A natural idea  is then
to employ a {\em power iteration}, as an idealized iteration in (CC1). 
To be specific, let $\wor{a^\circ_0} \in \Usp$ with $\Vert \wor{a^\circ_0}\Vert =1$,
$\alpha :=\<\wor{a^\circ_0},\ug_1\>\neq 0$ be an initial guess. Then,  define
\begin{equation}
\label{power}
a_{n+1}:= \Cop \wor{a^\circ_n},\quad \wor{a^\circ_{n+1}}:= \frac{a_{n+1}}{\Vert a_{n+1}\Vert },\quad n=0,1,2,\ldots.
\end{equation}
The associated Rayleigh quotients are
\begin{equation}
\label{ray}
\rho_{n+1}:= \<\Cop \wor{a^\circ_n},\wor{a^\circ_n}\>,\quad n=0,1,2,\ldots.
\end{equation}
Again, in the spirit of previous discussions, we proceed
in two stages: the first goal is
  to understand  quantitative convergence characteristics of 
this idealized power method, formulated in this section, we focus on this task. 
With quantitative convergence results on these iterations at hand, one would 
contrive in a second step approximate numerical approximations of the 
idealized iteration obeying suitable accurracy tolerances that ensure 
convergence to the correct principal eigenpair. Since for the power-method this
second step relies in essence on the ability to have an accuracy controlled 
application of $\Cop$ we dispense with corresponding details that in spirit follow the same lines as in previous sections but rather focus on the convergence
of \eqref{power} and \eqref{ray} in the infinite-dimensional setting.

The convergence of the power method for a general (non-normal) compact operator is already of interest in its own right since there does not seem
to be much known, if anything at all. One reason is perhaps that the classical convergence proof for finite dimensional matrices
is intrinsically finite dimensional and does not carry over to the infinite-dimensional case. For once, the concept of diagonalizability is far too restrictive. Moreover,
for a fixed finite dimension an eigenbasis is automatically stable but its condition
may depend on the dimension unless one is dealing with a normal matrix.
For the present type of operator in infinite dimension, neither the notion of eigenbasis is
appropriate nor are we lacking the necessary stability of an expansion system.
The resulting main result reads as follows.
\begin{theorem}
\label{thm:power}
For $\Cop$, defined in \eqref{Cop}  the power iteration \eqref{power}, \eqref{ray} converges \wo{linearly}
to the principal eigenpair $(\ug_1, \mu_1)\in \tilde\Usp$, i.e.,  for any $1>\bar\delta>1-\uDelta$ there exists a constant
$C<\infty$ such
 the quantities $a^\circ_n, \rho_n$,
defined in \eqref{power} and \eqref{ray}, satisfy
\be
\label{powerconv}
\Vert \ug_1- \wor{a^\circ_n}\Vert \le C\bar\delta^n,\quad |\mu_1- \rho_n|\le C\bar\delta^n,\quad n\in\N.
\ee
The constant $C$ depends on $\sigma(\Cop)$ and may be arbitrarily large. If in addition
one assumes $\Cop$ belongs to the Schatten class $S_p$ of compact operators on $\U$
  for some $p<\infty$,  the  constant $C$ in \eqref{powerconv} can be quantified in terms of $\wor{a^\circ_0}$, $\|\Cop\|_p$, and 
   $ \uDelta$.
\end{theorem} 

For the convenience of the reader we recall below the precise definition of the Schatten class $S_p$ and a few known relevant facts.
% is given below in Section \ref{ssec:spec}. 
 Membership to a Schatten class could perhaps be viewed as replacing the knowledge of
the condition of an eigenbasis in the matrix case. 

 As a substitute for eigen-expansions in the matrix case we resort to   classical functional calculus and Riesz projections, see Section \ref{ssec:spec}. 

Exploiting the bounds \eqref{powerconv} for step (CC2), i.e., for an accuracy controlled
approximate execution of \eqref{power} and \eqref{ray} not only requires an estimate of the constant $C$ in \eqref{powerconv} but also a quantitative (hopefully sharp) lower
bound of the spectral gap $\Delta$. The importance of knowing such an estimate has already transpired in previous sections. 

In summary, it seems that  intrinsic difficulties of numerically solving spectral problems with certified accuracy for the current scope of unsymmetric compact operators 
 can be reduced to some extent to assessing $\Delta$ (and this fact is actually  confirmed  by works such as \cite{ben2015computing}). 
Therefore, we highlight in Section \ref{ssec:ram} the difficulties in estimating $\Delta$
and the role of heuristics in gaining computational information on $\uDelta$.

The statements in Theorem \ref{thm:power} reflect that we are not able to quantify
  convergence   under the mere assumption of compactness when $\Cop$ is not normal.
 As we will see later below the difficulty is to quantify resolvent bounds of the
 type $\max_{\zeta\in \Gamma}\|\cR_{\Cop}(\zeta)\|$  in terms of the distance
 of the contour $\Gamma$ from the spectrum. In fact, 
 in the most favorable situation
that $\Cop$ is a {\em normal} operator, it is known that  
\be
\label{ideal}
\Vert \Rop_{\Cop_\omega}(\zeta)\Vert \le \big(\inf_{\zeta'\in \Gamma_\omega}|\zeta-\zeta'|\big)^{-1}:=\corr{d(\zeta,\Gamma_\omega)^{-1}},
\ee
while it is known that  
\eqref{ideal} is generally not true when $\Cop$ is not normal.
To see this, we recall
the following version of Schur's Lemma that 
ensures for every compact operator $\Aop$ the existence 
of a decomposition  $\Aop=D+N$ (\cite[Theorem 3.2]{Carleman}) where $D,N$
are compact, $D$ is normal $\sigma(\Aop)=\sigma(D)$, and $N$ is quasi-nilpotent, i.e., $\sigma(N)=\{0\}$. Then powers of $\Aop$ are of the form
$$
\Aop^k = D^k + B_k,
$$ 
where $B_k$ is a polynomial in $D$ and $N$ of degree $k$. Although, by compactness, properly rescaled
powers of $N$ will (eventually) tend to zero it is not so clear what the right scaling for the
mixed terms in $B_k$ should be in order to produce a {\em quantitative} damping of $B_k$
as $k$ tends to infinity. In fact, the deviation of normality of $\Aop$ reflected
by $N$ will play later a role in the proof of Theorem \ref{thm:power}. 

The deviation from normality can indeed be quantified when  $\Cop$  
   belongs to some {\em Schatten class}. 
We briefly recall this concept
for a general compact operator $\Aop$ on a Hilbert space $H$. Let  
 $S_\infty$
denote the set of all compact operators on $H$. Then the $p$-th Schatten class
is defined as \corrw{(see e.g. \cite{Carleman,Dunford2})}
$$
S_p:= \Big\{\Aop\in S_\infty: \Vert \Aop\Vert _p := \Big(\sum_{j=1}^\infty s_j(\Aop)^p \Big)^{1/p}<\infty\Big\},
$$
where the $s_j(\Aop)$ are the singular values of $\Aop$.  Obviously,
every finite rank operator belongs to every $S_p$, $p\in (0,\infty]$.
In general, the larger $p$ the weaker the condition. 

For later use in Section \ref{ssec:proof} we invoke now some results from \cite{Carleman} which we briefly formulate
again in general terms for a general compact operator $\Aop$ on a Hilbert space $H$.  
Recall the Schur decomposition $\Aop=D+N$ where $D$ is normal with
$\sigma(\Aop)=\sigma(D)$ and $N$ quasi-nilpotent. 
Since 
\corr{$(z\Iop- \Aop)^{-1} = (\Iop-(z\Iop-D)^{-1}N)^{-1}(z\Iop-D)^{-1}$}
 one gets
 $$
 \Vert (z\Iop- \Aop)^{-1}\Vert 
 \le \Vert (\Iop-(z\Iop-D)^{-1}N)^{-1}\Vert \Vert (z\Iop-D)^{-1}\Vert 
 \le \Vert (\Iop-(z\Iop-D)^{-1}N)^{-1}\Vert  d(z,\sigma(D))^{-1},
 $$
 ending up with the task of bounding $\Vert (\Iop-(zI-D)^{-1}N)^{-1}\Vert $. This is 
 done in \cite[Theorem 4.1]{Carleman} providing for $\Aop\in S_p$, $p<\infty$,
 \be
 \label{Ndev}
 \Vert (z\Iop- \Aop)^{-1}\Vert \le d(z,\sigma(\Aop))^{-1}\exp\Big\{\frac{a_p 2^{1+(p-1)_+}\Vert \Aop\Vert _p^p}{d(z,\sigma(\Aop))^p}+b_p\Big\},
 \ee
$a_p,b_p$ constants depending only on $p$. 
\begin{remark}
\label{ap}
{\rm (see \cite[Remark 2.2]{Carleman})} For certain values of $p$ the constants
$a_p,b_p$ are known:
\begin{enumerate}
\item
$p=1$: $a_1=1, \,b_1=0$;
\item
$p=2$: $a_2=\frac 12 = b_2$.
\item
For $0<p\le 1$ a possible choice is
\be
\label{p<1}
a_p=\sup_{z\in\bC}|z|^{-p}\log|(1+z)|,\quad b_p=0,
\ee
while for $p>1$ one can take \corr{for $a_p$} any real number larger than
\be
\label{pge1}
\underline{p}:= \sup_{z\in\bC}\log\Big|(1+z)\exp\Big\{\sum_{j=1}^{\lceil p\rceil-1}
\frac{(-z)^j}j\Big\}\Big|,
\ee
\corr{with $b_p\neq 0$ depending on $a_p$.}
\end{enumerate}
\end{remark}
%
%%%%%%%%%%%%%%%%%%%%%%%%%%%%%
\newcommand{\Omi}{{\Omega_\e}}
\newcommand{\Gami}{{\Gamma_\e}}
%%%%%%%%%%%%%%%%%%%%%%%%%%%%%%

%%%%%%%%%%%%%%%%%%%%%%%%%%%%%%%%%%
\subsection{\wo{Proof of Theorem \ref{thm:power}}}\label{ssec:proof}
%%%%%%%%%%%%%%%%%%%%%%%%%%%%%%%%%%%%%%%%%%
\newcommand{\bdel}{\overline{\delta}}
%%%%%%%%%%%%%%%%%%%%%%%%
We take up again the notation from Section \ref{ssec:spec}
for the spectrum
$$
\sigma(\Cop)= \{\mu_j:j=1,\ldots,\infty\},
$$
where $|\mu_j|$ decreases with increasing $j\in\N$ (recall the correspondence 
$\mu_1= \mu^\circ= (\l)^{-1}$, $\ug^\circ=\ug_1$). 
Note that the spectral gap \eqref{sgap} can equivalently be written as
$$
\wo{\oDelta = \Big|1-  \frac{\mu_2 }{\mu_1}\Big|\ge 1- \frac{|\mu_2|}{\mu_1}}=\uDelta.
$$
\corr{Finally, it will be convenient to assume in what follows the normalization $\|\ug_1\|=1$ for the principal eigenstate 
$\ug_1$.}

The proof of Theorem \ref{thm:power} is based on   
 analyzing the action of $\Cop$
on  invariant subspaces of the form $V_{\omega}:= \corrw{\Eop}_{\Cop}(\omega)\Usp$
for subsets $\omega \subset \sigma(\Cop)$. The following
facts are probably standard but
  we include them for completeness.
\begin{lemma}
\label{lem:Ck}
(a) For $\omega\subset \sigma(\Cop)$
let $\Cop_\omega := \Cop|_{V_\omega}$ denote the restriction of $\Cop$ to $V_\omega$
(recall that $\Cop V_\omega = V_\omega$). We have 
\be
\label{Resk0}
\Rop_{\Cop}(\zeta)\!\mid_{V_\omega}= \Rop_{\Cop_\omega}(\zeta),
\ee
and  the representation 
\be
\label{Ck}
\Cop^\ell_\omega =
 \frac{1}{2\pi i}\int_{\Gamma_\omega}\zeta^\ell\Rop_{\Cop_\omega}(\zeta)\,d\zeta,
\ee
holds for every $\ell\in \N$
where, as before, $\Gamma_\omega=\partial\Omega$ for some domain $\Omega\subset\bC$ 
such that $\omega\subset\Omega$, $\Omega \cap (\sigma(\Cop)\setminus \omega)=\emptyset$.\\[1.5mm]
(b) Moreover, we have 
\be
\label{CE0}
\Vert \Cop^\ell \ug \Vert \le \frac{|\Gamma_\omega|}{2\pi}\max_{\zeta\in\Gamma_\omega}|\zeta|^\ell
\Vert \Rop_{\Cop_\omega}(\zeta)\ug\Vert ,\quad \forall\, \ug\in V_\omega,
\ee
where $|\Gamma_\omega|$ denotes the length of $\Gamma_\omega$.
\end{lemma}
\begin{proof} By the projection property of $\Eop_{\Cop}$ we have
for every $\ug\in V_\omega$ that 
$$
\ug = \Eop_{\Cop}(\omega)\ug = \frac{1}{2\pi i}\int_{\Gamma_\omega}\Rop_{\Cop}(\zeta)\ug\,d\zeta.
$$
When $|\zeta|>\Vert \Cop\Vert \ge \Vert \Cop_\omega\Vert $ we have for any $\ug\in V_\omega$
\begin{align}
\label{Resk}
 (\zeta\Iop -\Cop_\omega)^{-1}\ug &=  \zeta^{-1}  (\Iop - \zeta^{-1}\Cop_\omega)^{-1}\ug = \zeta^{-1} \sum_{j=0}^\infty
\frac{ \zeta^{-j} \Cop_\omega^j\Eop_{\Cop}(\omega)^j\ug }{j!}\nonumber\\
&=
\zeta^{-1} \sum_{j=0}^\infty
\frac{ \zeta^{-j} \Cop^j\ug }{j!} = (\zeta\Iop-\Cop)^{-1}\ug.
\end{align}
Since the resolvent is holomorphic on $\rho(\Cop)$ the above agreement
must hold for all $\zeta\in \rho(\Cop)$, confirming \eqref{Resk0}.

Clearly, $\sigma(\Cop_\omega)= \omega$. Let $\Omega_\omega'$ with boundary $\Gamma'$ be any domain in $\bC$
containing $\Gamma_\omega$ while one still has $\Omega'_\omega \cap (\sigma(\Cop)\setminus\omega)=\emptyset$. Then,   we have (by standard holomorphic functional calculus)
$
\Cop^\ell_\omega = \frac{1}{2\pi i}\int_{\Gamma'} \xi^\ell\Rop_{\Cop_\omega}(\xi)\,d\xi$.
Next recall 
  Hilbert's formula (1st resolvent formula)
\be
\label{1resolv}
 {\Rop}_{\Cop}(\zeta) {\Rop}_{\Cop}(\xi) = (\xi-\zeta)^{-1}\big( {\Rop}_{\Cop}(\zeta)- {\Rop}_{\Cop}(\xi)\big).
\ee
Since $\Omega'_\omega$ strictly contains $\Gamma_\omega $ 
we obtain for $\ug\in V_\omega$ (so that $\ug=\Eop_{\Cop}(\omega)\ug$)
\begin{align*}
\Cop^\ell_\omega \ug &= \Big(\frac{1}{2\pi i}\Big)^2\int_{\Gamma'_\omega}\int_{\Gamma_\omega}
\xi^\ell\Rop_{\Cop}(\xi)\Rop_{\Cop}(\zeta)\ug\, d\xi\,d\zeta 
= \Big(\frac{1}{2\pi i}\Big)^2\int_{\Gamma'_\omega}\int_{\Gamma_\omega}
\frac{\xi^\ell}{\zeta-\xi}\big(\Rop_{\Cop}(\xi)-\Rop_{\Cop}(\zeta)\big)\ug\,
d\xi\,d\zeta.
\end{align*}
Since residue calculus yields
$$
\frac{1}{2\pi i}\int_{\Gamma'_\omega}\frac{\xi^\ell}{\zeta-\xi}\Rop_{\Cop}(\zeta)\ug\,d\xi
= -\zeta^\ell \Rop_{\Cop}(\zeta)\ug,\quad 
\frac{1}{2\pi i}\int_{\Gamma_\omega}\frac{\xi^\ell}{\zeta-\xi}\Rop_{\Cop}(\xi)\ug\,d\zeta
= 0,
$$
(the second relation, because $\xi$ is outside $\Omega_\omega$) relation \eqref{Ck} follows. The rest of the assertion (b) follows from \eqref{Ck}.
\end{proof}

To see whether $(\mu_1^{-1}\Cop)^\ell$ contracts on subspaces $V_\omega$ when $\mu_1\notin\omega$, 
we   specialize Lemma \ref{lem:Ck} to a domain $\Omega=\Omi$
with boundary $\Gami$
such that
\be
\label{Omegae}
\mu_1\notin \Omega_\e,\quad \sigma_{>1}:=\sigma(\Cop)\setminus\{\mu_1\}\subset\Omi,\quad {\rm dist}(\Gamma_\e,\sigma_{>1})\ge {\rm dist}(\Gami,\mu_2)= \e.
\ee
Here $\e$ is a fixed constant whose value will be stipulated later.
%%%%%%%%%%%%%%%%
\renewcommand{\Omi}{\sigma_{>1}}
%%%%%%%%%%%%%%%%%%%%%%%%%%
The goal is to show that the scaled powers $(\mu_1^{-1}\Cop)^\ell$
(eventually) contract on the invariant subspaces $V_{\Omi}$. 
\begin{lemma}
\label{lem:Ck2}
Adhering to the above notation, let $\wor{1-\Delta =\rat} < \bdel< 1$. Then,
there exists an $\ell_0\in \N$ such that
\be
\label{ezero}
\Vert (\mu_1^{-1}\Cop)^\ell \cE_{\Cop}(\sigma_{>1})\ug\Vert \le \bdel^{\ell-\ell_0} \Vert \cE_{\Cop}(\sigma_{>1})\ug\Vert  ,\quad \ell\in \N,\, \ug\in \Usp.
\ee
If one assumes in addition that 
 $\Cop\in S_p$
for some $1\le p<\infty$, then  $\ell_0\in \N$ can be bounded in terms of
$\bdel,\|\Cop\|_p,\Omega_\e$. 
\end{lemma}
\begin{proof} It is well-known that $\cR_{\Cop}(\zeta)$ is bounded for each $\zeta\not\in\sigma\wor{(\Cop)}$. Hence $M_\e:= \max_{\zeta\in \Gamma_\e}\|\cR_{\Cop}(\zeta)\|<\infty$. 
Thus, 
\eqref{CE0}
yields for any $\ug\in V_{\sigma_{>1}}$
\be
\label{genest}
\Vert (\mu_1^{-1}\Cop)^\ell \ug\Vert \le \frac{|\Gami|}{2\pi}\Big(\frac{\max_{\zeta\in \Gami}|\zeta|}{\mu_1}\Big)^\ell M_\e \|\ug\|
\le \frac{|\Gami|}{2\pi}\Big(\frac{|\mu_2|+\e}{\mu_1}\Big)^\ell M_\e \|\ug\|, \quad \ug\in V_{\Omi}.
\ee
Now recall that \wo{$\uDelta= 1-\frac{|\mu_2|}{\mu_1} $} and  fix some $\beta\in (0,1)$ to
define 
\wor{
\be
\label{bdelta}
\e:= \beta(\mu_1-|\mu_2|) = \beta\mu_1 \Delta,\quad \bdel:= \frac{|\mu_2|+\e}{\mu_1}=\Delta +\beta(1-\Delta)=\beta + (1-\beta)\Delta.
\ee
}
Thus, $\bdel$ is strictly less than one for $\beta< 1$ and tends to $\Delta$ when 
$\beta$ tends to zero. Hence, on account of \eqref{genest},  \eqref{ezero} follows with
\be
\label{l0gen}\wor{
\ell_0:= \min\,\Big\{\ell\in\N: \frac{\bdel^\ell|\Gami|M_\e}{2\pi}\le 1\Big\}.}
\ee
This confirms the first part of the assertion.

%\wor
{
\begin{remark}
\label{rem:whySp}
For a general compact non-normal operator $\Cop$ on $\Usp$ we cannot further
quantify the constant $M_\e \worr{=}\max_{\zeta\in \Gamma_\e}\|\cR_{\Cop}(\zeta)\|$ which could be arbitrarily large. Hence, $\ell_0$ 
in \eqref{l0gen} could be arbitrarily large which prevents
us from concluding any meaningful bound on the constant in \eqref{powerconv}.
\end{remark}
}

Now assume that in addition $\Cop\in S_p$ for some $1\le p<\infty$. 
Then \eqref{CE0}
combined with \eqref{Ndev} yields
\begin{align}
\label{CE}
\Vert \Cop^\ell\ug\Vert&= \Vert \Cop_{\Omi}^\ell\ug\Vert  \le  \frac{|\Gami|}{2\pi}(|\mu_2|+\e)^\ell
 \e^{-1}\exp\Big\{\frac{2^{1+(p-1)_+}a_p\Vert \Cop_{\Omi}\Vert _p^p}{\e^p
 }+b_p\Big\}\Vert u\Vert ,
 \quad \ug\in  V_{\sigma_{>1}}.
\end{align}
Next observe that $\Cop_{\Omi}=\Cop E_{\Cop}(\Omi)$ belongs to $S_p$ when $\Cop$ does.
In fact,
\be
\label{OmS}
\|\Cop_{\Omi}\|_p\le 2\|\Cop\|_p,
\ee
To see this, recall that $\Cop = \Cop E_{\Cop}(\mu_1)+ \Cop E_{\Cop}(\Omi)$ so that
$\|\Cop E_{\Cop}(\Omi)\|_p\le \|\Cop  \|_p +  \|\Cop E_{\Cop}(\mu_1)\|_p$. 
Since $\Cop E_{\Cop}(\mu_1)$ has rank one its single non-zero singular value is
given by the square root of
$$
\max_{\ug\in \Usp}\frac{\< E_{\Cop}(\mu_1)\ug,\Cop E_{\Cop}(\mu_1)(\ug)\>}{\|\ug_1\|^2} = \frac{\< E_{\Cop}(\mu_1)\ug_1,\Cop E_{\Cop}(\mu_1)\ug_1\>}{\|\ug_1\|^2} = \mu_1^2.
$$
Since by Weyl's inequality eigenvalues are dominated by singular values \eqref{OmS} follows.

 Hence, we conclude from \eqref{CE} that
\be
\label{1or2}
\Vert \mu_1^{-\ell}\Cop_{\Omi}^\ell\ug\Vert \le \Big(\frac{|\mu_k|+\e}{\mu_1}\Big)^\ell 
M(\beta,\Delta, \Cop,p) \Vert u\Vert,\quad \ug\in V_{\Omi},
\ee 
where
\be
\label{MM}
M(\beta,\Delta, \Cop,p):= \frac{|\Gami|}{2\pi \e}
\exp\Big\{\frac{2^{p+1+(p-1)_+}a_p\Vert \Cop \Vert _p^p}{(\beta(\mu_1-|\mu_2|))^p
 }+b_p\Big\}.
\ee 
Taking in analogy to \eqref{l0gen}
\be
\label{l0}
\ell_{0 }=\ell_{0 }(\Delta,\beta,\Cop,p):= \argmin\big\{\ell\in\N: \bdel^\ell  M(\Delta,\beta,\Cop,p)\le 1\big\},
\ee
finishes the proof.
\end{proof}
\medskip

To proceed we define a new  (equivalent) norm on $\Usp$ by 
\be 
\label{newnorm}
\tripnorm{\ug}:=  \Vert E_{\Cop}(\mu_1)\ug\Vert + \Vert E_{\Cop}(\Omi)\ug\Vert ,
\ee
i.e., there exist constants $c_1,C_1$, 
depending on $\Cop$ such that
\be
\label{equiv}
  c_1\tripnorm{\cdot}\le \Vert \cdot\Vert \le C_1 \tripnorm{\cdot}.
  \ee

Recalling the shorthand notation $\< \ug\>:= {\rm span}\{\ug\}$, it will be convenient
to introduce the ``distance''
\be
\label{distsp}
\dist(\<\ug\>,\<\ug_1\>):= \min_{c\in \R} \Vert c\ug - \ug_1\Vert .
\ee

\begin{lemma}
\label{lem:contract}
Assume that $\wor{a^\circ_0}\in \Usp$ satisfies $\Vert \wor{a^\circ_0}\Vert =1$ and 
\wor{$a^\circ_0\ge 0$} so that
$$
E_{\Cop}(\mu_1)\wor{a^\circ_0} =\alpha \ug_1\quad \wor{\mbox{for some $\alpha > 0$}}.
$$
Let $\wor{a^\circ_n}, \rho_n, \ell_0$ be defined by \eqref{power}, \eqref{ray}, and \eqref{l0}.    Then we have
\be
\label{contract1}
\alpha\, \corr{\dist}(\<\wor{a^\circ_n}\>,\<\ug_1\>) \le \Vert (\mu_1^{-1}\Cop)^n\wor{a^\circ_0}- \alpha\ug_1\Vert \le \frac{C_1}{c_1}\bdel^{n- \ell_0} ,\quad n\in \N.
\ee
Furthermore,
\be
\label{contract2}
\Vert \wor{a^\circ_n}- \ug_1\Vert \le    {\wor{\sqrt{2}}\dist(\<\wor{a^\circ_n}\>,\<\ug_1\>)\le \frac{2C_1}{\alpha c_1}}\bdel^{n- \ell_0},\quad n\in \N,  
\ee
and
\be
\label{contract3}
|\rho_{n+1}- \mu_1|\le C\bdel^{n- \ell_0},\quad n\in \N,
\ee
where $C = \alpha^{-1}(\mu_1+\Vert \Cop\Vert )\frac{C_1}{c_1}$.
\end{lemma}
\begin{proof}
Since $\ug_1= (\mu_1^{-1}\Cop)^n\ug_1$ and $E_{\Cop}(\Omi)(\mu_1^{-1}\Cop)^n\ug_1
=0$ for all $n\in\N$,  we have
\begin{align*}
C_1^{-1}\Vert (\mu_1^{-1}\Cop)^n\wor{a^\circ_0} - \alpha \ug_1\Vert &\le
\tripnorm{(\mu_1^{-1}\Cop)^n\wor{a^\circ_0} - \alpha \ug_1}\\
&= \Vert E_{\Cop}(\mu_1)((\mu_1^{-1}\Cop)^n\wor{a^\circ_0} - \alpha\ug_1)\Vert +
 \Vert E_{\Cop}(\Omi)(\mu_1^{-1}\Cop)^n\wor{a^\circ_0}\Vert   \\
&= \Vert ((\mu_1^{-1}\Cop)^n (E_{\Cop}(\mu_1)\wor{a^\circ_0} - \alpha\ug_1)\Vert +
 \Vert E_{\Cop}(\Omi)(\mu_1^{-1}\Cop)^n\wor{a^\circ_0}\Vert   \\
 &=  \Vert (\mu_1^{-1}\Cop)^n(E_{\Cop}(\Omi)\wor{a^\circ_0})\Vert \\
 &\le \bdel^{\wor{n}-\ell_0} \Vert E_{\Cop}(\Omi)\wor{a^\circ_0}\Vert\\
 &\le \bdel^{n-\hat\ell}\tripnorm{  \wor{a^\circ_0} }\le c_1^{-1}\bdel^{n-\hat\ell}\Vert \wor{a^\circ_0}\Vert 
= c_1^{-1}\bdel^{n-\hat\ell},
\end{align*}
where we have use Lemma \ref{lem:Ck2} in the second but last line.
This confirms the
second inequality in \eqref{contract1}. Since $\<\wor{a^\circ_n}\>= \<\Cop^n \wor{a^\circ_0}\>= 
\<c\Cop^n \wor{a^\circ_0}\>$ for all $c\neq 0$ and $n\in\N$, the first inequality in 
\eqref{contract1} follows as well.

To proceed let $s_n:= \argmin_{s\in\wor{\R}}\Vert s \wor{a^\circ_n}- \ug_1\Vert\corr{=\<a^\circ_n, u_1\>}$. \wor{Since $a^\circ_0\ge 0$ and $\Cop$ is
a positive operator  we have $s_n\ge 0$ for all $n\in\N$. Observing by direct calculation that $\|a^\circ_n-u_1\|^2=2(1-s_n)$ and $\|s_n a^\circ_n -u_1\|^2 = 1-s_n^2$,
we obtain
$$
\dist\big(\<a^\circ_n\>,\<u_1\>\big)^2 = \|s_n a^\circ_n -u_1\|^2= (1+s_n)(1-s_n)= \frac{1+s_n}{2}\|a^\circ_n-u_1\|^2,
$$
which, upon using \eqref{contract1}, proves \eqref{contract2}.}

Concerning \eqref{contract3}, one has
\begin{align*}
|\rho_{n+1}-\mu_1|&= |\< \Cop\wor{a^\circ_n},\wor{a^\circ_n}\>- \<\Cop \ug_1,\ug_1\>|\le
|\< \Cop(\wor{a^\circ_n}-\ug_1),\wor{a^\circ_n}\>| + |\<\Cop\ug_1,\wor{a^\circ_n}- \ug_1\>|\\
&\le \Vert \Cop(\wor{a^\circ_n}-\ug_1)\Vert + \mu_1\Vert \wor{a^\circ_n}- \ug_1\Vert \le (\mu_1+\Vert \Cop\Vert )
\Vert \wor{a^\circ_n}- \ug_1\Vert 
\end{align*}
which completes the proof.
\end{proof}

%%%%%%%%%%%%%%%%%%%%%%%%%%%%%%%
\subsection{\wo{Practical Aspects}}
\label{ssec:ram}
%%%%%%%%%%%%%%%%%%%%%%%%%%%%%%%%%%%%%%
%%%%%%%%%%%%%%%%%
\newcommand{\ec}{{e^\circ}}
%%%%%%%%%%%%%%%%%%%%%%%%
The preceding developments neither pretend nor intend to offer a directly applicable algorithm, but rather  formulate concepts that, in contrast to the ``classical'' approach 
realize accuracy quantification without (unrealistic) regularity assumptions. Nevertheless, the various routines are implementable. For the work horse $[\cB^{-1},f;\eta]$ this has been demonstrated in \cite{DGM2020}. A complete certifiable eigensolver based on the proposed paradigm, would of course require deviating significantly from existing software structures, hence requires in essence coding from scratch. More importantly, certifiability would require knowledge of (or at least good estimates for) the ``problem parameters'' 
    $\|\Cop\|, \theta, \uDelta$ which is generally not available. The question remains, are these quantities computationally accessible if the knowledge of specific optical parameters in the Boltzmann operator doesn't suffice.  Regarding $\|\Cop\|$ an answer should be affirmative. In fact, the Rayleigh \worr{quotients} for the self adjoint positive definite operator $\Cop^*\Cop$ would tend to $\|\Cop\|^2$
and the power method is in this case much easier to analyze and faster, due to the underlying orthonormal eigensystem. This also yields an upper bound for $\m$. 

Assessing the quantities $\theta, \uDelta$ instead seems to be much harder. Although we do not know of a concise algorithmic strategy for approximating these quantities with quantifiable certainty, there are heuristics that we expect to provide sufficient information to underpin an implementation. This is beyond the scope (and also spirit) of the present paper so that we are content with some brief indications and leave a more careful
elaboration to future work with a numerical focus.
First, in view of the relations \eqref{sandwich}, a good (lower bound for $\uDelta$ would
also yield an indication of $\theta$ (up to a factor depending on the angle between eigenvector and singular vector in $\U_\circ^\perp$). Approximating $\theta$ directly,
based on \eqref{newtheta}, could be attempted   through a power method where one has to
deal though with the restriction to $\U_\circ^\perp$ in each step. Although we don't know $\u$ yet, in view of 
$\U_\circ^\perp = \cE_{\Cop^*}(\sigma(\Cop^*)\setminus \{\m\})\U$, this restriction
can be realized through the Riesz projection $\cE_{\Cop^*}(\sigma(\Cop^*)\setminus \{\m\})$
which, in turn, can be approximated, in principle, by the quadrature approach, outlined at the end of
Section \ref{ssec:prac}. However, the accuracy of this step is uncertain, precisely because we don't know $\uDelta$. Nevertheless, a lower bound for $\m$ might lead to a reasonable contour that permits further computational exploration. Such a lower bound, in turn, could
be obtained from the Rayleigh quotients of the power iteration discussed above. Lacking
knowledge of $\uDelta$, it seems that all one can do is to apply $\Cop$ with very high accuracy and monitor the increase of the Rayleigh quotients, admittedly, a heuristic approach whose cost to success cannot be estimated.

These comments suffice perhaps to appraise remaining intrinsic obstructions reflecting the principal hardness of this type of spectral problems.

%%%%%%%%%%%%%%%%%%%%%%%%%%
%\section{Different inner products}
%%%%%%%%%%%%%%%%%%%%%%%%%%%%%%%%
\newcommand{\bw}{\mathbf{w}}
\newcommand{\bv}{\mathbf{v}}
\newcommand{\C}{\mathbb{C}}
\newcommand{\bc}{\mathbf{c}}
\renewcommand{\bC}{\mathbf{C}}
\renewcommand{\bD}{\mathbf{D}}
\renewcommand{\bJ}{\mathbf{J}}
\renewcommand{\lll}{\<\!\<}
\newcommand{\rr}{\>\!\>}

  %%%%%%%%%%%%%%%%
\section{Concluding Remarks}\label{sec:conclusion}
%%%%%%%%%%%%%%%
Our overarching goal is to develop and analyze a framework that
eventually leads to an accuracy controlled solution of the criticality problem for a model version of the  neutronic equation which nevertheless is general enough to exhibit its
intrinsic difficulties. On the one hand, this model is behind important 
applications related to future energy production. On the other hand, it is an example of an intrinsically difficult spectral problem and many of the proposed concepts
have some bearing beyond the specific model setting.
  Rendering a rigorous accuracy control, independent of any unrealistic a priory regularity requirements, is achieved by
a paradigm that deviates strongly from common practice. 
It requires
properly intertwining idealized iterations in a model compliant function space
with dynamically updated numerical approximations that are controlled by a posteriori quantities. At no stage will the numerical outcome result from a single a priori chosen discretization.
In both regards, stable variational formulations of the underlying PDE model play a key role. This paradigm, albeit in different formal disguises, has led in the past to a first complete complexity and convergence analysis of {\em wavelet methods} for PDEs \cite{CDD2}, as well as of finite element methods \cite{BDD}, and of low-rank and tensor methods for high-dimensional PDEs and Uncertainty Quantification \cite{BD}. From an analytical perspective, the problem types in these latter scenarios are in many ways more benign.
In the present setting it seems that we cannot completely ensure convergence success, at least not in conjunction with quantitative complexity bounds, \wo{see the comments in Section \ref{ssec:ram}}. Just being able 
to computationally approximate the principal eigenpair within a desired accuracy
is a challenge and the importance of the application may welcome even enormous computational effort if it leads to a certifiable  improved outcome quality. In this regard, we have tried
to complement the rigorous part of the analysis by possible computational techniques
that better and better cope with the most essential knowledge gap, namely a sufficiently
good estimate of the spectral gap.
 
\bmhead{Acknowledgements}

Olga Mula kindly thanks the Smart State Chair from South Carolina University for funding her stay at that university in 2019 when the present research was initiated.
We thank the reviewers for valuable suggestions that helped us to improve on the
presentation of the material.

\bmhead{Funding:} This research was supported by
the  NSF Grants DMS 2038080, DMS-2012469, DMS-2245097, by the SmartState and Williams-Hedberg Foundation, by the  SFB 1481, funded by the German Research Foundation. It was also funded by the Emergences Project Grant ``Models and Measures'' from the Paris City Council.

\begin{appendices}
\section{Proof of Lemma \ref{lem:contr}}
\label{appendixA}
%%%%%%%%%%%%%%
We observe first that $\int_{\Ddom\times \Vdom} u (v\cdot\nabla u) \dx\dv\ge 0$ holds for
every $u\in H_{0,-}(\Ddom\times\Vdom)$ where $H_{0,-}(\Ddom\times\Vdom)$ is defined in analogy to \eqref{SP-w}.
In fact, integrating by parts gives
$$
\int_{\Ddom\times \Vdom} u (v\cdot\nabla u) \dx\dv 
= 
- \int_{\Ddom\times \Vdom} u (v\cdot\nabla u) \dx\dv
+
\int_{\Gamma} u^2 v\cdot n.
$$
Since $u\in H_{0,-}(\Ddom\times\Vdom)$, this provides
\begin{equation}
\label{eq:monotonocity-advection}
\int_{\Ddom\times \Vdom} u (v\cdot\nabla u) \dx\dv
= \frac 1 2 \int_{\Gamma} u^2 v\cdot n
= \frac 1 2 \int_{\Gamma_+} u^2 v\cdot n \geq 0,
\end{equation}
by definition of $\Gamma_+$, see also \cite[28, p1105]{DLvol6}.

Next, given $g\in L_2
(\Ddom\times \Vdom)$, there exists a unique $u\in H_{0,-}(\Ddom\times \Vdom)$ (see Theorem \ref{thm:T-1}, \eqref{Ts-1}) such that
   $\cT u= \cK g$. 
Multiplying $\cT u= \cK g$ by $u$, integrating over $\Ddom\times\Vdom$, 
and using \eqref{eq:monotonocity-advection} and $\kappa\ge 0$, we obtain
\begin{align}
\int_{\Ddom\times\Vdom} \sigma u^2
&\leq
\int_{\Ddom\times\Vdom\times \Vdom} \kappa(x, v, v') u(x, v)g(x, v')\dx\dv\dv' \nonumber\\
&\leq \int_{\Ddom} \left(\int_{\Vdom\times \Vdom}\kappa(x, v, v') u(x, v)^2\dv\dv' \right)^{1/2} \left(\int_{\Vdom\times \Vdom}\kappa(x, v, v') g(x, v')^2\dv\dv' \right)^{1/2} \dx,% \quad \text{(by Cauchy-Schwarz)} \nonumber
\end{align}
where we have used Cauchy-Schwarz' inequality.

One deduces now from condition (H4) that  $\int_{\Vdom}
\kappa(x,v,v') dv'\le \rho\sigma (x,v)$ and  $\int_{\Vdom}
\kappa(x,v,v') dv\le \rho\sigma (x,v')$ for some constant $\rho<1$.
Using these bounds in the above inequality, yields
$$
\int_{\Ddom\times\Vdom}\sigma |u|^2 \dx\dv \le \rho^2 \int_{\Ddom\times\Vdom}\sigma |g|^2 \dx\dv
$$
which means $\|u\|_{\U^{(\sigma)}} \le \rho \|g\|_{\U^{(\sigma)}}$ and hence \eqref{T-1K}.
\hfill$\Box$

\section{Proof of Lemma \ref{lem:bal}}
\label{appendixB}
One shows first that $\u$ is strictly positive except on $\Gamma_{-}$.
In fact, by the preceding observations,
$$
\m\u=(\id- \cT^{-1}\cK)^{-1}\cT^{-1}\cF \u \ge \cT^{-1}\cF \u.
$$
Defining for $x,x'\in \Ddom$ the optical path  
$$
p(x,x',v):= \int_0^{|x-x'|}\sigma\Big(x+s\frac{x'-x}{|x'-x|},v\Big) ds, 
$$
and characteristic distance $d(x,v)$
from the boundary $\partial\Ddom$, i.e., $x-d(x,v)v\in \partial\Ddom$, the explicit representation of $\Top^{-1}$ yields
\be
\label{positive}
\m \u(x,v) \ge \int_0^{d(x,v)}\exp\big(-p(x,x-s d(x,v),v)\big)(\cF \u)(x-sv,v)ds.
\ee

One needs to show then that the right-hand side of \eqref{positive} is strictly positive
except for $(x,\v)\in \Gamma_-$. This is not quite obvious yet because a non-trivial $\u$
could have the property that for $x$ in a small neighborhood of some $x_0$, $\u(x-s\v,\v')=0$ holds for all $s$ and $\v'\in\Vdom$. Therefore, $\Fop$ being linear and positive,
an additional application of $\Fop$ to both sides of \eqref{positive} yields
$\m\Fop\u\ge \Fop\Top^{-1}\Fop\u$. Thus, as soon as one shows $\Fop\Top^{-1}\Fop\u$
is strictly positive on $\overline{(\Ddom\times \Vdom)}\setminus \Gamma_-$, 
corresponding strict positivity follows from \eqref{positive}. To that end,
elementary manipulations show that
\begin{align*}
(\Fop\Top^{-1}\Fop\u)(x,\v)&= \int_{\Vdom}\varphi(x,\v',\v)\int_0^{d(x,\v')}
e^{-p(x,x-s\v',\v')}\worr{\int_{\Vdom}\varphi(x-s\v',\v'',\v')d\v'\u(x-s\v',\v'')\, ds dv' d\v''}.
\end{align*}
\worr{In \cite[Theorem 7, p. 1154]{DLvol6} it is proved that the right-hand side expression is strictly positive in $\Ddom\times \Vdom$ except on $\Gamma_-$. We briefly explain the
essence of the argument.
By (H1) and (H2), $\Ddom$ is convex,   hence star-shaped for every $x\in \Ddom$, and
$\Vdom$ contains a set that is homeomorphic to the sphere. Thus,
the arguments $(x-s\v',\v'')$ of $\u$ traverse all of $\Ddom\times \Vdom$. Moreover by (H5), $ \u(x-s\v',\v'')$ is multiplied by strictly positive quantities. Thus,
for the right-hand side to vanish on a set of strictly positive measure   in $(\Ddom\times\Vdom)\setminus \Gamma_-$, $\u$ would have
to vanish everywhere in $\Ddom\times \Vdom$.}  This means, in particular, that 
$\Fop\Top^{-1}\Fop$ can be represented as an  integral operator with a strictly positive kernel. We refer to \cite[Theorem 7, p. 1154]{DLvol6} for detailed arguments to that end.

As a next step it is shown that $\u$ is the only positive eigenstate. 
 This is based on considering the dual problem $\Cop^* u=\mu u$ where $
\Cop^*= \Fop^*\Bop^{-*}$ is the adjoint of $\Cop$. Note that $\Cop$ and $\Cop^*$ share the same spectral radius and $\m$ is also eigenvalue of $\Cop^*$ associated with a (different) positive eigenstate, denoted by $\u_*$ vanishing only on $\Gamma_+$. Suppose that $(\bar\mu,\bar u)$ is any other
eigenpair of $\Cop$ with $\bar\mu\neq \m$ so that 
$$
\m \<\bar u,\u_*\> = \<\bar u,\Cop^*\u_*\>=\<\Cop \bar u,\u_*\> =\bar\mu\<\bar u,\u_*\>  
$$
hence
$$
(\bar\mu^{-1}-\m^{-1})\<\bar u,\u_*\>=0,
$$
which, since $\bar\mu\neq \m$, implies $\<\bar u,\u_*\>=0$. This is impossible when 
$\bar u$ is non-trivial and non-negative, a contradiction. 

It remains to prove that 
$\m$ is simple. 
%%%%%%%%%%%%
\newcommand{\w}{w^\circ}
%%%%%%%%%%%%%%%%
To that end, consider the operator $(\Fop\Bop^{-1})^* = \Bop^{-*}\Fop^*$ which has the same
spectral radius as $\Fop\Bop^{-1}$ and as $\Bop^{-1}\Fop$, where the last fact follows from
the general property that for bounded linear operators $A,B$ the products $AB$ and $BA$
have the same spectral radius. This can be deduced by using that the spectral radius 
of an operator $A$ equals $\lim_{n\to\infty}\|A^n\|^{1/n}$.
Moreover, by the same arguments as before, $\Bop^{-*}\Fop^*$ has a   positive 
eigenstate associated with the spectral radius $\m$ which we denote by $\hat\u_*$.

Suppose now there exists another linearly independent eigenstate $\w$ of $\Cop$, associated with 
$\m$. \worr{Then we can orthonormalize $\w$ in the eigenspace spanned by $\u$ and $\w$
yielding an eigenvector $ w$ in this subspace that satisfies $\<\u, w\>=0$. Since $\u$ is
strictly positive (except on $\Gamma_-$) $ w$ must change sign.}
Since
$|\Cop^2 w|= |\m^2 w|=\m^2|w|$, one observes that
\be
\label{E1}
E_1= \< |\Cop^2 w|, \Fop^* \hat\u_*\> = \m^2\<|w|,\Fop^* \hat\u_*\> >0.
\ee
On the other hand,
\be 
\label{E2}
E_2=\<\Cop^2 |w|,\Fop^*\hat\u_*\>= \<|w|, (\Cop^*)^2 \Fop^*\hat\u_*\> = 
\<|w|, \Fop^*\Bop^{-*}\Fop^*\Bop^{-*}\Fop^* \hat\u_*\>.%= \<|v|. \Fop^* (
\ee
Moreover, note that $\Fop^*\hat\u_*$ is a strictly positive eigenstate of $\Fop^*\Bop^{-*}$
because $\Fop^*\Bop^{-*}\Fop^*\hat\u_*=\Fop^* (\Fop\Bop^{-1})^*\hat\u_*= \Fop^*\m \hat\u_*$. Thus, 
$$
E_2= \<|w|,\m^2\Fop^*\hat\u_*\> = E_1 >0.
$$
Next one shows $E_1<E_2$ which is the desired contradiction and confirms simplicity of $\m$. To that end, 
since $w$ changes sign, $|w|\pm v$ are both non-trivial non-negative functions.
Note also that 
$$
\Bop^{-1}= (\Top(\id- \Top^{-1}\Kop))^{-1} = (\id -\Top^{-1}\Kop)^{-1}\Top^{-1}
= \sum_{j=0}^\infty (\Top^{-1}\Kop)^j\Top^{-1}\ge \Top^{-1},
$$
where we have used Lemma \ref{lem:contr}. Now recall that $\Fop\Top^{-1}\Fop$ is an integral
operator with strictly positive kernel. Then 
$$
\Cop^2(|w|\pm w)= (\Top-\Kop)^{-1}\Fop(\Top-\Kop)^{-1}\Fop(|w|\pm w)\ge  (\Top-\Kop)^{-1}
\Fop \Top^{-1}\Fop (|w|\pm w) >0 \quad \mbox{in~ \worr{$\overline{\Ddom\times\Vdom}\setminus \Gamma_-$}}.
$$
This, in turn, implies
$$
\Cop^2|w| > |\Cop^2w|,
$$
and hence $E_2>E_1$, the announced contradiction.\hfill$\Box$

%%%%%%%%%%%%%%%%%%%%%%%%%%%%%%%
\section{Proof of Remark \ref{rem:relation}}
\label{appendixC}
Regarding the interrelation between the inf-sup constant $\theta$ and the spectral gap 
 $\uDelta$, stated in Remark \ref{rem:relation}, we first identify $\U_\circ^\perp$ in terms of Riesz projections. To that end,
 recall the two direct sum decompositions
that
\be
\label{UV}
\U=\U_\circ \oplus_\perp \U_\circ^\perp = \U_\circ\oplus \,\Vsp, \quad \Vsp:= \cE_{\Cop}( \sigma(\Cop)\setminus\{\m\})\U .
\ee 
 $\U_\circ^\perp$ and $ \Vsp $ both have co-dimension one but are generally different 
 (they are equal when $\Cop$ is a normal operator). Observe next that
 \be
\label{ident}
\Vsp^* = \U_\circ^\perp,\quad  \Vsp= \<\u_*\>^\perp.
\ee
To see this,
we apply analogous properties of Riesz projections  to the adjoint $\Cop^*$, where now $\Vsp^*$ is the 
$\Cop^*$-invariant subspace associated with $\sigma(\Cop^*)\setminus \{u_*^\circ\}$,
$\u_*$ being the simple eigenvector of $\Cop^*$ associated with $\m$.
Then, for any $v\in \Vsp$ and $\sigma_{>1}:= \sigma(\Cop)\setminus\{\m\}$, one has  
$$
\<\u_*, v\>= \<\cE_{\Cop^*}(\m)\u_*,
\cE_{\Cop}(\sigma_{>1})v\>= \<\u_*, \cE_{\Cop^*}^*(\m)\cE_{\Cop}(\sigma_{>1})v\>
= \<\u_*,\cE_{\Cop}(\m)\cE_{\Cop}(\sigma_{>1})v\> =0,
$$
where we have used \eqref{anihi} in the last step. The same reasoning applies to
$\U_\circ^\perp=\<\u\>^\perp$, so that
\be
\label{ortho*}
\u_*\perp \Vsp,\quad  \u \perp \Vsp^* ,
\ee
which is \eqref{ident}.
 
The verification of the inequalities \eqref{sandwich} will follow from
judicious substitutes of the maximizer/minimizer in the inf-sup condition from
these complement spaces.
To that end, recall the meaning of  $\rc\in \U_\circ^\perp = \Vsp^*$, $u^*_\La\in \Vsp^*=\U_\circ^\perp$, $\lambda_\La$ from  Remark \ref{rem:relation}.
In particular, $\bar\lambda_\La\Cop^*u^*_\La = u^*_\La$ so that $u^*_\La$ is also
an eigenstate of $(\Mop_\l^\circ)^*$ with eigenvalue $1- \l/\bar\lambda_\La$.
Then
one deduces from \eqref{thetaint} that
$$
\theta= \sup_{\substack{v\in \U_\circ^\perp\\ \|v\|=1}}\< \Mop_\l \rc,v\>
\ge %\sup_{\substack{v\in \U_\circ^\perp}\\ \|v\|=1}}
|\<  \rc,\Mop_\l^* u^*_{\La}\>| =|1-\l/\bar\lambda_\La||\<  \rc, u^*_{\La}\>|
\ge (1-\l/|\bar\lambda_{\La}|)\<  \rc, u^*_{\La}\>|,
$$
which is the lower estimate in \eqref{sandwich}.

To see the upper estimate in \eqref{sandwich}, denote by $u_{\La}$
an eigenvector of $\Cop$ associated with the largest in modulus eigenvalue $\mu_2\in \sigma_{>1}$.
Since $u_{2}$ belongs to $\Vsp$, which generally differs from $\U_\circ^\perp=\Vsp^*$,
we let $z:= P_{\U_\circ^\perp}u_{\La}\in \U_\circ^\perp$, i.e., $u_{\La}=
z+ P_{\U_\circ}u_{\La}$. Since $\|z\| = \sqrt{\|u_{\La}\|^2- \|P_{\U_\circ}u_{\La}\|^2}
= \sqrt{1- \|P_{\U_\circ}u_{\La}\|^2} =:s  >0$ and since
$$
\Mop^\circ_\l z= \Mop^\circ_\l\big( u_{\La}- P_{\U_\circ}u_{\La}\big)= \Mop^\circ_\l 
u_{\La}= (1- \l \mu_{\La})u_{\La}
$$
we conclude that  
$$
\theta\le \sup_{\substack{v\in \U_\circ^\perp\\ \|v\|=1}}\frac{|\<\Mop_\l z,v\>|}{ s }
= \sup_{\substack{v\in \U_\circ^\perp\\ \|v\|=1}}|1- \l\mu_{2}|s^{-1}|\<u_{\La},v\>|
=s^{-1} |1- \l\mu_{2}|\|z\|^2 =s \oDelta,
$$
which is the upper bound in  \eqref{sandwich}.

When $\Cop$ is a normal operator $\U_\circ^\perp = {\U_\circ^*}^\perp$ so that ${\rm dist}\big(\U_\circ^\perp,
{\U_\circ^*}^\perp\big)=0$. Moreover, the smallest singular value of $\Mop^\circ_\l$
agrees then with its smallest eigenvalue which, in turn is one minus the largest
eigenvalue of $\l\Cop u= \mu u$ which we know is $\l \lambda_\La^{-1}$ with eigenvector 
$u_\Lambda$.  Hence, in this case $\theta = \oDelta = \uDelta$ which is \eqref{normal}.

 The non-trivial gap in \eqref{sandwich} can thus be
viewed as quantifying deviation from normality.
\hfill $\Box$
\end{appendices}


\bibliography{literature}% common bib file

\end{document}